\documentclass[11pt]{article}
\usepackage[utf8]{inputenc}
\usepackage[a4paper, width=150mm,top=25mm,bottom=25mm]{geometry}
\usepackage{mathtools}
\usepackage{pgf, tikz}
\usetikzlibrary{arrows, automata}
\usepackage{framed}

\usepackage{suffix}
\usepackage{enumerate}
\usepackage{enumitem}
\usepackage[affil-it]{authblk}

\usepackage[breaklinks=true]{hyperref}
\hypersetup{
colorlinks,
linkcolor=blue,
urlcolor=blue,
citecolor=black
}

\usepackage{accents}

\usepackage{listings}
\usepackage{soul}

\makeatletter
\newcommand*{\rom}[1]{\expandafter\@slowromancap\romannumeral #1@}
\usepackage{varwidth}
\usepackage{parskip}
\usepackage{titlesec}
\titleformat{\section}
  {\centering}{\thesection}{1em}{\MakeUppercase}

\usepackage{float}

\usepackage{pgf, tikz}
\usetikzlibrary{arrows, automata}

\usepackage{amsmath,amsthm,amssymb,nccmath}

\DeclareRobustCommand{\stirling}{\genfrac\{\}{0pt}{}}

\allowdisplaybreaks
\usepackage{bbm}
\usepackage[makeroom]{cancel}

\usepackage[maxbibnames=99,style=numeric,backend=biber,sorting=nyt]{biblatex}
\AtNextBibliography{\small}
\DeclareNameAlias{default}{family-given}
\usepackage{xpatch}
\xpatchbibmacro{cite}{\printnames{labelname}}%
{\ifciteseen{\printnames{labelname}}{\printnames[][-\value{listtotal}]{labelname}}}
{}{}

\xpatchbibmacro{textcite}{\printnames{labelname}}%
{\ifciteseen{\printnames{labelname}}{\printnames[][-\value{listtotal}]{labelname}}}
{}{}

\newtheoremstyle{plain}
  {.3\baselineskip\@plus.2\baselineskip\@minus.2\baselineskip}
  {.3\baselineskip\@plus.2\baselineskip\@minus.2\baselineskip}
  {\itshape}  
  {0pt}       
  {\bfseries} 
  {.}         
  {5pt plus 1pt minus 1pt} 
  {}  
\theoremstyle{plain}
\newtheorem{thm}{Theorem}[section]
\newtheorem{lemma}[thm]{Lemma}
\newtheorem{prop}[thm]{Proposition}
\newtheorem{cor}[thm]{Corollary}

\newcommand{\Prob}{\mathbbm{P}}
\newcommand{\E}{\mathbb{E}}

\newcommand\numberthis{\addtocounter{equation}{1}\tag{\theequation}}

\DeclarePairedDelimiter\ceil{\lceil}{\rceil}
\DeclarePairedDelimiter\floor{\lfloor}{\rfloor}

\newtheoremstyle{definition}
  {.3\baselineskip\@plus.2\baselineskip\@minus.2\baselineskip}
  {.3\baselineskip\@plus.2\baselineskip\@minus.2\baselineskip}
  {}  
  {0pt}       
  {\bfseries} 
  {.}         
  {5pt plus 1pt minus 1pt} 
  {}  
\theoremstyle{definition}
\newtheorem{defn}[thm]{Definition} 
\newtheorem{remark}[thm]{Remark}

\makeatletter
\patchcmd{\@maketitle}{\LARGE \@title}{\fontsize{16}{19.2}\selectfont\@title}{}{}
\makeatother

\title{Weak local limit of preferential attachment random trees with additive fitness\vspace{-1ex}}
\author[]{Tiffany Y. Y. Lo\thanks{School of Mathematics and Statistics, University of Melbourne.\\ Email address: tiffanyloyinyuan@gmail.com }\vspace{-1ex}}
\affil[]{University of Melbourne\vspace{-1ex}}

\begin{document}

\maketitle

\begin{abstract}
We consider linear preferential attachment random trees with additive fitness, where fitness is defined as the random initial vertex attractiveness. We show that when the fitness distribution has positive bounded support, the weak local limit of this family can be constructed using a sequence of mixed Poisson point processes. We also provide a rate of convergence of the total variation distance between the $r$-neighbourhood of the uniformly chosen vertex in the preferential attachment tree and that of the root vertex of its weak local limit. We apply the theorem to obtain the limiting degree distributions of the uniformly chosen vertex and its ancestors, that is, the vertices that are on the path between the uniformly chosen vertex and the initial vertex. Rates of convergence in the total variation distance are established for these results. 
\end{abstract}

\section{Introduction}\label{sintro}
There has been considerable interest in studying the preferential attachment random graphs since they were used by \textcite{barabasi1999emergence} to explain the observed power law degree distribution in some real networks such as the World Wide Web. The primary feature of the stochastic mechanism consists of adding vertices sequentially over time with some number of edges attached to them, and then connecting these edges to the existing graph in such a way that vertices with higher degrees are more likely to receive them. A general overview of preferential attachment random graphs can be found in the books \cite{van2009random} and \cite{vanderhofstad2}. 

In the basic models, vertices are born with the same `weight' as their initial vertex attractiveness. To relax this assumption, \textcite{ergun2002growing} introduced a class of preferential attachment graphs with additive fitness (referred to as Model A in \cite{ergun2002growing}), where \textit{fitness} is defined as the \textit{random} initial attractiveness. This family is the subject of recent works such as \textcite{iyer2020degree}, \textcite{lodewijks2020phase} and \textcite{delphin2019}, whose results we discuss in Section \ref{recentw}. In this paper, we study the weak local limit of this family, and provide a rate of convergence of the total variation distance. Our weak local limit theorem extends the result of \textcite{Berger12asymptoticbehavior}, which considered preferential attachment graphs whose vertices are born with the same initial attractiveness. As a by-product of our analysis, we obtain limiting degree distributions of the uniformly chosen vertex and its ancestors, and establish rates of convergence for these results. Another objective of this article is to present the arguments of \cite{Berger12asymptoticbehavior} in more detail, which is the main reason why we consider the preferential attachment tree instead of the case where multiple edges are possible.

Before defining our model, note that we view the edges as directed, where a newly-born vertex always sends a single outgoing edge to an existing vertex in the graph. We define the \textit{weight} of a vertex as its in-degree plus the fitness, and each time the vertex receives an edge from another vertex, its weight increases by one. The random rules behind the evolution of this model are described below.

\textbf{The $(\mathbf{x},n)$-sequential model.} Given a positive integer $n$ and the sequence $\mathbf{x}:=(x_i,i\geq 1)$, where $x_1>-1$ and $x_i>0$, $i\geq 2$, we construct the sequence of random trees $(G_{i}, 1\leq i\leq n)$ as follows. The seed graph $G_{1}$ is a vertex labelled as 1 with \textit{initial attractiveness} $x_1$. Given $G_{m-1}$ and $\mathbf{x}$, $G_m$ is constructed by attaching one edge between vertex $m$ and vertex $k\in \{1,...,m-1\}$, and the edge is directed towards vertex $k$ with probability
\begin{equation*}
    \frac{W_{k,m-1}+x_k}{m-2+\sum^{m-1}_{j=1}x_j}\quad\text{for $1\leq k\leq m-1$,}
\end{equation*}
where $W_{j,l}$ denotes the in-degree of vertex $j$ in $G_l$, and $W_{j,l}=0$ whenever $j\geq l$. The $m$th attachment step is completed by assigning vertex~$m$ with the initial attractiveness $x_m$. We call the resulting graph $G_n$ an \textit{$(\mathbf{x},n)$-sequential model}, and its law is denoted by Seq$(\mathbf{x})_n$.

Let $\mathbf{X}:=(X_i,i\geq 1)$ be a fitness sequence such that $x_1:=X_1$ is deterministic, and $(X_i,i\geq 2)$ are i.i.d.\ positive variables with distribution $\pi$. We shall assume that $\mathbf{x}$ is a realisation of $\mathbf{X}$. Thus, we obtain the distribution of $G_n$ by mixing the conditional distribution of $G_n$ with the distribution of $(X_i,i\geq 2)$, and we denote this unconditional law by PA$(\pi)_n$. 

\subsection{Weak local limit}\label{localweak}
Prior to stating the main result, we introduce the concept of the weak local limit, and we refer to \textcite{benjamini2001}, \textcite{aldoussteele} and \textcite[Chapter 2]{vanderhofstad2} for more detail. Informally, we explore some random graph $G_n$ from vertex $o_n$, chosen uniformly at random from $G_n$, and study the distributional limit of the neighbourhoods of radius $r$ rooted at $o_n$ for each $r<\infty$. To precisely define the weak local limit of random trees, we need a few definitions.

A rooted graph is a pair $(G,o)$, where $G=(V(G), E(G))$ is a graph with vertex set $V(G)$ and edge set $E(G)$, and $o\in V(G)$ is the designated root in $G$. Next, let $r$ be a finite, positive integer. For $(G,o)$, denote by $B_r(G,o)$ the neighbourhood of radius $r$ around $o$. More formally, $B_r(G,o)=(V(B_r(G,o))), E(B_r(G,o)))$, where in the case of trees,
\begin{align*}
    &V(B_r(G,o))=\{u\in V(G): \text{the distance between $o$ and $u$ is less than or equal to $r$ edges} \},\\
    &E(B_r(G,o))=\{\{u,v\}\in E(G): u,v\in V(B_r(G,o))\}.
\end{align*}
We refer to $B_r(G,o)$ as the \textit{$r$-neighbourhood} of vertex $o$, or simply as the \textit{local neighbourhood} of $o$ when reference to $r$ is not required. The last ingredient is the following.

\begin{defn}[Isomorphism of rooted graphs]
We say that two rooted graphs $(G,o)$ and $(H,o')$ are isomorphic, if there is a bijection $\psi:V(G)\to V(H)$ such that $\psi(o)=o'$ and $\{u,v\}\in E(G)$ if and only if $\{\psi(u), \psi(v)\}\in E(H)$. If $(G,o)$ and $(H,o')$ are isomorphic, then we write $(G,o) \cong (H,o')$.
\end{defn}

Following \textcite{benjamini2001}, we define the weak local limit of a sequence of finite, random tree $(G_{n},n\geq 1)$ as follows.

\begin{defn}[Local weak limit]\label{bs}
Let $(G_{n},n\geq 1)$ be a sequence of finite random trees, and $(G_{n}, o_n)$ be the rooted tree obtained by choosing $o_n\in V(G_{n})$ uniformly at random. We say that $(G,o)$ is the weak local limit of $(G_{n},o_n)$, if for all finite rooted graphs $(H,v)$ and all finite $r$, 
\begin{equation*}
    \Prob((B_r(G_{n},o_n),o_n)\cong (H,v))\overset{n\to\infty}{\longrightarrow}\Prob((B_r(G,o),o)\cong (H,v)).
\end{equation*}
\end{defn}

\subsection{Statement of result}\label{UHN}
\normalsize
To state our main theorem, we first construct the weak local limit of the graph with law PA($\pi)_n$. This is a rooted random tree that generalises the P\'olya-point tree introduced in \textcite{Berger12asymptoticbehavior}, so we refer to it as a $\pi$-P\'olya point tree, with $\pi$ being the fitness distribution of the preferential attachment tree. Moreover, from now on we assume that $\pi$ has a finite mean: $\mu:=\E X_2<\infty$, and define
\begin{equation*}
  \chi:=\frac{\mu}{\mu+1}.
\end{equation*}

Before defining the $\pi$-P\'olya point tree, we explain the variables and notations appearing in its construction. Vertex $0$ is the root of of the $\pi$-P\'olya point tree, and we denote the random tree by $(\mathcal{T},0)$. Using the Ulam-Harris labelling of trees, let $(0,i)$, $i=1,2,...$ be the vertices that are connected to vertex $0$. Recursively, if $\bar v:=(0,v_1,...,v_r)$, where $v_i$ are positive integers and $r$ is the distance from vertex $\bar v$ to vertex $0$, we label the vertices connected to $\bar v$ that are at distance $r+1$ from the root as $(\bar v,j)$, $j\in \mathbbm{N}:=\{1,2,3,...\}$. In other words, $(\bar v,i)\in \partial B_{r+1}:= V(B_{r+1}(\mathcal{T},0))\setminus V(B_{r}(\mathcal{T},0))$, as illustrated in Figure~\ref{graphlabelling} below.

\begin{figure}
    \centering
    \begin{tikzpicture}
     \tikzstyle{every state}=[
            fill = black,
            shape = circle,
            scale = 0.2
        ]
    \node[state, label={0}] (0) at (0,0) {};
    \node[state, label={{(0,1)}}] (01) at (1.05,1.05) {}; 
    \node[state, label={{(0,2)}}] (02) at (-1.5,0) {};
    \node[state, label={{(0,3)}}] (03) at (0,-1.5) {};
    \node[state, label=right:{{(0,4)}}] (04) at (1.05,-1.05) {};
    \node[state, label=right:{{(0,1,1)}}] (011) at (2.8,1.05) {};
    \node[state, label=right:{{(0,1,2)}}] (012) at (3,0) {};
    \node[state, label=left:{{(0,2,1)}}] (021) at (-2.6,1.5) {};
    \node[state, label=left:{{(0,2,2)}}] (022) at (-3,0) {};
    \node[state, label=left:{{(0,2,3)}}] (023) at (-2.6,-1.5) {};
    \node[state, label=below:{{(0,3,1)}}] (031) at (-1.5,-2.6) {};
    \node[state, label=below:{{(0,3,2)}}] (032) at (0,-3) {};
    \node[state, label=below:{{(0,3,3)}}] (033) at (1.5,-2.6) {};
    \draw[dashed] (0,0) circle [radius=1.5cm] node[above] at (0,1.5) {$\partial B_1$};
    \draw[dashed] (0,0) circle [radius=3cm] node[above] at (0,3) {$\partial B_2$};
    \draw[-] (0) -- (01);
    \draw[-] (0) -- (02);
    \draw[-] (0) -- (03);
    \draw[-] (0) -- (04);
    \draw[-] (01) -- (011);
    \draw[-] (01) -- (012);
    \draw[-] (02) -- (021);
    \draw[-] (02) -- (022);
    \draw[-] (02) -- (023);
    \draw[-] (03) -- (031);
    \draw[-] (03) -- (032);
    \draw[-] (03) -- (033);
\end{tikzpicture}
    \caption{\small An example of $B_2(\mathcal{T},0)$. Vertices are represented by black dots, and the labels are $0,(0,1),...,(0,3,3)$. The dashed circles are $\partial B_1$ and $\partial B_2$. Vertex $\bar u$ is in $\partial B_s$ if the distance between $\bar u$ and $0$ is exactly $s$ edges. Vertex 0 is the root, (0,1) and (0,1,1) are type L vertices, and the rest belong to type R. Note that $\tau_0=3$, $\tau_{0,1}=1$, $\tau_{0,2}=3$, $\tau_{0,3}=3$ and $\tau_{0,4}=0$.}
    \label{graphlabelling}
\end{figure}
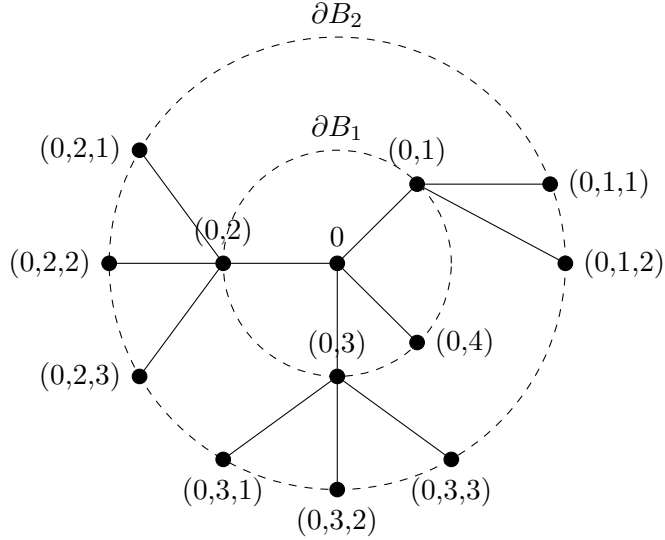

Furthermore, each vertex $\bar v\in V((\mathcal{T},0))$ has a fitness $X_{\bar v}$ and a random \textit{age} $a_{\bar v}$, where $0< a_{\bar v}\leq 1$. We write $a_{\bar v,i}:=a_{(\bar v,i)}$ for convenience. Apart from the root vertex 0, there are two types of vertices, namely, type L (for left) and R (for right). Vertex $\bar v$ belongs to type~L if $a_{\bar v,i}<a_{\bar v}$ for some $i\geq 1$; and $\bar v$ belongs to type R if $a_{\bar v,i}\geq a_{\bar v}$ for all $i\geq 1$. There is exactly one type~L vertex in $\partial B_r$ for all $r\geq 1$, and $(0,1,1,...,1)$ are designated as these vertices. Regardless of the vertex type of $\bar v \in \partial B_r$, there is a random number of type R vertices in $\partial B_{r+1}$ that are connected to $\bar v$, and we denote this number by~$\tau_{\bar v}$. We label the type R vertices in increasing order of their ages, so that if $\bar v$ is the root or belongs to type L, then $a_{\bar v,1}\leq a_{\bar v}\leq a_{\bar v,2}\leq ...\leq a_{\bar v,1+\tau_{\bar v}}$; and if $\bar v$ belongs to type~R, then $a_{\bar v}\leq a_{\bar v,1}\leq ...\leq a_{\bar v,\tau_{\bar v}}$, as shown in Figure \ref{graphlabelling} and \ref{pipolyaex}. 

The distributions of $X_{\bar v}$, $\tau_{\bar v}$ and $a_{\bar v}$ will be made precise when we define $(\mathcal{T},0)$ below, but to understand these variables and the types, consider the breadth-first exploration of the $r$-neighbourhood of a uniform vertex $k_0\in \mathbbm{N}$ in $G_n\sim\mathrm{PA}(\pi)_n$: $B_r(G_n,k_0)$. Let $\partial \mathcal{B}_{r}:= V(B_{r}(G_n,k_0))\setminus V(B_{r-1}(G_n,k_0))$ and $\partial \mathcal{B}_{0}:=k_0$. We define an Ulam-Harris labelling for the vertices in $B_r(G_n,k_0)$ to better relate $B_r(G_n,k_0)$ and $B_r(\mathcal{T},0)$. The Ulam-Harris labels now appear as subscripts. Starting from $k_0$, the labels are generated recursively as follows. If $\bar v=(0,v_1,...,v_r)$, where $v_i\in \mathbbm{N}$ and $r$ is the distance from vertex $k_{\bar v}\in \mathbbm{N}$ to vertex $k_0$, we label the vertices in $\partial \mathcal{B}_{r+1}$ that are connected to $k_{\bar v}$ as $k_{\bar v,j}$, and such that $k_{\bar v,1}<k_{\bar v,2}<k_{\bar v,3}<...$. 

Next, we explain the significance of the vertex types. Apart from vertex $k_0$, a vertex in $\partial \mathcal{B}_k$ is either discovered through (\rom{1}) the incoming edge that it has received from a vertex in $\partial \mathcal{B}_{k-1}$; or (\rom{2}) the outgoing edge it has sent to a vertex in $\partial \mathcal{B}_{k-1}$. The probability that $B_r(G_n,k_0)$ contains vertex 1 is $o(1)$ as $n\to\infty$; and outside this event, a moment's thought shows that there is exactly one vertex of type (\rom{1}) at each $\partial \mathcal{B}_j$, $1\leq j\leq r$, which we designate as $k_{0,1,...,1}$. It follows that for all~$1\leq j \leq r$, $k_{0,1,1,...,1} \in \partial \mathcal{B}_{j}$ is the vertex that receives the incoming edge from $k_{0,1,1,...,1} \in \partial \mathcal{B}_{j-1}$ ($k_0$ if $j=1$). Vertices of type (\rom{1}) and (\rom{2}) correspond to vertices of type L and R in $(\mathcal{T},0)$, so from now on we refer to (\rom{1}) and (\rom{2}) simply as L and R.

Finally, we elucidate the roles of $X_{\bar v}$, $a_{\bar v}$ and $\tau_{\bar v}$. In particular, we shall couple $(G_n,k_0)$ and $(\mathcal{T},0)$ such that $(B_r(\mathcal{T},0),0)\cong (B_r(G_n,k_0),k_0)$ with high probability, where we match vertex $\bar v\in B_r(\mathcal{T},0)$ and vertex $k_{\bar v}\in B_r(G_n,k_0)$. The fitness $X_{\bar v}$ corresponds to the fitness of vertex $k_{\bar v}$. Furthermore, vertex $k_{\bar v}\in \partial \mathcal{B}_{r-1}$ has a random number of type R neighbours in $\partial \mathcal{B}_r$, which we denote by $\theta_{\bar v}$. In the graph coupling, we couple $\theta_{\bar v}$ and $\tau_{\bar v}$ such that $\theta_{\bar v}=\tau_{\bar v}$ with high probability. The ages in $B_r(\mathcal{T},0)$ are continuous analogs of the vertex labels $k_{\bar v}$ in $B_r(G_n,k_0)$; and with high probability, $a_{\bar v}$ can be closely coupled with the scaled vertex label $(k_{\bar v}/n)^\chi$. We stress that the random variables $k_{\bar v}$ and $\theta_{\bar v}$ depend on $n$, but we drop $n$ from the notation in favour of simpler expressions.

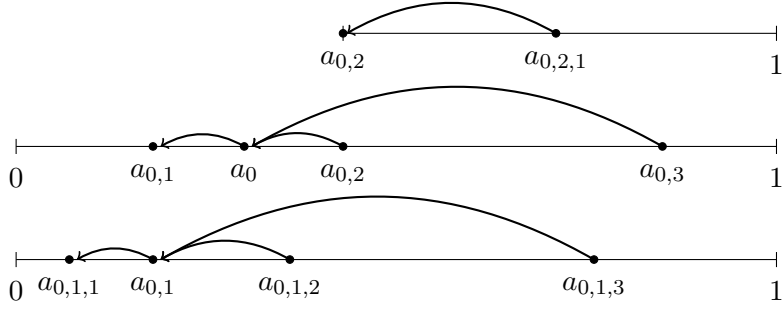
\begin{figure}
    \centering
    \begin{tikzpicture}
    \draw (0,0) -- (10,0);
    \draw (0,0.1) -- (0,-0.1)  node[below=0.5mm] {$0$};  
    \draw (10,0.1) -- (10,-0.1)  node[below=0.5mm] {$1$};
    \draw [fill] (1.8,0) circle [radius=0.05] node[below=1mm] {$a_{0,1}$};
    \draw [fill] (3,0) circle [radius=0.05] node[below=1mm] {$a_{0}$};
    \draw [fill] (4.3,0) circle [radius=0.05] node[below=1mm] {$a_{0,2}$};
    \draw [fill] (8.5,0) circle [radius=0.05] node[below=1mm] {$a_{0,3}$};
     \draw [->, thick] (3,0) to[bend right] (1.9,0);
     \draw [->, thick] (4.3,0) to[bend right] (3.1,0);
     \draw [->, thick] (8.5,0) to[bend right] (3.1,0);
    \draw (4.3,1.5) -- (10,1.5);
    \draw (4.3,1.6) -- (4.3,1.4) node[below=0.5mm]{} ;
    \draw (10,1.6) -- (10,1.4) node[below=0.5mm] {$1$} ;
    \draw [fill] (7.1,1.5) circle [radius=0.05] node[below=1mm] {$a_{0,2,1}$};
    \draw [fill] (4.3,1.5) circle [radius=0.05] node[below=1mm] {$a_{0,2}$};
     \draw [->, thick] (7.1,1.5) to[bend right] (4.35,1.5);
    \draw (0,-1.5) -- (10,-1.5);
    \draw (0,-1.4) -- (0,-1.6) node[below=0.5mm] {$0$};
    \draw (10,-1.4) -- (10,-1.6) node[below=0.5mm] {$1$};
    \draw [fill] (0.7,-1.5) circle [radius=0.05] node[below=1mm] {$a_{0,1,1}$};
    \draw [fill] (1.8,-1.5) circle [radius=0.05] node[below=1mm] {$a_{0,1}$};
    \draw [fill] (3.6,-1.5) circle [radius=0.05] node[below=1mm] {$a_{0,1,2}$};
    \draw [fill] (7.6,-1.5) circle [radius=0.05] node[below=1mm] {$a_{0,1,3}$};
    \draw [->, thick] (1.8,-1.5) to[bend right] (0.8,-1.5);
    \draw [->, thick] (3.6,-1.5) to[bend right] (1.9,-1.5);
    \draw [->, thick] (7.6,-1.5) to[bend right] (1.9,-1.5);
\end{tikzpicture}
    \caption{\small A pictorial example of the ages in the $\pi$-P\'olya point tree. Vertices $0$, $(0,1)$ and $(0,2)$ are respectively, the root, of type L and of type R. Here we draw an edge between the two ages if the corresponding vertices are connected, and we think of the edges as directed such that if $\bar v$ is the root or a type~L vertex, then it sends an outgoing edge to $(\bar v,1)$, and receives incoming edges from $({\bar v,i})$, $2\leq i\leq 1+\tau_{\bar v}$. If $\bar v$ belongs to type R, then it receives incoming edges from $(\bar v,i)$, $1\leq i\leq \tau_{\bar v}$. }
    \label{pipolyaex}
\end{figure} 

\begin{defn}[$\pi$-P\'olya point tree]\label{pipolyapointgraph}
A $\pi$-P\'olya point tree $(\mathcal{T},0)$ is defined recursively as follows. The root $0$ has an age $a_0=U_0^{\chi}$, where $U_0\sim \mathrm{U}[0,1]$. Assuming that $\bar v=(0,v_1,...,v_{r})\in \partial B_r$ and $a_{\bar v}$ have been generated, we define $(\bar v, j)\in \partial B_{r+1}$ for $j=1,2,...$ as follows. Independently of all random variables generated before, let $X_{\bar v}\sim \pi$ and
\begin{align*}
    Z_{\bar v}\sim\begin{cases}
    \mathrm{Gamma}(X_{\bar v}, 1),\quad\text{if $\bar v$ is the root or of type R.}\\
    \mathrm{Gamma}(X_{\bar v}+1, 1),\quad\text{if $\bar v$ is of type L.}
    \end{cases}
\end{align*}
If $\bar v$ is the root or of type L, choose $a_{\bar v, 1}$ uniformly at random from $[0,a_{\bar v}]$; and $(a_{\bar v, i},2\leq i\leq 1+\tau_{\bar v})$ as points of a mixed Poisson point process on $(a_{\bar v}, 1]$ with intensity
\begin{equation*}
    \lambda_{\bar v}(y)dy:=\frac{Z_{\bar v}}{\mu a^{1/\mu}_{\bar v}}y^{1/\mu-1}dy. 
\end{equation*}
If $\bar v$ is of type R, then $(a_{\bar v, i},1\leq i\leq \tau_{\bar v})$ are sampled as points of a mixed Poisson process on $(a_{\bar v}, 1]$ with intensity $\lambda_{\bar v}$. We obtain $(\mathcal{T},0)$ by continuing this process ad infinitum.
\end{defn}

\begin{remark}
For any $1\leq r<\infty$, the gamma variable of vertex $(0,1,...,1)\in \partial B_{r}$ is size-biased by the outgoing edge from $k_{0,1,...1}\in\partial \mathcal{B}_{r-1}$ to $k_{0,1,...,1}\in \partial \mathcal{B}_r$, therefore the shape parameter increases by one. 
\end{remark}

Before stating the theorem, we also define the total variation distance of two probability distributions $\nu_1$ and $\nu_2$ on the same countable probability space $\Omega$ as 
\begin{align*}
    d_{\mathrm{TV}}\left(\nu_1,\nu_2\right)
    &=\sup_{A\subseteq \Omega} |\nu_1(A)-\nu_2(A)|\numberthis \label{tvddef}\\
    &=\inf\{\Prob(V\not = W):\text{$(V,W)$ is a coupling of $\nu_1$ and $\nu_2$}\}.\numberthis \label{coupinter}
\end{align*}
Let $\mathcal{G}$ be the set of connected, rooted finite graphs, then $(B_r(G_n,k_0),k_0)$ and $(B_r(\mathcal{T},0),0)$ are random elements of $\mathcal{G}$. So taking $\Omega=\mathcal{G}$, we may consider the total variation distance between $(B_r(G_n,k_0),k_0)$ and $(B_r(\mathcal{T},0),0)$. Definition (\ref{coupinter}) is useful because our main tools are coupling techniques. Denoting the distributional law of any random element by $\mathcal{L}(\cdot)$, we are ready to state our weak local limit result. We emphasize that the local weak convergence does not take into account the ages and fitness of the $\pi$-P\'olya point tree, but they are important for the graph construction and are used for the graph couplings later. 

\begin{thm}\label{bigthm}
Assume that $\pi$ is a fitness distribution supported on $(0,\kappa]$ for some $\kappa<\infty$. Let $G_n\sim \mathrm{PA}(\pi)_n$, $k_0$ be the uniformly chosen vertex of $G_n$ and $(\mathcal{T},0)$ be a $\pi$-P\'olya point tree. Then given $r<\infty$ and $n\gg r$, there is a positive constant $C:=C(X_1,\mu,r,\kappa)$ such that 
\begin{align}\label{maindtv}
    d_{\mathrm{TV}}\left(\mathcal{L}((B_r(G_n,k_0),k_0)),\mathcal{L}((B_r(\mathcal{T},0),0) \right)\leq C (\log \log n)^{-\chi}.
\end{align}
In particular, this implies the weak local limit of $(G_n,k_0)$ is the $\pi$-P\'olya point tree.
\end{thm}

\begin{remark}
The bound on the total variation distance follows from the fact that $k_0\geq n(\log\log n)^{-1}$ with probability at least $(\log \log n)^{-1}$, and on this event, we can couple the two graphs such that $(B_r(G_n,k_0),k_0)\cong (B_r(\mathcal{T},0),0)$ with high probability. It is likely possible to improve the rate of convergence by optimising this and similar choices of thresholds, but with much added technicality.
\end{remark}

\begin{remark}\label{loop}
Theorem \ref{bigthm} can be generalised for the model with multiple edges by modifying the proofs here. Furthermore, the theorem should hold for fitness distributions with exponentially decaying tails, but the assumption of bounded fitness greatly simplifies our proof.

When $X_1=0$ and $X_i=1$ almost surely for all $i\geq 2$, \textcite[Theorem 1]{bubeck2015influence} established that the choice of the seed graph $G_1$ has no effect on the weak local limit; and by simply replacing $G_1$ in our proof, we can prove that Theorem \ref{bigthm} holds for more general seed graphs. Moreover, consider the following model that allows for self-loops. Given that $\mathbf{X}=\mathbf{x}$ and $G'_1$ is a single vertex with initial attractiveness $X_1:=x_1$, for $2\leq m\leq n$, we construct the graph $G'_m$ from $G'_{m-1}$ by attaching vertex $m$ to vertex $k\in \{1,...,m\}$ with probability proportional to its in-degree plus $x_k$. Let $k'_0$ be the uniformly chosen vertex in $G'_n$. With some straightforward adjustments to our proofs, we can show that when the fitness is bounded, the $\pi$-P\'olya point tree is the weak local limit of $(G'_n,k'_0)$, and the bound on the total variation distance is of order at most $(\log\log n)^{-\chi}$. However, we again work in the simplified settings to streamline the arguments. 
\end{remark}

\subsection{Applications of Theorem \ref{bigthm} to some degree statistics}\label{subsecapp}
Using Theorem \ref{bigthm}, we can obtain the limiting degree distributions of the vertices in the local neighbourhood of the uniformly chosen vertex $k_0$ of $G_n\sim \mathrm{PA}(\pi)_n$. We focus on $k_0$ and the type L vertices. Note that the type L vertices are the ancestors of the uniformly chosen vertex, and they are of particular interest in fringe tree analysis (see for example, the recent survey by \textcite{holmgren2017fringe}). We state these results here, starting from the uniformly chosen vertex. Let $D_{i,n}:=W_{i,n}+1$ be the degree of vertex $i$, noting that $W_{i,n}$ is the in-degree of vertex $i$ in $G_n$. Define $D^0_n:=D_{k_0,n}$. In view of Definition \ref{localweak}, the limiting distribution of $D^0_n$ and the rate of convergence of the total variation distance can be read from Theorem \ref{bigthm}. However, we shall modify the coupling proof of Theorem \ref{bigthm}, and prove that this result holds without the assumption of bounded fitness, and derive a sharper convergence rate.

\begin{thm}\label{uniformvert}
Assuming that $\E X^p_2<\infty$ for some $p>4$, let $\xi_0:=\tau_0+1$, with $\tau_0$ being the random variable with distribution 
\begin{equation*}
    \mathrm{Po}\left(Z_0(a^{-1/\mu}_0-1) \right),
\end{equation*}
where given $X_0\sim \pi$, $Z_0\sim \mathrm{Gamma}(X_0,1)$, and independently of $Z_0$, $U_0\sim \mathrm{U}[0,1]$ and $a_0:=U^\chi_0$. There are positive constants $C:=C(X_1,\mu,p)$ and $d:=d(\mu,p)<1$ such that
\begin{align}\label{unidegtvd}
    d_{\mathrm{TV}}\left(\mathcal{L}(D^0_n), \mathcal{L}(\xi_0)\right) \leq Cn^{-d}.
\end{align}
\end{thm}

Next, we give the probability mass function of the distribution of $\xi_0$, which is helpful for relating Theorem \ref{uniformvert} and some known results. Below we write $a_n\sim b_n$ to indicate $\lim_{n\to\infty} a_n/b_n=1$.

\begin{prop}\label{pmfuni}
Retaining the notations in Theorem \ref{uniformvert}, the probability mass function of the distribution of $\xi_0$ is given by 
\begin{align}\label{degdistn}
     p_{\pi}(k)= (\mu+1)\int^\infty_0 \frac{\Gamma(x+k-1)\Gamma(x+\mu+1)}{\Gamma(x)\Gamma(x+\mu+k+1)} d\pi(x),\qquad k\geq 1.
\end{align}
Furthermore, if $\E X^{\mu+1}_2<\infty$, then as $k\to\infty$,
\begin{equation}\label{powerlaw}
    p_{\pi}(k)\sim C_{\pi} k^{-(\mu+2)},\qquad C_{\pi}:=(\mu+1)\int^\infty_0 \frac{\Gamma(x+\mu+1)}{\Gamma(x)} d\pi(x).
\end{equation}
\end{prop}

\begin{remark}\label{extrarelatedwork}
When for $i\geq 2$, $X_i=1$ almost surely, $(p_{\pi}(k),k\geq 1)$ is the probability mass function of $\mathrm{Geo}_1(\sqrt{U})$, where $\mathrm{Geo}_1(p)$ is the geometric distribution supported on the positive integers with parameter $p$ and $U$ is a standard uniform variable. For such fitness sequence (with $X_1=2$) and the model in Remark \ref{loop}, \textcite{bollobas2001degree} established that the limiting degree distribution of the uniform vertex is $(p_{\pi}(k),k\geq 1)$; and using Stein's method for the geometric distribution, \textcite[Theorem 6.1]{pekoz2013total} showed that the total variation distance is of order at most $n^{-1}\log n$. 

When our model is extended to the multi-edge setting and the fitness distribution has a finite mean, \textcite{lodewijks2020phase} used stochastic approximation to obtain the almost sure limit of the empirical degree distribution (\cite[Theorem 2.4]{lodewijks2020phase}), and analysed the tail behaviour of the limit for different types of fitness distribution (\cite[Theorem 2.6]{lodewijks2020phase}). In the tree setting, \cite{lodewijks2020phase} showed that the almost sure limit is given by $(p_{\pi}(k),k\geq 1)$ and established the power law behaviour in (\ref{powerlaw}) when $\E X^{\mu+1}_2<\infty$. However, we note that the representation in Theorem \ref{uniformvert} was not given by \cite{lodewijks2020phase}.
\end{remark}

The next theorem concerns the joint degree distribution of the uniformly chosen vertex and its type L vertices, and can be read from Theorem \ref{bigthm}. To state the result, let $L[0]:=0$ and $L[q]=(0,1,1,...,1)$, $|L[q]|=q+1$ for $q\geq 1$, so that $k_{L[0]}:=k_0$, and for $q\geq 1$, $k_{L[q]}$ is the type L vertex in $\partial \mathcal{B}_q:=V(B_q(G_n,k_0)\setminus V(B_{q-1}(G_n,k_0))$. Observe that $D_{k_{L[q]},n}=\theta_{L[q]}+2$, where $\theta_{L[q]}$ is defined as the number of type R neighbours of vertex $k_{L[q]}$, and the additional edges are the edges joining $k_{L[q]}$ to the vertices $k_{L[q-1]}$ and $k_{L[q+1]}$. Fix $r<\infty$. Note that on the event $k_{L[q]}=1$ for some $q<r$, $k_{L[q+1]},...,k_{L[r]}$ do not exist. Hence, we define $D^q_n:=D_{k_{L[q]},n}$ for $1\leq q\leq r$ if $k_{L[q]}\not = 1$ for all $1\leq q\leq r$; and if $k_{L[j]}=1$ for some $j\leq r$, we let $D^q_n:=D_{k_{L[q]},n}$ for $q<j$ and $D^k_n=-1$ for $j\leq k\leq r$. As discussed in Section \ref{UHN}, with high probability we do not observe vertex 1 in the local neighbourhood of $k_0$, so $(D^0_n,...,D^r_n)$ can be understood as the joint degree sequence of $k_0, k_{L[1]},...,k_{L[r]}$. 

\begin{thm}\label{ancestors}
Retaining the notations above , assume that the fitness distribution $\pi$ is supported on $(0,\kappa]$ for some $\kappa<\infty$, and given $r<\infty$, let $n\gg r$. Define $U_0\sim \mathrm{U}[0,1]$, $a_{L[0]}=U^\chi_0$; and given $a_{L[q-1]}$ for $1\leq q\leq r$, let $a_{L[q]}\sim \mathrm{U}[0,a_{L[q-1]}]$. Independently from $(a_{L[i]}, 0\leq i\leq r)$, let $X_{L[q]}$ be i.i.d.\ random variables with distribution $\pi$, and
\begin{align*}
Z_{L[q]}\sim
    \begin{cases}
    \mathrm{Gamma}(X_{L[q]},1),\quad\text{if $q=0$,}\\
    \mathrm{Gamma}(X_{L[q]}+1,1),\quad \text{if $1\leq q\leq r$.}
    \end{cases}
\end{align*}
Let $\tau_{L[q]}$ be conditionally independent random variables with distributions
\begin{align*}
    \mathrm{Po}\left(Z_{L[q]} (a^{-1/\mu}_{L[q]}-1) \right),
\end{align*}
and define $\tau:=(\tau_0+1, \tau_{L[1]}+2,...,\tau_{L[r]}+2)$ and $D_n:=(D^0_n, D^1_n,...,D^r_n)$. There is a positive constant $C:=C(X_1,\mu,r,\kappa)$ such that 
\begin{align*}
    d_{\mathrm{TV}}\left(\mathcal{L}(D_n)), \mathcal{L}(\tau)\right)\leq C(\log\log n)^{-\chi}. 
\end{align*}
\end{thm}

In the following, we give the limiting probability mass function of $D^1_n$, and shows that the distribution of $D^1_n$ also exhibits a power-law behaviour. 

\begin{prop}\label{ancestorpowerlaw}
Retaining the assumption and the notations of Theorem \ref{ancestors}, the probability mass function of the random variable $\tau_{L[1]}+2$ is given by
\begin{align}\label{dlimit}
    q_{\pi}(k)=\mu(\mu+1)(k-1)\int^\infty_0\frac{\Gamma(x+k-1)\Gamma(x+\mu+1)}{\Gamma(x+1)\Gamma(x+\mu+k+1)}d\pi(x),\qquad k\geq 2.
\end{align}
Furthermore, if $\E X^\mu_2<\infty$, then as $k\to\infty$,
\begin{align*}
q_{\pi}(k)\sim C_{\pi}k^{-(\mu+1)},\qquad C_{\pi}=\mu(\mu+1)\int^\infty_0\frac{\Gamma(x+\mu+1)}{\Gamma(x+1)} d\pi(x).
\end{align*}
\end{prop}

Comparing Proposition \ref{pmfuni} and \ref{ancestorpowerlaw}, we see that in the limit, the degree of vertex $k_{L[1]}:=k_{0,1}$ has a heavier tail than the degree of the uniform vertex $k_0$. This is due to the fact that $k_{L[1]}$ has received an incoming edge from the uniform vertex, and so $k_{L[1]}$ is more likely to have a higher degree than $k_0$.

\subsection{Idea of proof of Theorem \ref{bigthm}}
\subsubsection{Two graphs for intermediate coupling steps}\label{twographs}
Let $\mathbf{x}$ be a realisation of the fitness sequence $\mathbf{X}$. To prove Theorem \ref{bigthm} using graph couplings, we mostly work with the $(\mathbf{x},n)$-sequential model. We now introduce two random trees used in the intermediate coupling steps that are constructed using $\mathbf{x}$. Once we define these trees, we provide the finer detail of the graph couplings. The first tree is an alternative definition of the $(\mathbf{x},n)$-sequential model, which is commonly known as the P\'olya urn representation, and is better related to the $\pi$-P\'olya point tree. Its construction relies on the fact that the dynamics of the preferential attachment graphs can be represented as embedded classical P\'olya urns. By de Finetti's theorem, the attachment steps of the graph constructed using P\'olya urns are conditionally independent, so the graph is more tractable for our analysis. 

When $X_1=0$ and $X_i=1$ almost surely for $i\geq 2$, \textcite{Berger12asymptoticbehavior} used a P\'olya urn representation to prove their weak local limit theorem. The representation result below is Theorem~1 of \textcite{delphin2019}. The construction uses a line-breaking procedure, where we sample a sequence of conditionally independent beta variables.

\begin{defn}[$(\mathbf{x},n)$-P\'olya urn tree]\label{urnrep}
Given $\mathbf{x}$ and $n$, let $T_m:=\sum^m_{j=1} x_j$, and $(B^{(x)}_j, 1\leq j\leq n)$ be independent random variables such that $B^{(x)}_1:=1$ and
\begin{equation} 
   B^{(x)}_i \sim \mathrm{Beta}(x_{i},i-1+T_{i-1})\quad \text{for $2\leq i\leq n$.}
\end{equation}
 Moreover, let $S^{(x)}_{0,n}:=0$, $S^{(x)}_{n,n}:=1$ and
\begin{equation*}
   S^{(x)}_{k,n}:=\prod^{n}_{i=k+1}(1-B^{(x)}_i)\quad\text{for $1\leq k \leq n-1$.}
\end{equation*}
Starting with $n$ vertices labelled $\{1,...,n\}$ and no edges between them, we connect them as follows. Let $I_j=[S^{(x)}_{j-1,n}, S^{(x)}_{j,n})$ for $1\leq j \leq n$. Conditionally on $(S^{(x)}_{i,n},1\leq i \leq n-1)$, let $(U_k, 2\leq k\leq n)$ be independent variables such that $U_k~\sim~\mathrm{U}[0,S^{(x)}_{k-1,n}]$. If $j<k$ and $U_k\in I_j$, then we attach an outgoing edge from vertex $k$ to vertex $j$. We say that the resulting graph is a $(\mathbf{x},n)$-P\'olya urn tree.
\end{defn}

Note that $1-B^{(x)}_{j}$ in Definition \ref{urnrep} is $\beta_{j-1}$ in \cite{delphin2019}. The key result relating the $(\mathbf{x},n)$-sequential model and the $(\mathbf{x},n)$-P\'olya urn tree is the following.

\begin{thm}[Theorem 1, \textcite{delphin2019}]\label{linebreaking}
Let $G_n$ be an $(\mathbf{x},n)$-P\'olya urn tree, then $G_n\sim \mathrm{Seq}(\mathbf{x})_n$.
\end{thm}

Unlike Theorem \ref{bigthm}, the theorem does not assume $x_i$ are uniformly bounded. Moreover, it is possible to derive similar urn representations for preferential attachment models that allow for multi-edges and self-loops (see for example, \cite[Theorem 2.1]{Berger12asymptoticbehavior} and \textcite[Lemma 1 and 2]{pekoz2017}), and proving weak local limits of these models using their respective representations. We reproduce the proof of Theorem~\ref{linebreaking} in Section \ref{pflinebreaking}, as it will be useful for proving a variation of this result that we need for the proof of Theorem \ref{bigthm}. An example of the $(\mathbf{x},n)$-P\'olya urn tree is given in Figure~\ref{polyagraph}. 

In preparation for the graph coupling, we equip the breadth-first search of the $(\mathbf{x},n)$-P\'olya urn tree (or equivalently the $(\mathbf{x},n)$-sequential model) with the Ulam-Harris labels $k_{\bar w}$, as defined in Section \ref{UHN}. Let $G_n\sim\mathrm{Seq}(\mathbf{x})_n$ and $k_0$ be its uniform vertex. Then, the probability that vertex $k$ attaches to $k_0$ can be read from Theorem \ref{linebreaking}, and importantly, we can encode the type R neighbours of vertex $k_0$ in the (mixed) Bernoulli point process $(\mathbbm{1}[U_k\in I_{k_0}], k_0+1\leq k\leq n)$. This Bernoulli point process can be coupled to a mixed Poisson point process which, after randomisation of the fitness sequence $\mathbf{X}$, is the mixed Poisson point process on $(a_0,1]$ with intensity $\lambda_0$. As conditioning on the discovered edges in the breadth-first search of $(G_n,k_0)$ changes the distributions of the vertex weights, Theorem \ref{linebreaking} does not give the probability that an unexplored vertex attaches to a non-root vertex $k_{\bar v}$. However, for a fixed sequence $\mathbf{x}$, it is possible to quantify the conditioning effect. In more detail, the neighbours of $k_{\bar v}=j$ are distributed as the neighbours of vertex $j$ in an $(\mathbf{x},n)$-sequential model conditional on the edges discovered before $k_{\bar v}$. This conditional model has a P\'olya urn representation, from which we can obtain the attachment probabilities for $k_{\bar v}$. In particular, we can derive a Bernoulli point process that encodes the type R neighbours of vertex $k_{\bar v}$, whose distribution is close to that of $(\mathbbm{1}[U_k\in I_{k_{\bar v}}], k_{\bar v}+1\leq k\leq n)$ when $n$ is large enough. Thus, we can couple this modified process and a mixed Poisson point process that is distributed as the mixed Poisson process on $(a_{\bar v},1]$ with intensity $\lambda_{\bar v}$ after randomisation of $\mathbf{X}$. 
\begin{figure}
    \centering
    \begin{tikzpicture}
    \draw (0,0) -- (5.4,0);
    \draw (0,0) -- (0,0.2) node[anchor=south] {$0$};
    \draw (1.5,0) -- (1.5,0.2) node[anchor=south] {$S^{(X)}_{1,5}$};
    \draw (3.2,0) -- (3.2,0.2) node[anchor=south] {$S^{(X)}_{2,5}$};
    \draw (4.1,0) -- (4.1,0.2) node[anchor=south] {$S^{(X)}_{3,5}$};
    \draw (5.4,0) -- (5.4,0.2) node[anchor=south] {$S^{(X)}_{4,5}$};
    \draw [->, thick] (0.8,-0.5)--(0.8,0) node[below=5mm] {$U_2$} ;
    \draw [->, thick] (1.9,-0.5)--(1.9,0) node[below=5mm] {$U_3$} ;
    \draw [->, thick] (2.8,-0.5)--(2.8,0) node[below=5mm] {$U_4$} ;
    \draw [->, thick] (4.6,-0.5)--(4.6,0) node[below=5mm] {$U_5$} ;
    \draw [fill] (8,0) circle [radius=0.1] node[below=3mm] {$1$};
    \draw [fill] (9.5,0) circle [radius=0.1] node[below=3mm] {$2$};
    \draw [fill] (11,0) circle [radius=0.1] node[below=3mm] {$3$};
    \draw [fill] (12.5,0) circle [radius=0.1] node[below=3mm] {$4$};
    \draw [fill] (14,0) circle [radius=0.1] node[below=3mm] {$5$};
    \draw [->,thick] (9.5,0) to[bend right] (8.1,0);
    \draw [->,thick] (11,0) to[bend right] (9.6,0);
    \draw [->,thick] (12.5,0) to[bend right] (9.6,0);
    \draw [->,thick] (14,0) to[bend right] (12.6,0);
\end{tikzpicture}
    \caption{\small An example of the $(\mathbf{x},n)$-P\'olya urn tree for $n=5$, where $U_i\sim\mathrm{U}[0,S^{(x)}_{i-1,n}]$ for $i=2,...,5$ and an outgoing edge is drawn from vertices $i$ to $j$ if $U_i\in [S^{(x)}_{j-1,n}, S^{(x)}_{j,n})$.}
    \label{polyagraph}
\end{figure}
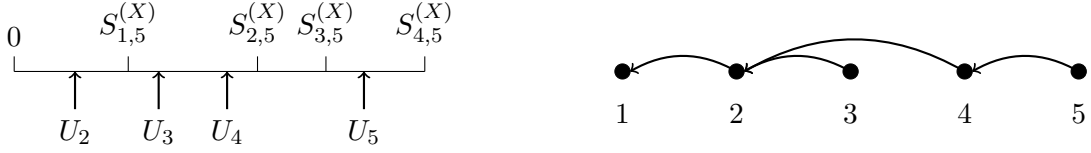
\normalsize

Hence using the mixed Poisson point processes conditional on $\mathbf{X}=\mathbf{x}$, and the betas for constructing the urn representations, we generate a conditional analog of the $\pi$-P\'olya point tree, which we call the intermediate P\'olya point tree. We denote this rooted random tree by $(\mathcal{T}_{\mathbf{x},n},0)$, where vertex $0$ is its root, and the subscripts are parameters corresponding to $\mathbf{x}$ and $n$ in the $(\mathbf{x},n)$-P\'olya urn tree. To construct the point processes in a similar way as in Definition \ref{pipolyapointgraph}, we adorn each vertex of this tree with an Ulam-Harris label $\bar v$, an age $\hat a_{\bar v}$, a type (except for the root), and a \textit{PA label} $\hat k_{\bar v}$. Moreover, vertex $\bar v$ has a random number of type R neighbours $\hat \tau_{\bar v}$, which is exactly the number of points of the mixed Poisson process. The Ulam-Harris labels, types and ages are defined similarly as in Section \ref{UHN}, and they are useful for matching the vertices of $(\mathcal{T}_{\mathbf{x},n},0)$ and the $\pi$-P\'olya point tree $(\mathcal{T},0)$. On the other hand, the PA label $\hat k_{\bar v}$ of vertex $\bar v\in V((\mathcal{T}_{\mathbf{x},n},0))$ determines the initial attractiveness of the vertex by setting its attractiveness as $x_{\hat k_{\bar v}}$. Importantly, $\hat k_0$ is distributed as the uniformly chosen vertex $k_0$; and for $\bar v\not =0$, the PA label $\hat k_{\bar v}$ is approximately distributed as $k_{\bar v}\in V((G_n,k_0))$. So using $\hat k_{\bar v}$, we can match the vertices $k_{\bar v}\in V((G_n,k_0))$ and $\bar v\in V((\mathcal{T}_{\mathbf{x},n},0))$, and hence the initial attractiveness of these vertices in the graph coupling. 

Before defining $(\mathcal{T}_{\mathbf{x},n},0)$, we need a final notation to help us track the existing vertices in the construction below. Write $\bar w <_{UH} \bar v$ if $\bar w$ is smaller than $\bar v$ in the breadth-first order. Letting $|\cdot|$ be the cardinality of any set, this means that either $|\bar w|< |\bar v|$, or when $\bar w=(0,w_1,...,w_q)$ and $\bar v=(0,v_1,...,v_q)$, $w_j<v_j$, where $j=\min\{l: v_l\not= w_l\}$. For example, $(0,2,3)<_{UH} (0,1,1,1)$ and $(0,3,1,5)<_{UH} (0,3,4,2)$. If $\bar w$ is either smaller than or equal to $\bar v$ in the breadth-first order, then we write $\bar w \leq_{UH} \bar v$. 

\begin{defn}[Intermediate P\'olya point tree]\label{condpttree}
Given $n$ and $\mathbf{x}$, $(\mathcal{T}_{\mathbf{x},n},0)$ is constructed recursively as follows. The root 0 has an age $\hat a_0=U_0^\chi$ and an initial attractiveness $x_{\hat k_0}$, where $U_0\sim \mathrm{U}[0,1]$ and $\hat k_0=\ceil{nU_0}$. Assume that $((\hat a_{\bar w}, \hat k_{\bar w}),\bar w \leq_{UH} \bar v)$ have been generated, such that $\hat k_{\bar w}>1$ for all $\bar w \leq_{UH} \bar v$. If vertex $\bar v$ is the root or belongs to type L, we generate $((\hat a_{\bar v,i}, \hat k_{\bar v,i}),1\leq i\leq 1+\hat \tau_{\bar v})$ as follows.
\begin{enumerate}
    \item \label{Lstep}We sample the age of the type L neighbour $(\bar v,1)$ by letting $U_{\bar v,1}\sim\mathrm{U}[0,1]$ and $\hat a_{\bar v,1}=\hat a_{\bar v} U_{\bar v,1}$.
    \item \label{betastep} Next we choose the PA label $\hat k_{\bar v,1}$. Noting that $\{\hat k_{\bar u}: \bar u <_{UH} 0 \}=\varnothing$, we define the independent variables $((\zeta_j[\bar v], \tilde \zeta_j[\bar v]), j\in\{2,...,n\}\setminus \{\hat k_{\bar u}:\bar u<_{UH}\bar v\})$. Define $T_{m}:=\sum^m_{j=1} x_j$. When $\bar v=0$, let $\zeta_j[0]\sim \mathrm{Gamma}(x_j,1)$ and $\tilde \zeta_i[0]\sim \mathrm{Gamma}(T_{i-1}+i-1,1)$; whereas when $\bar v=(0,1,1,...,1)$, let $\zeta_j[\bar v]\sim \mathrm{Gamma}(x_j + \mathbbm{1}[j=\hat k_{\bar v}],1)$, and 
    \begin{align*}
        \tilde \zeta_j[\bar v]\sim
        \begin{cases}
        \mathrm{Gamma}(T_{j-1}+j-1,1),\quad 2\leq j\leq \hat k_{\bar v},\\
        \mathrm{Gamma}(T_{j-1}+j,1),\quad \hat k_{\bar v}<j<\hat k_{\bar v'}\\
        \mathrm{Gamma}\left(T_{j-1}+j+1-|\bar v|-\sum_{\{\hat k_{\bar x}<j: \bar x <_{UH} \bar v\}}\{x_{\hat k_{\bar x}}+\hat \tau_{\bar x}\},1\right),\quad \hat k_{\bar v'}<j\leq n,
        \end{cases}
    \end{align*}
    where $\bar v'=(0,1,1,...,1)$ and $|\bar v'|=|\bar v|-1$. For either the root or the type L vertex, define $\beta_1[\bar v]:=1$, $\beta_i[\bar v]:=0$ for $i \in \{\hat k_{\bar s}: \bar s <_{UH} \bar v\}$ and 
    \begin{align}\label{condbeta}
        \beta_j[\bar v]:=\frac{\zeta_j[\bar v]}{\zeta_j[\bar v] + \tilde \zeta_j[\bar v]},\quad j\in\{2,...,n\}\setminus \{\hat k_{\bar u}:\bar u<_{UH}\bar v\};
    \end{align}
    then let $\mathcal{S}_{0,n}[\bar v]:=0$, $\mathcal{S}_{n,n}[\bar v]:=1$ and $\mathcal{S}_{k,n}[\bar v]:=\prod^n_{\ell=k+1}(1-\beta_\ell[\bar v])$ for $1\leq k\leq n-1$. Choose the PA label $\hat k_{\bar v,1}$ such that $ \mathcal{S}_{\hat k_{\bar v,1}-1,n}[\bar v]\leq U_{\bar v,1} \mathcal{S}_{\hat k_{\bar v}-1,n}[\bar v] < \mathcal{S}_{\hat k_{\bar v,1},n}[\bar v]$.
    
    \item \label{Rstep} We generate the ages and PA labels of the type R neighbours. Let $(\hat a_{\bar v,i},2\leq i\leq 1+\hat \tau_{\bar v})$ be points of a mixed Poisson process on $(\hat a_{\bar v},1]$ with intensity 
    \begin{equation}\label{condintensity}
    \hat \lambda_{\bar v}(y)dy:=\frac{\zeta_{\hat k_{\bar v}}[\bar v]}{\mu \hat a^{1/\mu}_{\bar v}} y^{1/\mu-1} dy.
    \end{equation}
    Denote $M_{\bar v}:=\min\{k: (k/n)^\chi \geq \hat a_{\bar v}\}$. Then choose $\hat k_{\bar v,i}$ such that 
    \begin{equation}\label{bins}
   ((\hat k_{\bar v,i}-1)/n)^\chi <\hat a_{\bar v,i}\leq  (\hat k_{\bar v,i}/n)^\chi\quad \text{for $M_{\bar v}\leq \hat k_{\bar v,i}\leq n$.}
    \end{equation}
\end{enumerate}
If $\bar v$ belongs to type R, let $\zeta_{\hat k_{\bar v}}[\bar v]\sim \mathrm{Gamma}(x_{\hat k_{\bar v}},1)$ and apply step \ref{Rstep} only to obtain $((\hat a_{\bar v,l}, \hat k_{\bar v,l}), 1\leq l\leq \hat \tau_{\bar v})$. We build $(\mathcal{T}_{\mathbf{x},n},0)$ by iterating this process, and terminate the construction whenever there is some vertex $\bar v$ such that $\hat k_{\bar v}=1$. 
\end{defn}

We give several remarks on the construction of $(\mathcal{T}_{\mathbf{x},n},0)$. When $\hat k_{\bar v}=1$, it must be the case that $\bar v=0$ or $\bar v=(0,1...,1)$. We stop the construction in this case because vertex 1 does not have a type L neighbour, so steps \ref{Lstep} and \ref{betastep} are unnecessary, and $\zeta_1[\bar v]$ is undefined in step \ref{Rstep}. Nevertheless, for $r<\infty$ and any vertex $\bar v$ in the $r$-neighbourhood of vertex $0\in V(\mathcal{T}_{\mathbf{x},n},0)$, $B_r(\mathcal{T}_{\mathbf{x},n},0)$, the probability that $\hat k_{\bar v}=1$ tends to zero as $n\to\infty$.  

By the beta-gamma algebra, $(\beta_j[0], 1\leq j\leq n)=_d (B^{(x)}_j, 1\leq j\leq n)$ and $(\mathcal{S}_{k,n}[0], 1\leq k\leq n)=_d (S^{(x)}_{k,n}, 1\leq k\leq n)$, where $B^{(x)}_j$ and $S^{(x)}_{k,n}$ are as in Definition \ref{urnrep}. We use the gamma variables to generate the beta variables to better compare $(\mathcal{T}_{\mathbf{x},n},0)$ and $(\mathcal{T},0)$ in the coupling. When vertex $\bar v\not=0$, the beta variables $\beta_j[\bar v]$ are building blocks of the urn representation of the $(\mathbf{x},n)$-sequential model conditional on the edges joining the vertices $\{\bar w:\bar w\leq_{UH} \bar v\}$, with the parameters chosen to adjust for the conditioning effect. More specifically, we can use these betas to construct a Bernoulli point process that encodes the type R neighbours of vertex $k_{\bar v}$ in the $r$-neighbourhood of $k_0$, $B_r(G_n,k_0)$, and then couple with the mixed Poisson point process on $(\hat a_{\bar v},1]$ with intensity $\hat \lambda_{\bar v}$. The comprehensive arguments for deriving these betas are deferred to Section \ref{indeg01} and \ref{embmodel}. 

\subsubsection{Coupling of the 1-neighbourhoods }\label{couplingidea}
Keeping the notations above, we are ready to give an overview of the proof. The proof proceeds in two major steps. Firstly, given that $G_n\sim\mathrm{Seq}(\mathbf{x})_n$, we couple $(G_n,k_0)$ and $(\mathcal{T}_{\mathbf{x},n},0)$ such that $(B_r(G_n,k_0),k_0)\cong (B_r(\mathcal{T}_{\mathbf{x},n},0),0)$ with high probability. Denote by $(\mathcal{T}_{\mathbf{X},n},0)$ be the intermediate P\'olya point tree after randomisation of $\mathbf{X}$. In the second step, we bound the total variation distance between $(B_r(\mathcal{T}_{\mathbf{X},n},0),0)$ and $(B_r(\mathcal{T},0),0)$. The theorem then follows from applying the triangle inequality for the total variation distance.

We begin by sketching out the coupling of $(G_n,k_0)$ and $(\mathcal{T}_{\mathbf{x},n},0)$. As the vertices of $G_n$ do not have ages, we define the \textit{age} of vertex $j$ in $G_n$ as $(j/n)^\chi$ to compare $(G_n,k_0)$ and $(\mathcal{T}_{\mathbf{x},n},0)$. In the coupling, we consider the vertex pairs $k_{\bar w} \in V(B_{r-1}(G_n,k_0))$ and $\bar w \in  V(B_{r-1}(\mathcal{T}_{\mathbf{x},n},0))$ in the breadth-first order. Assuming that $\bar w$ and $k_{\bar w}$ are already coupled such that $\hat k_{\bar w}=k_{\bar w}$ and the ages are sufficiently close: $(k_{\bar w}/n)^\chi\approx \hat a_{\bar w}$, we construct a coupling such that with high probability, $\hat \tau_{\bar w}=\theta_{\bar w}$, and for all $1\leq i\leq \hat \tau_{\bar w}+\mathbbm{1}[\bar w=0\text{ or }(0,1,1,...,1)]$, $k_{\bar w,i}=\hat k_{\bar w,i}$ and $\hat a_{\bar w,i}\approx(k_{\bar w,i}/n)^\chi$. Note that when $\theta_{\bar u}=\hat\tau_{\bar u}$ for all $\bar u\in V(B_{r-1}(\mathcal{T}_{\mathbf{x},n},0))$, we have $(B_r(G_n,k_0),k_0)\cong (B_r(\mathcal{T}_{\mathbf{x},n},0),0)$.

Next, we elaborate on this outline, but to spare the technical detail of conditioning on the existing edges in the case of type L and R vertices, below we only focus on coupling the 1-neighbourhoods of $(G_n,k_0)$ and $(\mathcal{T}_{\mathbf{x},n},0)$. 

\textbf{Step 1.} We couple the initial attractiveness and the ages of the root vertices. Let $U_0\sim \mathrm{U}[0,1]$, $k_0=\ceil{n U_0}$ and $\hat a_0=U^\chi_0$. In light of Definition \ref{condpttree}, we have $\hat k_0=k_0$, implying that the initial attractiveness of the root of $(\mathcal{T}_{\mathbf{x},n},0)$ is $x_{k_0}$, and $\hat a_0\approx (k_0/n)^\chi$. The closeness of the ages and the matching of the initial attractiveness are required to couple $\theta_0$ and $\hat \tau_0$, as well as the ages and initial attractiveness of $k_{0,i}\in V(B_1(G_n,k_0))$ and $(0,i)\in V(B_1(\mathcal{T}_{\mathbf{x},n},0))$, as we illustrate next.

\textbf{Step 2a.} Given that the root vertices are coupled as in the previous step, we construct a coupling such that with high probability, $\hat \tau_0=\theta_0$, and for $2\leq i\leq 1+\theta_0$, $\hat k_{0,i}=k_{0,i}$ and $\hat a_{0,i}\approx (k_{0,i}/n)^\chi$. We build a Bernoulli point process on $((k_0/n)^\chi,1]$ that encodes the vertex labels in $G_n\sim \mathrm{Seq}(\mathbf{x})_n$. For $2\leq k\leq n$, let $U_k\sim \mathrm{U}[0,S^{(x)}_{k-1,n}]$ and $I_{k_0}=[S^{(x)}_{k_0-1,n}, S^{(x)}_{k_0,n})$. As shown in Figure~\ref{bernoulliprocesspic}, we put a point on $(k/n)^\chi$ if and only if $\mathbbm{1}[U_k \in I_{k_0}]=1$ for $k_0+1\leq k\leq n$, which from Definition \ref{urnrep}, occurs with the conditional probability
\begin{equation}\label{berprob}
    \frac{S^{(x)}_{k_0,n}}{S^{(x)}_{k-1,n}} B^{(x)}_{k_0}.
\end{equation}

\begin{figure}
    \centering
    \begin{tikzpicture}
    \draw (0,0) -- (9,0);
    \draw (0,0) -- (0,0.2) node[anchor=south] {$\left(\frac{k_0}{n}\right)^\chi$};
    \draw (9,0) -- (9,0.2) node[anchor=south] {$1$};
    \foreach \y in {0.2,0.4,...,8.8}
    {
    \draw (\y, 0) -- (\y, 0.2);
    }
    \draw (2.2,0) -- (2.2,0.2) node[anchor=south] {$\left(\frac{j_1}{n}\right)^\chi$};
    \node at (2.2,0.1) {$\times$};
    \draw (3.4,0) -- (3.4,0.2) node[anchor=south] {$\left(\frac{j_2}{n}\right)^\chi$};
    \node at (3.4,0.1) {$\times$};
    \draw (6.4,0) -- (6.4,0.2) node[anchor=south] {$\left(\frac{j_3}{n}\right)^\chi$};
    \node at (6.4,0.1) {$\times$};
    \draw [fill] (0,-1.5) circle [radius=0.1] node[below=3mm] {$k_0$};
    \draw [fill] (2.2,-1.5) circle [radius=0.1] node[below=3mm] {$j_1$};
    \draw [fill] (3.4,-1.5) circle [radius=0.1] node[below=3mm] {$j_2$};
    \draw [fill] (6.4,-1.5) circle [radius=0.1] node[below=3mm] {$j_3$};
    \draw [->, thick] (2.2,-1.5) to[bend right] (0.1,-1.5);
    \draw [->, thick] (3.4,-1.5) to[bend right] (0.1,-1.5);
    \draw [->, thick] (6.4,-1.5) to[bend right] (0.1,-1.5);
\end{tikzpicture}
    \caption{\small An illustration of the relation between the Bernoulli point process on $((k_0/n)^\chi,1]$ constructed using $(\mathbbm{1}[U_k \in I_{k_0}],k_0+1\leq i\leq n)$ and the $(\mathbf{x},n)$-P\'olya urn tree. Here $\theta_0=3$, $k_{0,2}=j_1$, $k_{0,3}=j_2$ and $k_{0,4}=j_3$. }
    \label{bernoulliprocesspic}
\end{figure}
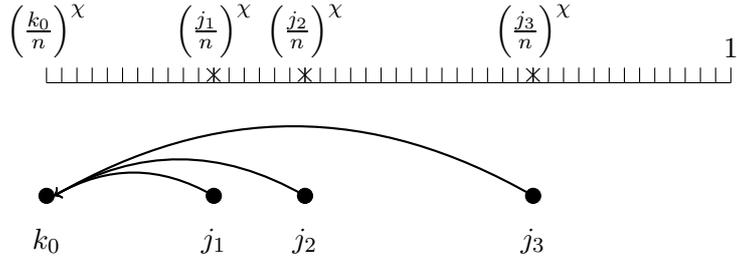
\textbf{Step 2b.} We couple the Bernoulli process and a discretisation of the mixed Poisson process on $(\hat a_0,1]$ that encodes the ages and the PA labels $((\hat k_{0,i},\hat a_{0,i}),2\leq i\leq 1+\hat \tau_0)$, as appeared in step \ref{Rstep} of Definition \ref{condpttree}. To this purpose, we need an estimate of (\ref{berprob}). For $n$ and $k$ sufficiently large, we can apply the law of large numbers and standard moment inequalities to show that $(k/n)^\chi\approx S^{(x)}_{k,n}$ with high probability. So applying the beta-gamma algebra and the law of large numbers to $B^{(x)}_{k_0}$, we can closely couple (\ref{berprob}) and
\begin{equation}\label{coupberprob}
    \left(\frac{k_0}{k}\right)^\chi \frac{\mathcal{Z}_{k_0}}{(\mu+1)k_0},\quad \text{where $\mathcal{Z}_{k_0}\sim \mathrm{Gamma}(x_{k_0},1)$.}
\end{equation}
Importantly, (\ref{berprob}) and (\ref{coupberprob}) are close enough that we can couple the Bernoulli processes constructed using with these means. On the other hand, (\ref{coupberprob}) are small enough that the Bernoulli process with these means is approximately distributed as the discretised mixed Poisson process with means
\begin{equation}\label{dismean}
     \int^{(k/n)^\chi}_{((k-1)/n)^\chi} \frac{\mathcal{Z}_{k_0}}{\mu a^{1/\mu}_0} y^{1/\mu-1}dy;
\end{equation}
noting that $(k_0/n)^\chi \approx \hat a_0$. Thus, we use standard techniques to couple these point processes such that $\theta_0=\hat \tau_0$ and $\hat k_{0,i}=k_{0,i}$ for $i=2,...,1+\tau_0$ with high probability. From (\ref{bins}) of Definition \ref{condpttree}, it is clear that under the coupling, $(k_{0,i}/n)^\chi\approx \hat a_{0,i}$ for $n$ large enough.

\textbf{Step 3.} We couple $\hat k_{0,1}$ and $k_{0,1}$ such that $k_{0,1}=\hat k_{0,1}$, and then show that under the coupling, $\hat a_{0,1}\approx (k_{0,1}/n)^\chi$ with high probability. In view of Definition \ref{urnrep} and \ref{condpttree}, we can achieve this by using the same set of betas in the constructions of the two graphs. To establish the closeness of ages, it is enough to prove that $S^{(x)}_{k_{0,1},n}\approx (k_{0,1}/n)^\chi$ because $\hat a_{0}\approx (k_0/n)^\chi$. This follows showing that with high probability, $(k/n)^\chi\approx S^{(x)}_{k,n}$ for sufficiently large $k$ and $n$, and that $k_{0,1}$ is large enough. This completes the coupling of the 1-neighbourhoods.

We reiterate that although the closeness of the ages are not part of the local weak convergence, we need $\hat a_{0,i}\approx(k_{0,i}/n)^\chi$ to couple the 2-neighbourhoods of $(G_n,k_0)$ and $(\mathcal{T}_{\mathbf{x},n},0)$. This is not possible if for the 1-neighbourhoods, we only couple $\sum^n_{j=k_0+1}\mathbbm{1}[U_k\in I_{k_0}]$ to a Poisson variable with parameter $ \int^1_{a_0}\mathcal{Z}_{k_0}(\mu a^{1/\mu}_0)^{-1} y^{1/\mu-1}dy$. In that case, we can only couple $(G_n,k_0)$ and $(\mathcal{T}_{\mathbf{x},n},0)$ such that $(B_1(G_n,k_0),k_0)\cong (B_1(\mathcal{T}_{\mathbf{x},n},0),0)$ with high probability, which is not enough to prove Theorem \ref{bigthm}. When coupling the $r$-neighbourhoods for $r>1$, we consider the Bernoulli point processes constructed using the betas $\beta_j[\bar v]$, $\bar v\not =0$, in Definition \ref{condpttree}. As $\beta_j[\bar v]$ are approximately distributed as $B^{(x)}_j$ when the number of discovered edges is not too large, an event that occurs with high probability, it follows that the coupling is similar to the above.

Having coupled $(G_n,k_0)$ and $(\mathcal{T}_{\mathbf{x},n},0)$ as above, we bound the total variation distance between $(B_r(\mathcal{T}_{\mathbf{X},n},0),0)$ and $(B_r(\mathcal{T},0),0)$. The main issue is that two vertices $\bar u,\bar w\in V(B_r(\mathcal{T}_{\mathbf{x},n},0))$ share the same initial attractiveness $x_k$ (or equivalently the PA label) whenever the Poisson points $\hat a_{\bar u}$ and $\hat a_{\bar w}$ land in the same bin $(((k-1)/n)^\chi,(k/n)^\chi]$, as described in (\ref{bins}); while the fitness of each vertices in $(\mathcal{T},0)$ are independent, and so the initial attractiveness of the vertices in $(\mathcal{T}_{\mathbf{x},n},0)$ must be derived from separate $x_i$. Hence to bound the total variation distance, we couple the graphs using a simple procedure, and show that the probability that the PA labels are not distinct tends to zero as $n\to\infty$. 

\subsection{Related works}\label{recentw}
In this section we give an overview on the recent development in the study of preferential attachment graphs with additive fitness. We also collect some known results on the weak local limit of the model when the initial attractiveness are equal almost surely. 

The recent work \textcite{lodewijks2020phase} used martingale techniques to investigate the maximum degree for fitness distributions with different tail behaviours, where the results are applicable to preferential attachment graphs with additive fitness that allow for multiple edges. They also studied the empirical degree distribution, whose detail we already mentioned in Remark \ref{extrarelatedwork}. Our model is also a special case of the preferential attachment tree considered in \textcite[Section 3]{iyer2020degree}, where vertices are chosen with probability proportional to a suitable function of their fitness and degrees at each attachment step. Using branching processes, \cite{iyer2020degree} studied the empirical degree distribution under the assumption of bounded fitness, as well as the condensation phenomena of the model. Our model was also studied in \textcite[Section 5.5.1]{bhamidi2007}. Assuming bounded fitness, \cite{bhamidi2007} used theories of branching processes to investigate the empirical degree distribution, the height and the degree of the initial vertex. We note that `global' properties such as the maximum degree, the height and the degree of the initial vertex cannot be deduced from the weak local limit. 

Moreover, our model is closely related to the random recursive trees introduced in \textcite{borovkov2006}. The random recursive tree is constructed follows: starting from a single vertex with weight one, a new vertex with an edge attached to it is added to the existing graph at each step. The new vertex and edge respectively have a random weight and a random length, and given the weights of the existing vertices, the recipient of the new edge is chosen from these vertices with probability proportional to their weights. \cite{borovkov2006} studied the average degree of a fixed vertex and the distance between of a newly added vertex and the initial vertex; while \cite{delphin2019} investigated the degree sequence of fixed vertices, the height and the profile, assuming that the weight of the initial vertex is also random and all edges have length one. As observed in \cite{delphin2019}, we can view the preferential attachment tree as a random recursive tree via Theorem \ref{linebreaking}. More specifically, let $(X_i,i\geq 1)$ is the fitness sequence. Define $B^{(X)}_j\sim \mathrm{Beta}(X_j,\sum^{j-1}_{\ell=1}X_\ell +j-1)$ and $S^{(X)}_{k,n}=\prod^n_{i=k+1}(1-B^{(X)}_i)$, with $S^{(X)}_{0,n}=0$ and $S^{(X)}_{n,n}=1$. Then, the preferential attachment tree is a random recursive tree, where the weights of the vertices $1\leq j\leq n$ are distributed as $(S^{(X)}_{j,n}-S^{(X)}_{j-1,n}, {2\leq j\leq n})$, and each edge has length one.  

We now survey the weak local limit results developed for preferential attachment graphs. When $x_1=0$ and $x_i=1$ for all $i\geq 2$, the $(\mathbf{x},n)$-sequential model is the pure `sequential' model in \textcite{Berger12asymptoticbehavior} with no multi-edges; and a special case of the model considered in \textcite{rudas2007random}, where the `weight' function in their model is the identity function plus one. In \cite[Theorem 2.2]{Berger12asymptoticbehavior}, the authors showed that the weak local limit of several types of preferential attachment random graphs is the P\'olya point tree (see Section 2.3.2 of \cite{Berger12asymptoticbehavior} for the precise definition). This weak local limit is the $\pi$-P\'olya point tree when there is no multi-edges, and the fitness distributions of these models put a unit mass on value one. Furthermore, the models studied in \cite{Berger12asymptoticbehavior} can be generalised in such a way that self-loops are allowed, and the fitness distribution puts a unit mass on $\delta>0$. Thus, \textcite[Chapter 5, Theorem 5.8]{vanderhofstad2} and \textcite[Chapter 4, Theorem 4.2.1]{garavaglia} adapted the proof of \cite{Berger12asymptoticbehavior} to obtain the weak local limit of these models. Theorem \ref{bigthm} can therefore be seen as a generalisation of their results. Using branching process techniques, \cite{rudas2007random} studied the asymptotic distribution of the subtree rooted at a uniformly chosen vertex, which, as observed by \cite{Berger12asymptoticbehavior}, implies the local weak limit of the preferential attachment family. 

Finally, we note that a different preferential attachment random graph with fitness was studied in \cite{bhamidi2007}, \cite{borgs2007first}, \cite{dereich2016preferential}, and \cite{dereich2014robust}, where the probability that a new vertex attaches to an existing vertex is proportional to its fitness \textit{times} its degree.


\subsection{Organisation of the paper}
The remainder of this paper is organised as follows. In the next section, we state the approximations of $S^{(x)}_{k,n}$ and $B^{(x)}_j$ used for the estimation of~(\ref{berprob}); and we describe the distributions of the neighbours of the type L and R vertices in Section \ref{indeg01}. In Section~\ref{slocal} we specialise to coupling 1-neighbourhoods; while in Section \ref{secgenr} we use the results in Section \ref{indeg01} to inductively extend the coupling to the type L and R vertices, hence proving Theorem \ref{bigthm}. Section \ref{ssup} collects the supplementary proofs for Theorem \ref{bigthm}; and in Section \ref{sapp} we prove the results in Section \ref{subsecapp}. In the last section, we construct the P\'olya urn representation of the $(\mathbf{x},n)$-sequential model conditional on a finite collection of edges, which we use to obtain the results appearing in Section \ref{indeg01}.

\section{Approximation of the beta variables}\label{sasymp}
Recall that for any fitness sequence $\mathbf{X}:=(X_i,i\geq 1)$, we assume that $X_1>-1$ and $(X_i,i\geq 2)$ are i.i.d.\ positive random variables, and $\mathbf{x}:=(x_i,i\geq 1)$ is a realisation of $\mathbf{X}$ with $X_1=x_1$. We stress that in this section, we do not assume bounded fitness, because we shall apply the lemmas to prove Theorem \ref{uniformvert}. 

To approximate $S^{(x)}_{k,n}$ and $B^{(x)}_j$, we require that with high probability, $\mathbf{x}$ is such that $\sum^j_{\ell=2} x_\ell$ is close enough to $(j-1)\mu$ for all $j$ sufficiently large. So given $1/2<\alpha<1$ and $n$ the size of the $(\mathbf{x},n)$-sequential model, we define
\begin{equation}\label{eventX}
    A_{\alpha,n} = \bigg\{\mathbf{x}: \bigcap^\infty_{j=\ceil{\phi(n)}} \bigg\{\bigg|\sum^j_{h=2} x_h-(j-1)\mu\bigg|\leq j^{\alpha} \bigg\} \bigg\},
\end{equation}
where $\phi(n)=\Omega(n^\chi)$. We adapt the techniques in \textcite{pekoz2019} to produce the following lemmas, and defer the proofs to Section \ref{pfssasymp}. Arguments of similar flavour can be found in \textcite{bloemreddy2017}, where the authors studied a different random graph. The first lemma is due to an application of standard moment inequalities.

\begin{lemma}\label{goodsum}
Assume that $\E(X^p_2)<\infty$ for some $p>2$. Given a positive integer $n$ and $1/2+1/p<\alpha<1$, there is a constant $C:=C(\mu,\alpha,p)$ such that $\Prob(\mathbf{x}\in A_{\alpha,n})\geq 1- Cn^{\chi[-p(\alpha-1/2)+1]}$. 
\end{lemma}

The lemma above indicates that for large $n$, the probability that $\mathbf{X}$ satisfies
\begin{equation*}
    \bigcap^\infty_{j=\ceil{\phi(n)}} \bigg\{\bigg|\sum^j_{h=2} X_h-(j-1)\mu\bigg|\leq j^{\alpha} \bigg\}
\end{equation*}
tends to one as $n\to\infty$. The next lemma is an approximation result of $S^{(x)}_{k,n}$, $\ceil{\phi(n)}\leq k \leq n$, and is an extension of \textcite[Lemma 3.1]{Berger12asymptoticbehavior}. Note that from now on we use the subscript $\mathbf{x}$ to indicate the conditioning on $\mathbf{X}=\mathbf{x}$. 

\begin{lemma}\label{Sasymptotic}
Given a positive integer $n$ and $1/2<\alpha<1$, assume that $\mathbf{x}\in A_{\alpha,n}$. Then there are positive constants $C:=C(x_1,\mu,\alpha)$ and  $c:=c(x_1,\mu,\alpha)$ such that
\begin{equation}\label{maxbound}
    \Prob_{\mathbf{x}}\left(\max\limits_{\ceil{\phi(n)} \leq k \leq n}\left|S^{(x)}_{k,n}-\left(\frac{k}{n} \right)^{\chi}\right|\leq \delta_n \right) \geq 1-\varepsilon_n,
\end{equation}
where $\delta_n:=Cn^{-\chi(1-\alpha)/4}$ and $\varepsilon_n:=cn^{-\chi(1-\alpha)/2}$.   
\end{lemma}

We outline the proof of Lemma \ref{Sasymptotic} as follows. The first step is to derive an expression for $\E_x[S^{(x)}_{k,n}]$, where we modify a moment formula given in \textcite{delphin2019}. In the second step, we use the formula to show that when $\mathbf{x}\in A_{\alpha, n}$, the difference between the mean of $S^{(x)}_{k,n}$ and $(k/n)^\chi$ for $k\geq \ceil{\phi(n)}$ is small enough for large $n$. Once we take care of swapping the mean of $\E_x[S^{(x)}_{k,n}]$ and $(k/n)^\chi$, the lemma then follows from a martingale argument.

\begin{remark}
For vertices of order at most $n^\chi$, the upper bound $\delta_n$ in Lemma \ref{Sasymptotic} is only meaningful when $\chi>(\alpha+3)/4$, as otherwise $\delta_n$ is of order greater than $(k/n)^\chi$. However, the bound is useful for estimating (\ref{berprob}) as the probability that $k_0=o(n)$ tends to zero as $n\to \infty$.
\end{remark}

For the approximation of $B^{(x)}_j$, recall that $B^{(x)}_i\sim \mathrm{Beta}(x_i,T_{i-1}+i-1)$ for $i\geq 2$, where $T_m:=\sum^m_{i=1}x_i$. In what follows we drop the superscript $(x)$ to simplify notations. Independently of $B^{(x)}_i$, let $\mathcal{Z}_j$ and $\mathcal{\tilde Z}_k$ be independent variables such that $\mathcal{Z}_j\sim \mathrm{Gamma}(x_j,1)$ and $\mathcal{\tilde Z}_k \sim \mathrm{Gamma}(T_k+k,1)$ for $2\leq j\leq n$ and $1\leq k\leq n-1$. Then by the beta-gamma algebra, we have the following distributional identity:
\begin{equation*}
    \left(B^{(x)}_j, \mathcal{\tilde Z}_{j-1}+\mathcal{Z}_{j}\right)=_d \left(\frac{\mathcal{Z}_{j}}{\mathcal{Z}_{j}+\mathcal{\tilde Z}_{j-1}}, \mathcal{\tilde Z}_{j-1}+\mathcal{Z}_{j}\right) \quad \text{for $2\leq j\leq n$,}
\end{equation*}
where the two random variables on the right-hand side are independent. The next lemma is due to the law of the large numbers and Chebyshev's inequality, and is an extension of \cite[Lemma 3.2]{Berger12asymptoticbehavior}.
\begin{lemma}\label{betagamma}
Given positive integer $n$ and $1/2<\alpha<3/4$, let $\mathcal{Z}_{j}$ and $\mathcal{\tilde Z}_j$ be as above. Define the event
\begin{equation}\label{epsj}
   E_{\varepsilon,j,\mathbf{x}}:= \left\{\left|\frac{\mathcal{Z}_{j}}{\mathcal{Z}_{j}+\mathcal{\tilde Z}_{j-1}}-\frac{\mathcal{Z}_{j}}{(\mu+1)j}\right|\leq \frac{\mathcal{Z}_{j}}{(\mu+1)j}\varepsilon\right\}, \quad 2\leq j\leq n.
\end{equation}
When $\mathbf{x}\in A_{\alpha,n}$, there is a positive constant $C:=C(x_1,\alpha,\mu)$ such that
\begin{equation}\label{bgc}
    \Prob_{\mathbf{x}}\bigg(\bigcap^n_{j= \ceil{\phi(n)}}E_{\varepsilon,j,{\mathbf{x}}} \bigg)\geq 1-C(1+\varepsilon)^4\varepsilon^{-4}n^{\chi(4\alpha-3)}.
\end{equation}
In addition,
\begin{equation}\label{woboundedass}
    \Prob_{\mathbf{x}}\bigg(\bigcap^n_{j= \ceil{\phi(n)}} \{\mathcal{Z}_{j}\geq j^{1/2}\}\bigg) \geq 1-\sum^n_{j= \ceil{\phi(n)}} j^{-2} \prod^{3}_{\ell=0} (x_j+\ell);
\end{equation}
and if $x_i \in (0,\kappa]$ for all $i\geq 2$, then there is a positive constant $C$ such that
\begin{equation}\label{boundedwhp}
    \Prob_{\mathbf{x}}\bigg(\bigcap^n_{j= \ceil{\phi(n)}} \{\mathcal{Z}_{j}\geq j^{1/2}\}\bigg)\geq 1- C \kappa^4 n^{-\chi}.
\end{equation}
\end{lemma}

\section{The distributions of the neighbours of the type L and R vertices}\label{indeg01}
Let $(G_n,k_0)$ be an $(\mathbf{x},n)$-P\'olya urn tree rooted at its uniformly chosen vertex $k_0$. From Section \ref{UHN}, recall that $B_r(G_n,k_0)$ is the $r$-neighbourhood of the uniform vertex $k_0$ in $(G_n,k_0)$, and a vertex in $\partial \mathcal{B}_r:= V(B_r(G_n,k_0))\setminus V(B_{r-1}(G_n,k_0))$ is of type L if it receives an incoming edge from a vertex in $\partial \mathcal{B}_{r-1}$; and it is of type R if it sends an outgoing edge to a vertex in $\partial \mathcal{B}_{r-1}$. To prepare for the graph coupling in the later sections, here we first give a characterisation of the breadth-first search of $G_n$, which shows that the unexplored neighbours of vertex $k_{\bar v}=k$ are distributed as the neighbours of vertex $k$ in the $(\mathbf{x},n)$-sequential model conditional on the set of discovered edges. We shall refer to the (finite) collection of edges that we condition on as an \textit{embellishment}. Using this relation, we construct a Bernoulli point process that encodes the type R neighbours of vertex $k_{\bar v}\in V(B_r(G_n,k_0))$, and the distribution of the type L neighbour if $k_{\bar v}$ belongs to type L.

We use the Ulam-Harris labels introduced in Section \ref{UHN} to characterise the breadth-first search as follows. The exploration process of $G_n$ is a random sequence of partitions of $V(G_n)$, denoted $(\mathcal{A}_t, \mathcal{P}_t, \mathcal{N}_t)_{t\geq 0}$, where the letters respectively stand for \textit{active, probed} and \textit{neutral}. We initialise the process with 
\begin{equation*}
    (\mathcal{A}_0, \mathcal{P}_0, \mathcal{N}_0)=(\{k_0\}, \varnothing, V(G_n)\setminus\{k_0\}).
\end{equation*}
Given $(\mathcal{A}_{t-1}, \mathcal{P}_{t-1}, \mathcal{N}_{t-1})$, $(\mathcal{A}_t, \mathcal{P}_t, \mathcal{N}_t)$ is generated as follows. Let $k[1]=k_0$ and $k[j]\in \mathbbm{N}$ be the vertex in $\mathcal{A}_{j-1}$ that is the smallest in the breadth-first order. That is, if $k_{\bar v}\in \mathcal{A}_{j-1}$ and $\bar v<_{UH} \bar u$ for all $k_{\bar u}\in \mathcal{A}_{j-1}\setminus\{k_{\bar v}\}$, then $k[j]=k_{\bar v}$. Denote by $\mathcal{D}_t$ the set of vertices in $\mathcal{N}_{t-1}$ that are attached to $k[t]$:
\begin{equation*}
    \mathcal{D}_t :=\{u\in \mathcal{N}_{t-1}:\{u,k[t]\}\in E(G_n)\}.
\end{equation*}
Then in the $t$-th exploration step we set 
\begin{equation*}
    (\mathcal{A}_t, \mathcal{P}_t, \mathcal{N}_t)=(\mathcal{A}_{t-1}\setminus\{k[t]\}\cup \mathcal{D}_t, \mathcal{P}_{t-1}\cup\{k[t]\}, \mathcal{N}_{t-1}\setminus \mathcal{D}_t);
\end{equation*}
and if $\mathcal{A}_{t-1}=\varnothing$, then set $(\mathcal{A}_t, \mathcal{P}_t, \mathcal{N}_t)=(\mathcal{A}_{t-1}, \mathcal{P}_{t-1}, \mathcal{N}_{t-1})$. In words, at step $t$ we probe vertex $k[t]$ and mark the neutral vertices attached to $k[t]$ as active. Note that this characterisation of the exploration process is standard, and more examples can be found in \cite{karp1990transitive} and \cite[Chapter 4]{van2009random}.

The alternative vertex labelling $k[t]$ is useful as the distribution of the neighbours of $k[t]$ depends on $(\mathcal{A}_{t-1}, \mathcal{P}_{t-1}, \mathcal{N}_{t-1})$; whereas we use the Ulam-Harris scheme for the breadth-first search and matching up the vertices in $(G_n,k_0)$ and in $(\mathcal{T}_{\mathbf{x},n},0)$ during the coupling. As we show next, the original vertex labels, used in the construction of $G_n$, is helpful for identifying a type L vertex in the breadth-first search. Hereafter we ignore the possibility that $k[t]=1$, because when $t$ (or equivalently the number of discovered vertices) is not too large, the probability that $k[t]=o(n)$ tends to zero as $n\to\infty$. For $t\geq 1$, let $k_s[t]$ (resp.\ $k_s^*[t]$) be the vertex in $\mathcal{P}_{t-1}$ (resp.\ $\mathcal{A}_{t-1}$) that has the smallest original vertex label, that is
\begin{equation*}
    k_s[t]:=\min\{ v: v\in  \mathcal{P}_{t-1}\}\quad\text{and}\quad k_s^*[t]:=\min\{ v: v\in  \mathcal{A}_{t-1}\}.
\end{equation*}
The lemma below says that $k[t]$ is a type L vertex if and only if $k[t]=k^*_s[t]$, which can be understood as a consequence of the fact that immediately after we probe an active type L vertex, we uncover a new active type L vertex. A pictorial example is given in Figure~\ref{fig:kst}.
\begin{figure}
    \centering
    \begin{tikzpicture}
    \draw (0,0) -- (13,0);
    \draw [fill] (0.8,0) circle [radius=0.08] node[below=3mm] {$k_{0,1,1}$};
    \draw [fill] (2.5,0) circle [radius=0.08] node[below=3mm] {$k_{0,1}$};
    \draw [fill] (4.8,0) circle [radius=0.08] node[below=3mm] {$k_{0,1,2}$};
    \draw [fill] (6.4,0) circle [radius=0.08] node[below=3mm] {$k_{0}$};
    \draw [fill] (7.6,0) circle [radius=0.08] node[below=3mm] {$k_{0,2}$};
    \draw [fill] (10,0) circle [radius=0.08] node[below=3mm] {$k_{0,3}$};
    \draw [fill] (11.3,0) circle [radius=0.08] node[below=3mm] {$k_{0,2,1}$};
    \draw [->, thick] (2.5,0) to[bend right] (0.9,0);
    \draw [->, thick] (4.8,0) to[bend right] (2.6,0);
    \draw [->, thick] (6.4,0) to[bend right] (2.6,0);
    \draw [->, thick] (7.6,0) to[bend right] (6.5,0);
    \draw [->, thick] (10,0) to[bend right] (6.5,0);
    \draw [->, thick] (11.3,0) to[bend right] (7.7,0);
\end{tikzpicture}
    \caption{\small The vertices are arranged from left to right in increasing order of their \textit{original} labels in $G_n$. Here, $\mathcal{P}_{3}=\{k_{0},k_{0,1},k_{0,2}\}$ and $\mathcal{A}_{3}=\{k_{0,3},k_{0,1,1},k_{0,1,2},k_{0,2,1}\}$. So $k[4]=k_{0,3}$, $k_s[4]=k_{0,1}$ and $k^*_s[4]=k_{0,1,1}$. }
    \label{fig:kst}
\end{figure}
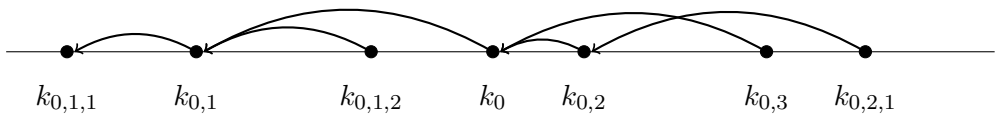

\begin{lemma}\label{explorationobs}
Assume that $\mathcal{A}_{t-1}\cup \mathcal{P}_{t-1}$ does not contain vertex 1. If $t=2$, then $k_s[2]=k_0$ and $k^*_s[2]=k_{0,1}$; while if $2<i\leq t$, $k_s[i]$ and $k^*_s[i]$ are type L vertices, where $k^*_s[i]$ is the only type L vertex in $\mathcal{A}_{i-1}$, and it receives an incoming edge from $k_s[i]$.
\end{lemma}

\begin{proof}
We prove the lemma by an induction on $2\leq i\leq t$. The base case is clear, since $\mathcal{P}_{1}=\{k_0\}$ and $\mathcal{A}_{1}=\{k_{0,1},...,k_{0,1+\theta_0}\}$. Assume that the lemma holds for some $2\leq i<t$. If we probe a type L vertex at time $i$, then $k[i]=k^*_s[i]$, and there is vertex $u\in \mathcal{N}_{i-1}$ that receives the incoming edge emanating from $k[i]$. Hence vertex $u$ belongs to type L and $k^*_s[i+1]=u$. Furthermore, $k_s[i+1]=k[i]$, as $k_s[i]$ sends an outgoing edge to $k[i]$ by assumption, implying $k[i]<k_s[i]$. If we probe a type R vertex at time $i$, then $k[i]>k^*_s[i]$ and we uncover vertices in $\mathcal{N}_{i-1}$ that have larger labels than $k[i]$, and set $\mathcal{P}_{i}=\mathcal{P}_{i-1}\cup k[i]$. It follows that $k_s[i+1]=k_s[i]$ and $k^*_s[i+1]=k^*_s[i]$, which are type L vertices.
\end{proof}

The next lemma shows that the type L vertex $k^*_s[t]$ is the only vertex in $\mathcal{A}_{t-1}$ to condition on having received at least one incoming edge, which has a biasing effect on the probability that an unexplored vertex attaches to $k^*_s[t]$. On the other hand, the edge linking a type R vertex $k \in \mathcal{A}_{t-1}$ and a vertex in $\mathcal{P}_{t-1}$ does not have a similar biasing effect, because the edge is necessarily an outgoing edge from vertex $k$, and conditioning on this edge does not alter the weight of vertex $k$.

\begin{lemma}\label{onlyedge}
Given $t\geq 2$, assume that $\mathcal{A}_{t-1}\cup \mathcal{P}_{t-1}$ does not contain vertex 1. Then the recipient of the incoming edge from vertex $v\in \mathcal{P}_{t-1}$ is also in $\mathcal{P}_{t-1}$, unless $v=k_s[t]$, in which case the recipient is $k^*_s[t]\in \mathcal{A}_{t-1}$. 
\end{lemma}

\begin{proof}
The vertex we probe at exploration time $2\leq s\leq t$, $k[s]$, either belongs to type L or type R. If it belongs to type R, then it must have been discovered at some step $i<s$ via the outgoing edge it sends to $k[i]\in \mathcal{P}_{s-1}$. If $k[s]$ is a type L vertex, then by Lemma \ref{explorationobs}, $k[s]=k_s[s+1]$, and vertex $k^*_s[s+1]\in \mathcal{A}_s$ receives the incoming edge from $k[s]$. 
\end{proof}

We proceed to construct the neighbour distribution of vertex $k[t]$. Let $\mathcal{E}_0 =\varnothing$, and for $s\geq 1$, let $\mathcal{E}_s$ be the set of edges connecting the vertices in $\mathcal{A}_s \cup \mathcal{P}_s$. Given $(\mathcal{A}_{t-1}, \mathcal{P}_{t-1},\mathcal{N}_{t-1})$, the neighbours of $k[t]=k$ are distributed as the neighbours of vertex $k$ in the embellished $(\mathbf{x},n)$-sequential model, with the edges in $\mathcal{E}_{t-1}$ being the embellishment. In Section \ref{embmodel}, we prove that the embellished model still retains the preferential attachment mechanism, and the embellished preferential attachment process can be represented as embedded P\'olya urns, hence there is an analog of Theorem~\ref{linebreaking} for the model. Here we only state the main consequence of this analog, that is, the type R neighbours of $k[t]$ can be encoded in a Bernoulli point process analogous to $(\mathbbm{1}[U_{k_0}\in I_k],k_0+1\leq k\leq n)$, where $U_{k_0}$ and $I_{k}$ are respectively the uniform variable and the intervals in Definition \ref{urnrep}; and when $k[t]$ belongs to type L, the distribution of its type L neighbour can be constructed using the beta variables appearing in Definition \ref{condpttree}. 

We need to define some variables for this purpose, which are distributed as the gammas and betas appearing in Definition \ref{condpttree}, but expressed using the notations used in the characterisation of the breadth-first search. These new variables are more convenient when we relate them to the betas in the urn representation of the embellished model. Given a positive integer $m$, denote the set of vertices and edges in $\mathcal{P}_t$ and $\mathcal{E}_{t}$ that are added to $G_n$ before vertex $m$ as
\begin{align}\label{eptm}
    \mathcal{P}_{t, m}=\{ v \in \mathcal{P}_t:v<m\}\quad\text{and}\quad
    \mathcal{E}_{t, m}=\{\{h,i\}\in \mathcal{E}_t:\max\{h,i\}<m\}.
\end{align}
Noting that $(\mathcal{A}_0, \mathcal{P}_0, \mathcal{N}_0)=(\{k_0\}, \varnothing, V(G_n)\setminus\{k_0\})$ and $\mathcal{E}_0=\varnothing$, let $((\mathcal{Z}_{i}[t], \mathcal{\tilde Z}_{i}[t]) , i \in \{2,...,n\}\setminus \mathcal{P}_{t-1})$ be a sequence of independent variables, where $[t]$ indicates the exploration step. Omitting $\mathbf{x}$ in the notations, let $\mathcal{Z}_h[1]\sim \mathrm{Gamma}(x_h,1)$ and $\mathcal{\tilde Z}_i[1]\sim \mathrm{Gamma}(T_{i-1}+i-1,1)$, where $T_m:=\sum^m_{l=1}x_l$. Assume that $j \in \{2,...,n\}\setminus \mathcal{P}_{t-1}$, we define
\begin{align}\label{newgamma}
    \mathcal{Z}_{j}[t]\sim \mathrm{Gamma}(x_{j}+\mathbbm{1}[j=k^*_s[t]],1), \quad t\geq 2,
\end{align}
and 
\begin{align*}
    \mathcal{\tilde Z}_{j}[t]\sim
    \begin{cases}
    & \mathrm{Gamma}(T_{j-1}+j-1,1),\quad \text{$2\leq j\leq k^*_s[t]$,}\\
    & \mathrm{Gamma}(T_{j-1}+j,1),\quad \text{$k^*_s[t]<j<k_{s}[t]$,}\\
    & \mathrm{Gamma}(T_{j-1}+j-\sum_{k\in \mathcal{P}_{t-1,j}}x_k - |\mathcal{E}_{t-1,j}|,1),\quad \text{$k_s[t]<j\leq n$.}\numberthis \label{newtildegamma}
    \end{cases}
\end{align*}
Define $B_1[t]:=1$, $B_k[t]:=0$ for $t\geq 1$ and $k\in \mathcal{P}_{t-1}$, as well as
\begin{equation}\label{newbeta1}
    B_j[t]:=\frac{\mathcal{Z}_j[t]}{\mathcal{Z}_j[t]+\mathcal{\tilde Z}_j[t]}\quad \text{for $j \in \{2,...,n\}\setminus \mathcal{P}_{t-1}$}.
\end{equation}
Denote $S_{0,n}[t]:=0$, $S_{n,n}[t]:=1$ and
 \begin{equation}\label{news}
     S_{k,n}[t]:=\prod^n_{j=k+1}(1-B_j[t])= \prod^n_{j=k+1;j\not \in \mathcal{P}_{t-1}}(1-B_j[t])\quad\text{for $1\leq k\leq n-1$,}
 \end{equation}
where the second equality is true because $B_j[t]=0$ for $j\in \mathcal{P}_{t-1}$. Observe that by the beta-gamma algebra, $(B_i[1], 1\leq i\leq n)=_d (B^{(x)}_j,1\leq j\leq n)$ and $(S_{i,n}[1],1\leq i\leq n)=_d (S^{(x)}_{j,n},1\leq j\leq n)$, where $B^{(x)}_j$ and $S^{(x)}_{j,n}$ are as in Definition \ref{urnrep}. Furthermore, $\mathcal{P}_{t-1}=\{k_{\bar u}:\bar u<_{UH} \bar v\}$ for $t\geq 2$ and $k_{\bar v}=k[t]$. Thus, when $\hat k_{\bar u}=k_{\bar u}$ for $k_{\bar u}\in \mathcal{P}_{t-1}\cup \mathcal{A}_{t-1}$, $(B_k[t],1\leq k\leq n)=_d(\beta_j[\bar v],1\leq j\leq n)$ and $(S_{k,n}[t], 1\leq k\leq n)=_d(\mathcal{S}_{i,n}[\bar v], 1\leq i\leq n)$, where $\beta_j[\bar v]$ and $\mathcal{S}_{i,n}[\bar v]$ are as in Definition \ref{condpttree}. 

Theorem \ref{linebreaking} implies that the $(\mathbf{x},n)$-P\'olya urn graph constructed using $(B_j[1],1\leq j\leq n)$ is distributed as Seq$(\mathbf{x})_n$. For the exploration step $u\geq 2$, the parameters of the gammas and betas can be understood as follows. The edges in $\mathcal{E}_{u-1}$ are the embellishment, and the variables $(B_j[u],1\leq j\leq n)$ are ingredients for building an urn graph that is distributed as the embellished $(\mathbf{x},n)$-sequential model. As the edges attached to $\mathcal{P}_{u-1}$ are already determined, $B_j[u]:=0$ for $j \in \mathcal{P}_{u-1}$. The shape parameter of $\mathcal{\tilde Z}_j[u]$ is the total weight of the vertices in $\mathcal{A}_{u-1}\cup \mathcal{N}_{u-1}$ with labels less than $j$. This is because when adding a new vertex to the embellished model, the recipient of its outgoing edge cannot be a vertex in $\mathcal{P}_{u-1}$, and is chosen with probability proportional to the current weights of the vertices in $\mathcal{A}_{u-1}\cup \mathcal{N}_{u-1}$ with labels less than the new vertex. The initial attractiveness of $k^*_s[u]$ is $x_{k^*_s[u]}+1$ due to the size-bias effect of $\{k_s[u], k^*_s[u]\}\in \mathcal{E}_{u-1}$, so the total weight of the vertices $\{1,...,j\}$, $k^*_s[u]<j<k_s[u]$, is $T_{j-1}+j$ instead of $T_{j-1}+j-1$; and the shape parameter of $\mathcal{Z}_{k^*_s[u]}[u]$ is $x_{k^*_s[u]}+1$, which explains the type L gamma variables in Definition~\ref{pipolyapointgraph}.

We now construct a Bernoulli point process that encodes the type R neighbours of vertex $k[t]$. For $j\in \mathcal{N}_{t-1}$ and $k[t]+1\leq j\leq n$, let $R_j[t]$ be an indicator variable that takes value one if and only if vertex $j$ sends an outgoing edge to $k[t]$; while for $j\not \in \mathcal{N}_{t-1}$, let $R_j[t]=0$ with probability one, since the recipient of the incoming edge from vertex $j$ is already in $\mathcal{P}_{t-1}$. Note that $\sum^n_{j=k[t]+1}R_j[t]$ is the number of type R neighbours of vertex $k[t]$. We also assume $\mathcal{N}_{t-1}\not =\varnothing$, because for large $n$, with high probability the local neighbourhood of vertex $k_0$ does not contain all the vertices of $G_n$. To state the distribution of $(R_j[t],k[t]+1\leq j\leq n)$, we use (\ref{newbeta1}) and (\ref{news}) to define
\begin{align*}
    P_{k\to k[t]} :=
    \begin{cases}
    \frac{S_{k[t],n}[t]}{ S_{k-1,n}[t]} B_{k[t]}[t],\quad \text{if $k \in \mathcal{N}_{t-1}$ and $k[t]+1\leq k\leq n$;}\\
    0,\qquad \text{if $k\not \in \mathcal{N}_{t-1}$ and $k[t]+1\leq k\leq n$.}\numberthis\label{tmean1}
    \end{cases}
\end{align*}

\begin{defn}\label{berproc01}
Given $(\mathcal{A}_{t-1},\mathcal{P}_{t-1},\mathcal{N}_{t-1})$ and $(B_j[t],k[t]\leq j\leq n)$, let $Y_{k\to k[t]}$, $k[t]+1\leq j\leq n$ be conditionally independent Bernoulli variables, each with parameter $P_{k\to k[t]}$. We define this Bernoulli point process by the random vector
\begin{equation*}
    \mathbf{Y}^{(k[t],n)}:=\left(Y_{(k[t]+1)\to k[t]},Y_{(k[t]+2)\to k[t]},...,Y_{n\to k[t]}\right).
\end{equation*}
\end{defn}

With the preparations above, we are ready to state the main results of this section. The proofs of the following lemmas are given in Section \ref{altdefembmodel}, which are immediate consequences of the urn representation of the embellished model established there. 

\begin{lemma}\label{condurnconseq}
Assume that $\mathcal{N}_{t-1} \not= \varnothing$ and $\mathcal{A}_{t-1}\cup\mathcal{P}_{t-1}$ does not contain vertex 1. Then given $(\mathcal{A}_{t-1},\mathcal{P}_{t-1},\mathcal{N}_{t-1})$, the random vector $(R_j[t],k[t]+1\leq j\leq n)$ is distributed as $\mathbf{Y}^{(k[t],n)}$.
\end{lemma}

\begin{remark}\label{hoodsize}
Further assume that $\mathbf{x}\in A_{\alpha,n}$, where $A_{\alpha,n}$ is the event that $\sum^k_{l=2} x_l$ is close enough to $\mu (k-1)$ for all $k=\Omega(n^\chi)$, as defined in (\ref{eventX}). When $|\mathcal{E}_{t-1}|=o(n)$ for some $t\geq 2$, $P_{k\to k[t]}$ is approximately distributed as 
\begin{equation*}
    \frac{S^{(x)}_{k[t],n}}{S^{(x)}_{k,n}}B^{(x)}_{k[t],n}
\end{equation*}
for $n$ sufficiently large. This is because the shape parameters of $\mathcal{Z}_j[t]$, $k_s[t]\leq j\leq n$, given in (\ref{newtildegamma}), are such that $n^{-1}(T_{j-1}+j-\sum_{k\in \mathcal{P}_{t-1}} x_k-|\mathcal{E}_{t-1,j}|)\to \mu+1$ as $n\to\infty$; and there are not too many $B_h[t]$ that are equal to zero almost surely. These imply that $B_j[t]$ and $S_{k,n}[t]$ are approximately distributed as $B^{(x)}_j$ and $S^{(x)}_{k,n}$ for $j\not \in \mathcal{P}_{t-1}$ and $1\leq k\leq n$. Intuitively, this says that when the local neighbourhood of $k_0$ is not too large, the type R neighbours of vertex $k[t]=k$ are approximately distributed as the vertices that send an outgoing edge to vertex $k$ in the $(\mathbf{x},n)$-sequential model.
\end{remark}

When $k[t]$ is the uniform vertex or a type L vertex, we use the betas in (\ref{newbeta1}) to obtain the distribution of the recipient of the incoming edge from $k[t]$. Observe that $(B_j[t],2\leq j\leq k[t]-1)$ does not appear in Definition \ref{berproc01}, but are required for this purpose. 

\begin{lemma}\label{tilted01}
Assume that $\mathcal{N}_{t-1} \not= \varnothing$, $\mathcal{A}_{t-1}\cup\mathcal{P}_{t-1}$ does not contain vertex 1, and $k[t]$ is either the root $k_0$ or a type L vertex. Given $(\mathcal{A}_{t-1},\mathcal{P}_{t-1}, \mathcal{N}_{t-1})$, let $S_{k,n}[t]$ be as in (\ref{news}), and $U\sim \mathrm{U}[0,S_{k[t]-1,n}{[t]}]$. Then for $1\leq h\leq k[t]-1$, the probability that vertex $h$ receives the incoming edge from $k[t]$ is given by the probability that $S_{h-1,n}[t] \leq  U < S_{h,n}[t]$.
\end{lemma}

\section{Local weak limit proof: the base case}\label{slocal}
Let $(G_n,k_0)$ be the $(\mathbf{x},n)$-P\'olya urn tree rooted at the uniform vertex $k_0$, and $(\mathcal{T}_{\mathbf{x},n},0)$ be the intermediate P\'olya point tree in Definition \ref{condpttree}. Note that $k_0=k[1]$. In this section, we couple $(G_n,k_0)$ and $(\mathcal{T}_{\mathbf{x},n},0)$ such that with high probability, $(B_1(G_n,k_0),k_0)\cong (B_1(\mathcal{T}_{\mathbf{x},n},0),0)$, and for each pair $k_{\bar v}\in V(B_1(G_n,k_0))$ and $\bar v\in V(B_1(\mathcal{T}_{\mathbf{x},n},0))$, $(k_{\bar v}/n)^\chi \approx \hat a_{\bar v}$ and $k_{\bar v}=\hat k_{\bar v}$.

We start by coupling the vertices $k_0\in V(B_1(G_n,k_0))$ and $0\in V(B_1(\mathcal{T}_{\mathbf{x},n},0))$. Let $U_0\sim \mathrm{U}[0,1]$, $\hat a_0=U^{\chi}_0$ and $k_0=\ceil{nU_0}$, where $\hat a_0$ is the age of vertex $0$. Then $\hat a_0\approx (k_0/n)^\chi$, and the initial attractiveness of vertex 0 is $x_{k_0}$ by a direct comparison to Definition \ref{condpttree}. Under this coupling we define the event
\begin{equation}\label{h0}
    \mathcal{H}_{1,0}=\{\hat a_0>(\log \log n)^{-\chi} \},
\end{equation}
which guarantees that we choose a vertex of low degree. To prepare for the coupling of the neighbours of the root vertices, let $((\mathcal{Z}_j[1],\mathcal{\tilde Z}_{j}[1]),2\leq j\leq n)$ and $(S_{k,n}[1],1\leq k\leq n)$ be as in (\ref{newgamma}), (\ref{newtildegamma}) and (\ref{news}). We use $(S_{k,n}[1],1\leq k\leq n)$ to construct the distribution of the type R neighbours of $k_0$, and for sampling the initial attractiveness of vertex $(0,1)\in V(B_1(\mathcal{T}_{\mathbf{x},n},0))$. Let $\hat a_{0,1}\sim \mathrm{U}[0,\hat a_0]$, and $(\hat a_{0,i},2\leq i\leq 1+\hat \tau_0)$ be the points of the mixed Poisson process on $(\hat a_0,1]$ with intensity $\mathcal{Z}_{k_0}[1]a^{-1/\mu}_0\mu y^{1/\mu-1}dy$. Recalling that $\theta_0$ is the number of type R neighbours of vertex $k_0$, below we define a coupling of $(G_n,k_0)$ and $(\mathcal{T}_{\mathbf{x},n},0)$, and for this coupling we define the events
\begin{equation}\label{h10}
\begin{split}
    &\mathcal{H}_{1,1} = \{\hat a_{0,1}>(\log \log n)^{-2\chi}\},\\
    &\mathcal{H}_{1,2}= \bigg\{ \theta_0= \hat \tau_0,\text{ for $\bar v\in V(B_1(\mathcal{T}_{\mathbf{x},n},0))$, $k_{\bar v}=\hat k_{\bar v}$,}
    \left|\hat a_{\bar v}-\left(\frac{k_{\bar v}}{n}\right)^\chi \right|\leq C_1n^{-\frac{\chi}{12}}(\log\log n)^{\chi}\bigg\},\\
    &\mathcal{H}_{1,3}=\{\hat \tau_0<(\log n)^{1/r}\},
\end{split}
\end{equation}
where $C_1:=C_1(x_1,\mu)$ is a constant that will be chosen in the proof of Lemma \ref{r1coupling} below. On the event $\mathcal{H}_{1,2}$, the two graphs are coupled such that $(B_1(G_n,k_0),k_0)\cong (B_1(\mathcal{T}_{\mathbf{x},n},0),0)$, and for the vertices $(0,i)\in \partial \mathfrak{B}_1:=V(B_1(\mathcal{T}_{\mathbf{x},n},0))\setminus\{0\}$ and $k_{0,i}\in\partial\mathcal{B}_1:=V(B_1(G_n,k_0))\setminus\{k_0\}$, their initial attractiveness match and their ages are close enough. Consequently, we can couple the Bernoulli point process and the mixed Poisson point process that encode the type R neighbours of vertices $k_{0,i}\in \partial\mathcal{B}_1$ and $(0,i)\in \partial \mathfrak{B}_1$, and hence coupling the $2$-neighbourhoods. The event $\mathcal{H}_{1,1}$ ensures that the neighbours of $k_0$ have low degrees, so that the size of the local neighbourhood of $k_0$ grows slowly. On the event $\mathcal{H}_{1,3}$, the number of vertex pairs that we need to couple for the $2$-neighbourhoods are not too large, so that by a union bound argument, the probability that the coupling of any of the vertex pairs fail tends to zero as $n\to\infty$.

The aim of this section is to show that when the difference between $\sum^k_{j=2}x_j$ and $(k-1)\mu$ is small enough for all $k$ sufficiently large, we can couple the two graphs such that $\bigcap^3_{i=0}\mathcal{H}_{1,i}$ occurs with high probability. Thus, let $A_n:=A_{2/3,n}$ be as in (\ref{eventX}) with $\alpha=2/3$; and we choose this $\alpha$ to simplify our calculations. The result is in the lemma below, and is the base case when we inductively prove an analogous result for a general radius $r<\infty$. Note that $B_1(\mathcal{T}_{\mathbf{x},n},0)$ is distributed as the 1-neighbourhood of the $\pi$-P\'olya point tree after randomisation of $\mathbf{X}$, and Lemma \ref{goodsum} implies $\mathbf{x}\in A_n$ with high probability. Hence, it follows from the lemma below and the coupling interpretation of the total variation distance that (\ref{maindtv}) of Theorem \ref{bigthm} holds for the 1-neighbourhoods. 

\begin{lemma}\label{r1goodevents}
Assume that $x_i\in (0,\kappa]$ for $i\geq 2$ and $\mathbf{x}\in A_n$. Let $\mathcal{H}_{1,i}$, $i=0,...,3$ be defined as in (\ref{h0}) and (\ref{h10}). Then given $r<\infty$ and $n\gg r$, there is a coupling of $(G_n,k_0)$ and $(\mathcal{T}_{\mathbf{x},n},0)$, with $p>\max\{r-1,7\}$ and a positive constant $C:=C(x_1,\mu,p,\kappa)$ such that
\begin{align}\label{withh10}
    \Prob_{\mathbf{x}}\bigg(
    \bigg(\bigcap^3_{i=0}\mathcal{H}_{1,i}\bigg)^c\bigg)\leq C (\log \log n)^{-\chi}.
\end{align}
\end{lemma}

We prove Lemma \ref{r1goodevents} in the remaining part of this section. From the definitions of $\hat a_0$ and $\hat a_{0,1}$, it is obvious that the probabilities of the events $\mathcal{H}_{1,0}$ and $\mathcal{H}_{1,1}$ tend to one as $n\to\infty$; and by Chebyshev's inequality, we can show that this is also true for the event $\mathcal{H}_{1,3}$. We take care of $\mathcal{H}_{1,2}$ in the lemma below, whose proof is the core of this section. 

\begin{lemma}\label{r1coupling}
Retaining the assumptions and notations of Lemma \ref{r1goodevents}, there is a coupling of $(G_n,k_0)$ and $(\mathcal{T}_{\mathbf{x},n},0)$, with $0<\gamma<\chi/12$ and a positive constant $C:=C(x_1,\mu,\kappa)$ such that
\begin{align*}
    \Prob_{\mathbf{x}}\left(\mathcal{H}_{1,0}\cap \mathcal{H}_{1,1} \cap \mathcal{H}^c_{1,2}\right)
    &\leq C n^{-\beta}(\log \log n)^{1-\chi},
\end{align*}
where $\beta=\min\{\chi/3-4\gamma,\gamma\}$.
\end{lemma}

Given that the vertices $k_0\in V(G_n,k_0)$ and $0\in V(\mathcal{T}_{\mathbf{x},n},0)$ are coupled, we use the strategy in Section \ref{couplingidea} to prove Lemma \ref{r1coupling} as follows. We first couple $\mathbf{Y}^{(k[1],n)}$ in Definition \ref{berproc01}, the Bernoulli point process that encodes the type R neighbours of vertex $k_0$, and a discretisation of the mixed Poisson process that encodes the ages and the PA labels $((\hat a_{0,i},\hat k_{0,i}),2\leq i\leq 1+\hat \tau_0)$. Then we couple the vertices $k_{0,1}\in \partial \mathcal{B}_1$ and $(0,1)\in \partial \mathfrak{B}_1$ such that on the event $\mathcal{H}_{1,0}\cap \mathcal{H}_{1,1}$, $\hat k_{0,1}=k_{0,1}$ and $(k_{0,1}/n)^\chi\approx \hat a_{0,1}$ with high probability. 

For the means of the discretised Poisson process, we define
\begin{equation}\label{mean3}
    \lambda^{[1]}_{k_0+1}:=\int^{\left(\frac{k_0+1}{n}\right)^\chi}_{\hat a_0} \frac{\mathcal{Z}_{k_0}[1]}{\hat a^{1/\mu}_0\mu} y^{1/\mu-1} dy\quad \text{and} \quad  
    \lambda^{[1]}_{k}:=\int^{\left(\frac{k}{n}\right)^\chi}_{\left(\frac{k-1}{n}\right)^\chi} \frac{\mathcal{Z}_{k_0}[1]}{\hat a^{1/\mu}_0\mu} y^{1/\mu-1} dy
\end{equation}
for $k_0+2\leq k \leq n$. 
\begin{defn}\label{dispoi}
Given $k_0=k[1]$, $\hat a_0$ and $\mathcal{Z}_{k_0}[1]$, let $V_{k\to k_0}$, $k_0+1\leq k\leq n$, be conditionally independent Poisson random variables, each with parameter $\lambda^{[1]}_{k}$ as in (\ref{mean3}). We define a discretised mixed Poisson process by the random vector
 \begin{equation*}
     \mathbf{V}^{(k[1],n)}:=(V_{(k_0+1)\to k_0},V_{(k_0+2)\to k_0},...,V_{n\to k_0}).
\end{equation*}
\end{defn}
Furthermore, we define the events that ensure $P_{k\to k_0}$ is close enough to $\lambda^{[1]}_k$. Let $\phi(n)=\Omega(n^\chi)$, $C^*:=C^*(x_1,\mu)$ be a positive constant such that (\ref{maxbound}) of Lemma \ref{Sasymptotic} holds with $\delta_n=C^* n^{-\frac{\chi}{12}}$, and $B_j[1]$ be as in (\ref{newbeta1}). Define the events 
\begin{equation} \label{3f}
\begin{split}
    &F_{1,1}=\left\{\max_{\ceil{\phi(n)}\leq k \leq n}\left|S_{k,n}[1] - \left(\frac{k}{n}\right)^\chi\right|\leq C^* n^{-\frac{\chi}{12}}\right\},\\
    & F_{1,2} = \bigcap^n_{j=\ceil{\phi(n)}}\left\{\left| B_{j}[1]-\frac{\mathcal{Z}_{j}[1]}{(\mu+1)j}\right|\leq \frac{\mathcal{Z}_{j}[1]n^{-\gamma}}{(\mu+1)j}\right\}, \\
    & F_{1,3} = \bigcap^n_{j=\ceil{\phi(n)}}\{\mathcal{Z}_{j}[1]\leq j^{1/2}\}.
\end{split}
\end{equation}

The next lemma is the major step towards proving Lemma \ref{r1coupling}, as it implies that we can couple the ages and the initial attractiveness of the type R neighbours of the vertices $k_0\in V(G_n,k_0)$ and $0\in V(\mathcal{T}_{\mathbf{x},n},0)$.

\begin{lemma}\label{poissonproc}
Retaining the notations and the assumption in Lemma \ref{r1coupling}, let $\mathbf{Y}^{(k[1],n)}$ and $\mathbf{V}^{(k[1],n)}$ be as in Definition \ref{berproc01} and \ref{dispoi}; and $F_{1,i}$, $i=1,2,3$ be as in (\ref{3f}). Then there is a coupling of the random vectors, and a positive constant $C:=C(x_1,\mu,\kappa)$ such that 
\begin{align*}
  \Prob_{\mathbf{x}}\bigg(\mathbf{Y}^{(k[1],n)}\not=\mathbf{V}^{(k[1],n)}, \bigcap^3_{i=1} F_{1,i} \cap \mathcal{H}_{1,0}\bigg)
  \leq C n^{-\gamma} (\log \log n)^{1-\chi}.
\end{align*}
\end{lemma}

The proof of Lemma \ref{poissonproc} is deferred to Section \ref{pfpoissonproc}, but we summarise it as follows. On the event $F_{1,1}\cap F_{1,2}\cap \mathcal{H}_{1,0}$, a little calculation shows that the Bernoulli success probability $P_{j\to k_0}$ in (\ref{tmean1}) is close enough to
\begin{equation}\label{mean2}
    \hat P_{k\to k_0}:=\left(\frac{k_0}{k}\right)^\chi \frac{\mathcal{Z}_{k_0}[1]}{(\mu+1)k_0},
\end{equation}
while $F_{1,3}$ ensures that $\hat P_{k\to k_0}\leq 1$. On the other hand, $\lambda^{[1]}_k\approx \hat P_{j\to k_0}$ as $\hat a_0\approx (k_0/n)^\chi$. Hence we use $\hat P_{k\to k_0}$ in (\ref{mean2}) as means to construct two intermediate Bernoulli and Poisson processes, and then explicitly couple the four processes using standard techniques. It is enough to consider the coupling under the event $\bigcap^3_{i=1} F_{1,i}\cap \mathcal{H}_{1,0}$, because when $\mathbf{x}\in A_{n}$, Lemma \ref{Sasymptotic} and \ref{betagamma} imply that $F_{1,1}$, $F_{1,2}$ and $F_{1,3}$ occur with high probability. 

Below we use Lemma \ref{Sasymptotic}, \ref{betagamma} and \ref{poissonproc} to prove Lemma \ref{r1coupling}.

\begin{proof}[Proof of Lemma \ref{r1coupling}]
To prove the lemma, we bound the right-hand side of
\begin{align*}
    &\Prob_{\mathbf{x}}(\mathcal{H}^c_{1,2}\cap \mathcal{H}_{1,1}\cap\mathcal{H}_{1,0}) \\
    &\leq \Prob_{\mathbf{x}}\bigg(\bigcap^3_{i=1} F_{1,i} \cap \bigcap^1_{k=0}\mathcal{H}_{1,k}\cap \mathcal{H}^c_{1,2} \bigg) + \Prob_{\mathbf{x}}\bigg(\bigg(\bigcap^3_{i=1} F_{1,i}\bigg)^c\bigg)\\
    &= \Prob_{\mathbf{x}}\bigg(\bigcap^3_{i=1} F_{1,i} \cap \bigcap^1_{k=0}\mathcal{H}_{1,k}\cap \mathcal{H}^c_{1,2} \bigg)+
   \Prob_{\mathbf{x}}\bigg(\bigcap^2_{i=1}F_{1,i}\cap F^c_{1,3} \bigg)+ \Prob_{\mathbf{x}}(F^c_{1,2}\cap F_{1,1} )+\Prob_{\mathbf{x}}( F^c_{1,1} ).
\end{align*}
under a suitable coupling of $(G_n,k_0)$ and $(\mathcal{T}_{\mathbf{x},n},0)$. We first handle the last three terms. In particular, apply (\ref{boundedwhp}) of Lemma \ref{betagamma} to obtain
\begin{align*}
   \Prob_{\mathbf{x}}\bigg(\bigcap^2_{j=1}F_{1,j}\cap F^c_{1,3} \bigg) \leq \Prob_{\mathbf{x}}( F^c_{1,3} )= 
    \Prob_{\mathbf{x}}\bigg(\bigcup^n_{j= \phi(n)}\{\mathcal{Z}_j[1]\geq j^{1/2}\}\bigg)
   \leq  C \kappa^4 n^{-\chi}, 
\end{align*}
where $C$ is a positive constant. As we assume $\mathbf{x}\in A_n$, we can apply (\ref{bgc}) of Lemma \ref{betagamma} (with $\varepsilon=n^{-\gamma}$, $0<\gamma<\chi/12$ and $\alpha=2/3$) and Lemma \ref{Sasymptotic} to deduce that 
\begin{gather*}
     \Prob_{\mathbf{x}}(F^c_{1,2}\cap F_{1,1})\leq \Prob_{\mathbf{x}}(F^c_{1,2})=\Prob_{\mathbf{x}}\bigg(\bigcup^n_{j=\ceil{\phi(n)}}\left\{\left| B_j[1]-\frac{\mathcal{Z}_j[1]}{(\mu+1)j}\right|\geq\frac{\mathcal{Z}_j[1] n^{-\gamma}}{(\mu+1)j}\right\}\bigg)
    \leq Cn^{4\gamma-\frac{\chi}{3}};\\
    \Prob_{\mathbf{x}}( F^c_{1,1} )\leq \Prob_{\mathbf{x}}\left(\max_{\ceil{\phi(n)}\leq k\leq n} \left|S_{k,n}[1]-\left(\frac{k}{n}\right)^\chi\right|\geq C^* n^{-\frac{\chi}{12}}\right)\leq c n^{-\frac{\chi}{6}},
\end{gather*}
where $C:=C(x_1,\gamma,\mu)$ and $c:=c(x_1,\mu)$ are positive constants.

To bound the first probability, we give the appropriate coupling. Let the vertices $k_0\in V(G_n,k_0)$ and $0\in V(\mathcal{T}_{\mathbf{x},n},0)$ be coupled as in the beginning of this section. Assume that they are such that the event $\mathcal{H}_{1,0}$ occurs; and the variables $((\mathcal{Z}_j[1],\mathcal{\tilde Z}_{j-1}[1]),2\leq j\leq n)$ are such that $\bigcap^3_{i=1}F_{1,i}$ holds. We first argue that under the event $\bigcap^3_{i=1}F_{1,i}\cap \mathcal{H}_{1,0}$, their type R neighbours can be coupled such that with high probability, $\theta_0=\hat \tau_0$ and $k_{0,i}=\hat k_{0,i}$ for $i=2,...,1+\hat \tau_0$. In view of Definition \ref{urnrep} and \ref{condpttree}, this follows readily from Lemma \ref{poissonproc}. It remains to prove that under this coupling, the age differences $|(k_0/n)^\chi-\hat a_0|$ and $|\hat a_{0,i}-(k_{0,i}/n)^\chi|$ are sufficiently small. Since $k_0=\ceil{n \hat a^{1/\chi}_0}$, and for $i=2,...,1+\hat \tau_0$, $k_{0,i}$ satisfies $k_{0,i}>k_0$ and $((k_{0,i}-1)/n)^\chi<\hat a_{0,i}\leq (k_{0,i}/n)^\chi$, it is enough to bound $(k/n)^\chi-((k-1)/n)^\chi$ for $k_0+1\leq k\leq n$. We use the fact that $k_0>n(\log\log n)^{-1}$ on the event $\mathcal{H}_{1,0}$, and the mean value theorem to obtain
\begin{equation}\label{Rage}
   \mathbbm{1}[\mathcal{H}_{1,0}] \left[\left(\frac{k}{n}\right)^\chi- \left(\frac{k-1}{n}\right)^\chi\right] \leq  \mathbbm{1}[\mathcal{H}_{1,0}] \frac{\chi}{n}  \left(\frac{n}{k-1}\right)^{1-\chi} \leq \frac{\chi(\log\log n)^{1-\chi}}{n}.
\end{equation}

Next, we couple the type L neighbours of $k_0\in V(G_n,k_0)$ and $0\in V(\mathcal{T}_{\mathbf{x},n},0)$, $k_{0,1}\in \partial \mathcal{B}_1$ and $(0,1)\in \partial \mathfrak{B}_1$ such that $k_{0,1}=\hat k_{0,1}$ and $(k_{0,1}/n)^\chi\approx \hat a_{0,1}$. Independently from $\hat a_0$, let $U_{0,1}\sim \mathrm{U}[0,1]$ and $\hat a_{0,1}=U_{0,1}\hat a_0$. From Definition \ref{urnrep} and \ref{condpttree}, we can define $k_{0,1}$ to satisfy
\begin{equation*}
    S_{k_{0,1}-1,n}[1]\leq U_{0,1}S_{k_{0}-1,n}[1]<  S_{k_{0,1},n}[1];
\end{equation*}
or equivalently, 
\begin{equation}\label{Lage}
    S_{k_{0,1}-1,n}[1]\leq \frac{\hat a_{0,1}}{\hat a_0}S_{k_{0}-1,n}[1]<  S_{k_{0,1},n}[1],
\end{equation}
and take $k_{0,1}=\hat k_{0,1}$. Now we show that under this coupling, $\hat a_{0,1}\approx (k_{0,1}/n)^\chi$ on the `good' event $\mathcal{H}_{1,0}\cap \mathcal{H}_{1,1}\cap F_{1,1}$. Observe that $S_{k_{0,1},n}[1]=\Omega((\log \log n)^{-3\chi})$ on the good event, because for $n$ large enough,
\begin{align*}
    S_{k_{0,1},n}[1]&\geq U_{0,1}S_{k_{0}-1,n}[1]
    \geq (\log \log n)^{-2\chi}\left[\left(\frac{k_0-1}{n}\right)^\chi-C^*n^{-\frac{\chi}{12}}\right]\\
    &\geq (\log \log n)^{-2\chi}[(\log \log n)^{-\chi} -2C^*n^{-\frac{\chi}{12}}].
\end{align*} 
Since $S_{k,n}[1]$ increases with $k$, $k_{0,1}=\Omega(n^\chi)$ on the good event, and hence $|S_{h,n}[1]-(h/n)^\chi|\leq C^* n^{-\frac{\chi}{12}}$ for $h=k_{0,1},k_{0,1}-1$. Furthermore, a little calculation shows that on the event $\mathcal{H}_{1,0}\cap F_{1,1}$, there is a constant $c:=c(x_1,\mu)$ such that
\begin{align*}
    \left|(S_{k_0-1,n}[1]/\hat a_0)-1\right| = c n^{-\frac{\chi}{12}}(\log \log n)^{\chi}.
\end{align*}
Now, swapping $(S_{k_0-1,n}[1]/\hat a_0)$, $S_{k_{0,1}-1,n}[1]$ and $S_{k_{0,1},n}[1]$ in (\ref{Lage}) for one, $(k_{0,1}/n)^\chi$ and $((k_{0,1}-1)/n)^\chi$ at the costs above, we get that there exists $\widehat C:=\widehat C(x_1,\mu)$ such that on the good event, $|\hat a_{0,1}-(k_{0,1}/n)^\chi|\leq \widehat Cn^{-\frac{\chi}{12}}(\log \log n)^{\chi}$. 

Finally, we bound $\Prob_{\mathbf{x}}\left(\bigcap^3_{i=1} F_{1,i} \cap (\bigcap^1_{k=0}\mathcal{H}_{1,k})\cap \mathcal{H}^c_{1,2} \right)$ under the coupling above. We take $C_1:=\widehat C \vee \chi$ for the event $\mathcal{H}_{1,2}$, and apply Lemma \ref{poissonproc} to obtain
\begin{align*}
    \Prob_{\mathbf{x}}\bigg(\bigcap^3_{i=1} F_{1,i} \cap \mathcal{H}^c_{1,2} \cap \mathcal{H}_{1,1}\cap\mathcal{H}_{1,0}\bigg)&= \Prob_{\mathbf{x}}\bigg(\mathbf{Y}^{(k[1],n)}\not=\mathbf{V}^{(k[1],n)}, \bigcap^3_{i=1} F_{1,i} \cap \bigcap^1_{k=0}\mathcal{H}_{1,k}\bigg)\\
    &\leq \Prob_{\mathbf{x}}\bigg(\mathbf{Y}^{(k[1],n)}\not=\mathbf{V}^{(k[1],n)}, \bigcap^3_{i=1} F_{1,i} \cap \mathcal{H}_{1,0}\bigg)\\
    &\leq C n^{-\gamma} (\log \log n)^{1-\chi},
\end{align*}
where $C:=C(x_1,\mu,\kappa)$ is a constant. This concludes the proof.
\end{proof}

Before proving Lemma \ref{r1goodevents}, we require a final result that says that under the graph coupling, $\mathcal{H}_{1,3}=\{\hat \tau _0<(\log n)^{1/r}\}$, the event that vertex $0\in V(\mathcal{T}_{\mathbf{x},n},0)$ (and hence vertex $k_0$) has a low degree occurs with high probability. Using Chebyshev's inequality, we prove the lemma in Section \ref{pfuniformd}.

\begin{lemma}\label{regularity}
Assume that $x_i\in (0,\kappa]$ for $i\geq 2$ and $\mathbf{x}\in A_n$. Given $r<\infty$ and $n$ large enough, there is a coupling of $(G_n,k_0)$ and $(\mathcal{T}_{\mathbf{x},n},0)$, with $p>\max\{r-1,7\}$ and a positive constant $C:=C(p,\kappa)$ such that 
\begin{equation*}
   \Prob_{\mathbf{x}}\bigg(\bigcap^2_{i=0} \mathcal{H}_{1,i}\cap \mathcal{H}^c_{1,3}\bigg)\leq C (\log n)^{-p/r}(\log \log n)^{p/(\mu+1)}.
\end{equation*}
\end{lemma}

We now complete the proof of Lemma \ref{r1goodevents} by applying Lemma \ref{r1coupling} and \ref{regularity}.

\begin{proof}[Proof of Lemma \ref{r1goodevents}]
The lemma follows from bounding the right-hand side of
\begin{align*}
    \Prob_{\mathbf{x}}\bigg(\bigg(\bigcap^3_{i=0} \mathcal{H}_{1,i}\bigg)^c\bigg)&=\Prob_{\mathbf{x}}(  \mathcal{H}^c_{1,0})+\sum^3_{\ell=1}\Prob_{\mathbf{x}}\bigg(\bigcap^{\ell-1}_{k=0}\mathcal{H}_{1,k}\cap \mathcal{H}^c_{1,\ell}\bigg),
\end{align*}
under the coupling described in the proof of Lemma \ref{r1coupling}. To bound $\Prob_{\mathbf{x}}( \mathcal{H}^c_{1,0})$ and $\Prob_{\mathbf{x}}(\mathcal{H}_{1,0}\cap \mathcal{H}^c_{1,1})$, recall that $U_{0}$ and $U_{0,1}$ are independent uniform variables on $[0,1]$, with $\hat a_0=U^\chi_0$ and $\hat a_{0,1}=U_{0,1} \hat a_{0}$. Hence
\begin{align*}
    \Prob_{\mathbf{x}}(\mathcal{H}_{1,0}\cap \mathcal{H}^c_{1,1})
    &=\E\left[\mathbbm{1}[ \mathcal{H}_{1,0}]\Prob_{\mathbf{x}}(U_{0,1}\leq \hat a^{-1}_0(\log \log n)^{-2\chi}|U_0)\right]\\
    &\leq \Prob_{\mathbf{x}}(U_{0,1}\leq (\log \log n)^{-\chi})\\
    &= (\log \log n)^{-\chi};
\end{align*}
and $\Prob_{\mathbf{x}}( \mathcal{H}^c_{1,0})=\Prob_{\mathbf{x}}(U_0\leq(\log \log n)^{-1})=(\log \log n)^{-1}$. The remaining probabilities can be bounded by Lemma \ref{r1coupling} and \ref{regularity}. This completes the proof.
\end{proof}

\section{Local weak limit proof: the general case}\label{secgenr}
Let $(G_n,k_0)$ be the $(\mathbf{x},n)$-P\'olya urn tree rooted at its uniformly chosen vertex $k_0$, we now couple $(G_n,k_0)$ and the intermediate P\'olya point tree $(\mathcal{T}_{\mathbf{x},n},0)$ such that with high probability, $(B_r(G_n,k_0),k_0)\cong (B_r(\mathcal{T}_{\mathbf{x},n},0),0)$, and for $k_{\bar w}\in V(B_r(G_n,k_0))$ and $\bar w\in V(B_r(\mathcal{T}_{\mathbf{x},n},0)$, $(k_{\bar w}/n)^\chi\approx\hat a_{\bar w}$ and $k_{\bar w}=\hat k_{\bar w}$. Then using the coupling result, we prove Theorem \ref{bigthm}.  

Recall that we write $L[s]=(0,1,...,1)$ and $|L[s]|=s+1$ in Section \ref{subsecapp}, so that $L[s]$ and $k_{L[s]}$ are the type L vertices in $\partial \mathfrak{B}_s:=V(B_s(\mathcal{T}_{\mathbf{x},n},0))\setminus V(B_{s-1}(\mathcal{T}_{\mathbf{x},n},0))$ and $\partial \mathcal{B}_s$. To define the coupling events, let $\chi/12:=\beta_1>\beta_2>\cdots>\beta_r>0$ so that $n^{-\beta_q}>n^{-\beta_{q-1}}(\log \log n)^{q\chi}$ for $n$ large enough. Then, we define
\begin{equation}\label{hqi}
\begin{split}
     \mathcal{H}_{q,1}&=\{\hat a_{L[q]}>(\log \log n)^{-\chi(q+1)}\},\\
    \mathcal{H}_{q,2}&= \bigg\{(B_q(G_n,k_0),k_0)\cong (B_q(\mathcal{T}_{\mathbf{x},n},0),0),\text{ and for $\bar v\in V(B_{q}(\mathcal{T}_{\mathbf{x},n},0))$, }\\
   &\hspace{5cm} k_{\bar v}=\hat k_{\bar v}\text{ and }\left|\hat a_{\bar v}-\left(\frac{k_{\bar v}}{n}\right)^\chi\right|\leq C_q n^{-\beta_q}\bigg \},\\
    \mathcal{H}_{q,3}&=\{\hat \tau_{\bar w}<(\log n)^{1/r} \text{ for $\bar w\in V(B_{q-1}(\mathcal{T}_{\mathbf{x},n},0))$}\}, 
\end{split}
\end{equation}
where $C_q:=C_q(x_1,\mu)$ is a constant that will be specified in the proof of Lemma \ref{gencoupling} later. The events $\mathcal{H}_{q,i}$, $i=1,2,3$ are analogous to the events $\mathcal{H}_{1,i}$ in (\ref{h10}). In particular, they ensure that the vertices in $\partial \mathfrak{B}_q$ and $\partial \mathcal{B}_{q}$ can be coupled in a similar way as the root vertices, and that under the coupling, $(B_{q+1}(G_n,k_0),k_0)\cong (B_{q+1}(\mathcal{T}_{\mathbf{x},n},0),0)$ with high probability.

The next lemma is the main result of this section, where we recall that $A_n:=A_{2/3,n}$, with $A_{2/3,n}$ as in (\ref{eventX}). The lemma essentially states that if we can couple $(G_n,k_0)$ and $(\mathcal{T}_{\mathbf{x},n},0)$ such that $(B_{q}(G_n,k_0),k_0)\cong (B_{q}(\mathcal{T}_{\mathbf{x},n},0),0)$ with high probability, then we can achieve this for the $(q+1)$-neighbourhoods too.  

\begin{lemma}\label{verygoodevents}
Let $\mathcal{H}_{q,i}$, $1\leq q\leq r$ and $i=1,2,3$ be as in (\ref{h10}) and (\ref{hqi}). Assume that $x_i\in (0,\kappa]$ for $i\geq 2$ and $\mathbf{x}\in A_n$. Given $r<\infty$, $1\leq q\leq r-1$ and $n>>r$, if there is a coupling of $(G_n,k_0)$ and $(\mathcal{T}_{\mathbf{x},n},0)$, and a positive constant $C:=C(x_1,\mu,\kappa,q)$ such that
\begin{align*}
    \Prob_{\mathbf{x}}\bigg(\bigg(\bigcap^{3}_{i=1}\mathcal{H}_{q,i}\bigg)^c\bigg)\leq C (\log \log n)^{-\chi},
\end{align*}
then there is a coupling of $(G_n,k_0)$ and $(\mathcal{T}_{\mathbf{x},n},0)$, and a positive constant $C':=C'(x_1,\mu,\kappa,q)$ such that 
\begin{align*}
    \Prob_{\mathbf{x}}\bigg(\bigg(\bigcap^{3}_{i=1}\mathcal{H}_{q+1,i}\bigg)^c\bigg)\leq C' (\log \log n)^{-\chi}.
\end{align*}
\end{lemma}

Since Lemma \ref{r1goodevents} implies that there is a coupling such that $\Prob_{\mathbf{x}}\left(
\left(\bigcap^3_{i=1}\mathcal{H}_{1,i}\right)^c\right)\leq C (\log \log n)^{-\chi}$, combining Lemma \ref{r1goodevents} and \ref{verygoodevents} yields the following corollary, which is the key to proving Theorem~\ref{bigthm}.

\begin{cor}\label{maincor}
Retaining the assumption and the notations in Lemma \ref{verygoodevents}, there is a coupling of $(G_n,k_0)$ and $(\mathcal{T}_{\mathbf{x},n},0)$, and a positive constant $C:=C(x_1,\mu,\kappa,r)$ such that 
\begin{align*}
    \Prob_{\mathbf{x}}\bigg(\bigg(\bigcap^{3}_{i=1}\mathcal{H}_{r,i}\bigg)^c\bigg)\leq C (\log \log n)^{-\chi}.
\end{align*}
\end{cor}

The remainder of this section is dedicated to proving Lemma \ref{verygoodevents} and Theorem \ref{bigthm}. 

\subsection{Proof of Lemma \ref{verygoodevents}}
Fix $1\leq q\leq r-1$, and we may assume that two graphs are already coupled such that the event $\bigcap^3_{i=1}\mathcal{H}_{q,i}$ has occurred. On the event $\mathcal{H}_{q,1}$, it is clear from the definitions of $\hat a_{L[q+1]}$ and $\mathcal{H}_{q+1,1}$, that $\mathcal{H}_{q+1,1}$ occurs with high probability. While on the event $\mathcal{H}_{q,1}\cap \mathcal{H}_{q,3}$, we can show that $\mathcal{H}_{q+1,3}$ occurs with high probability by applying standard inequalities. So for the most part of the proof, we handle the event $\mathcal{H}_{q+1,2}$. 

To take care of $\mathcal{H}_{q+1,2}$, we consider the vertices of $\partial \mathcal{B}_q$ and $\partial \mathfrak{B}_q$ in the breath-first order, and couple $k_{\bar v}\in \partial \mathcal{B}_q$ and $\bar v\in \partial \mathfrak{B}_q$ such that with high probability, $\theta_{\bar v}=\hat \tau_{\bar v}$, and for $1\leq i\leq \mathbbm{1}[\bar v=L[q+1]]+\hat \tau_{\bar v}$, $k_{\bar v,i}=\hat k_{\bar v,i}$ and $(k_{\bar v,i}/n)^\chi\approx\hat a_{\bar v}$. Below we only consider the type L vertices $k_{L[q]}$ and $L[q]$ in detail, because the type R case can be proved similarly. Observe that for $n$ large enough, there must be a type L vertex in $\partial \mathcal{B}_q$ on the event $\bigcap^3_{i=1}\mathcal{H}_{q,i}$. Furthermore, recall that we define $k[t]$ as the vertex probed at time $t$. Thus writing $k[\rho[q]]:=k_{L[q]}$, we have
\begin{equation*}
    \rho[1]=2\quad\text{and}\quad\rho[q]=|V(B_{q-1}(G_n,k_0))|+1\quad\text{for $2\leq q \leq r-1$.}
\end{equation*}
Noting that
\begin{align*}
    (\mathcal{A}_{\rho[q]-1},\mathcal{P}_{\rho[q]-1}, \mathcal{N}_{\rho[q]-1})=(\partial \mathcal{B}_q,V(B_{q-1}(G_n,k_0)),V(G_n)\setminus V(B_{q}(G_n,k_0))),
\end{align*}
let
\begin{align*}
    \left(\left(\mathcal{Z}_j[\rho[q]], \mathcal{\tilde Z}_{j}[\rho[q]]\right),j\in \mathcal{A}_{\rho[q]-1}\cup \mathcal{N}_{\rho[q]-1}\right)\quad \text{and} \quad \left(S_{k,n}[\rho[q]],1\leq k\leq n\right)
\end{align*}
be as in (\ref{newgamma}), (\ref{newtildegamma}) and (\ref{news}). For convenience, we also denote  $\zeta_q:=\mathcal{Z}_{k[\rho[q]]}{[\rho[q]]}$, so that $\zeta_q\sim\mathrm{Gamma}(x_{k_{L[q]}}+1,1)$. We use $(S_{k,n}[\rho[q]],1\leq k\leq n)$ to construct the neighbour distributions of vertex $k_{L[q]}$, and let $(\hat a_{L[q],j},2\leq j\leq 1+\hat \tau_{L[q]})$ be the points of the mixed Poisson process on $(\hat a_{L[q]},1]$ with intensity 
\begin{equation}\label{coupleppp}
   \frac{\zeta_q}{\mu \hat a^{1/\mu}_{L[q]}} y^{1/\mu-1} dy,
\end{equation}
and $\hat a_{L[q],1} \sim \mathrm{U}[0,\hat a_{L[q]}]$. We define a coupling of the vertices $k_{L[q]}\in \partial \mathcal{B}_q$ and $L[q]\in \partial \mathfrak{B}_q$ in the lemma below, and for this coupling we define the events
\begin{equation}\label{kr1}
    \begin{split}
    \mathcal{K}_{\rho[q],1}&= \bigg\{\theta_{L[q]}=\hat \tau_{L[q]},\text{ and for $1\leq i\leq 1+\hat \tau_{L[q]}$, }\hat k_{L[q],i}=k_{L[q],i}\\
    &\hspace{4cm}\text{and }
    \left|\hat a_{L[q],i}-\left(\frac{k_{L[q],i}}{n}\right)^\chi\right| \leq C_{q+1} n^{-\beta_{q+1}}\bigg\}, \\
     \mathcal{K}_{\rho[q],2}&=\{\hat \tau_{L[q]}<(\log n)^{1/r}\},
    \end{split}
\end{equation}
where $C_{q+1}$ and $\beta_{q+1}$ are as in the event $\mathcal{H}_{q+1,2}$. In the sequel, we develop lemmas analogous to Lemma \ref{r1coupling} and \ref{regularity} to show that on the event $\bigcap^3_{i=1} \mathcal{H}_{q,i}$, there is a coupling such that $\mathcal{K}_{\rho[q],1}$ and $\mathcal{K}_{\rho[q],2}$ occur with high probability. As in the $1$-neighbourhood case, the difficult part is proving the claim for $\mathcal{K}_{\rho[q],1}$.

\begin{lemma}\label{gencoupling}
Retaining the notations and the assumption in Lemma \ref{verygoodevents}, let $\beta_q$ and $\mathcal{K}_{\rho[q],1}$ be as in (\ref{hqi}) and (\ref{kr1}). Then there is a coupling of $(G_n,k_0)$ and $(\mathcal{T}_{\mathbf{x},n},0)$, with $0<\gamma<\chi/12$ and a positive constant $C:=C(x_1,\mu,\kappa,q)$ such that
\begin{align*}
   \Prob_{\mathbf{x}}\bigg(\bigcap^3_{i=1} \mathcal{H}_{q,i} \cap \mathcal{H}_{q+1,1}\cap \mathcal{K}^c_{\rho[q],1} \bigg)\leq C n^{-d}(\log\log n)^{q+1}(\log n)^{q/r},
\end{align*}
where $d:=\min\{\chi/3-4\gamma,\gamma, 1-\chi,\beta_q\}$.
\end{lemma}

Similar to Lemma \ref{r1coupling}, the critical step to proving Lemma \ref{gencoupling} is showing that we can couple the type R neighbours of vertex $k_{L[q]}$. For this purpose, let the Bernoulli random vector $\mathbf{Y}^{(k[\rho[q]],n)}$ be as in Definition \ref{berproc01}, which by Lemma \ref{condurnconseq}, encodes the type R neighbours of vertex $k_{L[q]}$. Additionally, define
\begin{align}\label{closestvert}
    M_{L[q]}:=\min\{\ell\in \mathbbm{N}: (\ell/n)^\chi\geq \hat a_{L[q]}\}.
\end{align}
In particular, we use the bins $(\hat a_{L[q]}, ((\ell/n)^\chi)^n_{\ell=M_{L[q]}})$ to construct the mixed Poisson process that encodes the ages and the PA labels $((\hat a_{L[q],j},\hat k_{L[q],j}),2\leq j\leq 1+\hat \tau_{L[q]})$. However, it is possible that $M_{L[q]}\not = k_{L[q]}+1$, and in that case the number of Bernoulli and Poisson variables is not the same. Hence, we modify $\mathbf{Y}^{(k[\rho[q]],n)}$ to match these numbers. If $M_{L[q]}\leq k_{L[q]}$, define $Y_{k\to k_{L[q]}}$, $M_{L[q]}\leq k \leq k_{L[q]}$ as `Bernoulli variables' with means $P_{k\to k_{L[q]}}:=0$, and concatenate the vectors $(Y_{k\to k_{L[q]}},M_{L[q]}\leq k \leq k_{L[q]})$ and $\mathbf{Y}^{(k[\rho[q]],n)}$. This corresponds to the fact that vertex $k$ cannot send an outgoing edge to vertex $k_{L[q]}$ in $(G_n,k_0)$. If $M_{L[q]}\geq k_{L[q]}+1$, let $\mathbf{Y}^{(k[\rho[q]],n)}$ be as in Definition \ref{berproc01}. Saving notations, we redefine
\begin{align}\label{adjberpoi}
     \mathbf{Y}^{(k[\rho[q]],n)}:= \left(Y_{k^*_{L[q]}\to k_{L[q]}}, Y_{(k^*_{L[q]}+1)\to k_{L[q]}}, ..., Y_{n\to k_{L[q]}}\right), 
\end{align}
where
\begin{equation}\label{nextvert}
    k^*_{L[q]}:=\min\{M_{L[q]}, k_{L[q]}+1\}.
\end{equation}
When $M_{L[q]}\leq k_{L[q]}+1$, define the means of the discretised mixed Poisson process:
\begin{align}
    \lambda^{[\rho[q]]}_{M_{L[q]}}:=\int^{\left(\frac{M_{L[q]}}{n}\right)^\chi}_{\hat a_{L[q]}}\frac{\zeta_q}{\mu \hat a^{1/\mu}_{L[q]}} y^{1/\mu-1} dy\quad\text{and}\quad
    \lambda^{[\rho[q]]}_{k}:=\int^{\left(\frac{k}{n}\right)^\chi}_{\left(\frac{k-1}{n}\right)^\chi}\frac{\zeta_q}{\mu \hat a^{1/\mu}_{L[q]}} y^{1/\mu-1}dy \numberthis\label{tmean2}
\end{align}
for $M_{L[q]}+1\leq k\leq n$; whereas when $M_{L[q]}\geq k_{L[q]}+2$, we let
\begin{equation}\label{tmeanex}
    \lambda^{[\rho[q]]}_{j}:=0\quad\text{for $k_{L[q]}+1\leq k\leq M_{L[q]}$}
\end{equation} 
in addition to (\ref{tmean2}), so that the upcoming random vector is a discretisation of the mixed Poisson point process on $(\hat a_{L[q]},1]$.

\begin{defn}\label{pp01}
Given $k_{L[q]}$, $\hat a_{L[q]}$ and $\zeta_q$, let $k^*_{L[q]}$ be as in (\ref{nextvert}); and for $k^*_{L[q]}+1\leq k\leq n$, let $V_{k\to k[\rho[q]]}$ be conditionally independent Poisson random variables, each with parameters $\lambda^{[\rho[q]]}_k$ given in (\ref{tmean2}) and (\ref{tmeanex}). We define this discretised mixed Poisson point process with the random vector
\begin{equation*}
    \mathbf{V}^{(k[\rho[q]],n)}:=\left(V_{k^*_{L[q]}\to k_{L[q]}},V_{(k^*_{L[q]}+1)\to k_{L[q]}},...,V_{n\to k_{L[q]}}\right).
\end{equation*}
\end{defn}

We proceed to define the events analogous to $F_{1,i}$, $i=1,2,3$ in (\ref{3f}). These events ensure that $P_{k\to k[\rho[q]]}$ in (\ref{tmean1}) and $\lambda^{[\rho[q]]}_k$ are close enough for most $k\in [\max\{M_{L[q]},k_{L[q]}+1\},n]$, so that we can couple the two point processes. Let $\phi(n)=\Omega(n^\chi)$, $0<\gamma<\chi/12$, $B_{j}{[\rho[q]]}$ be as in (\ref{newbeta1}) and $C^*_q:=C^*_q(x_1,\mu)$ be a positive constant that we specify later. Define
\begin{align*}
    & F_{\rho[q],1}=\left\{\max_{\ceil{\phi(n)}\leq k\leq n}\left|S_{k,n}{[\rho[q]]}-\left(\frac{k}{n}\right)^\chi\right|\leq C^*_q n^{-\frac{\chi}{12}}\right\},\\
    & F_{\rho[q],2}=\bigcap^{n}_{\substack{j=\ceil{\phi(n)};j\not \in\mathcal{P}_{\rho[q]-1}}}\left\{\left|B_{j}{[\rho[q]]}-\frac{\mathcal{Z}_{j}{[\rho[q]]}}{(\mu+1)j}\right|\leq \frac{\mathcal{Z}_{j}{[\rho[q]]}n^{-\gamma}}{(\mu+1)j}\right\},\\
    & F_{\rho[q],3} = \bigcap^n_{j=\ceil{\phi(n)};j\not \in\mathcal{P}_{\rho[q]-1} }\{\mathcal{Z}_{j}{[\rho[q]]}\leq j^{1/2}\}. \numberthis \label{3fL}
\end{align*}

The following analog of Lemma \ref{poissonproc} is the main ingredient for proving Lemma \ref{gencoupling}.

\begin{lemma}\label{extnhood}
Retaining the notations and the assumption in Lemma \ref{gencoupling}, let $\mathbf{ Y}^{(k[\rho[q]],n)}$, $\mathbf{V}^{(k[\rho[q]],n)}$ and $F_{\rho[q],i}$, $i=1,2,3$ be as in Definition \ref{berproc01}, \ref{pp01} and (\ref{3fL}). Then there is a coupling of the random vectors, and a positive constant $C=C(x_1,\mu,\kappa,q)$ such that 
\begin{align*}
     \Prob_{\mathbf{x}}\bigg(\mathbf{Y}^{(k[\rho[q]],n)}\not = \mathbf{V}^{(k[\rho[q]],n)}, \bigcap^3_{i=1} F_{\rho[q], i}\cap \bigcap^3_{k=1} \mathcal{H}_{q,k}\bigg)
     \leq C n^{-d}(\log \log n)^{q+1}(\log n)^{q/r},
\end{align*}
where $d:=\min\{\beta_q,\gamma,1-\chi \}$.
\end{lemma}

The proof of Lemma \ref{extnhood} is in Section \ref{pfextnhood} and we summarise the main steps here. As in the proof of Lemma \ref{poissonproc}, we construct a Bernoulli and a discretised Poisson processes using suitable means $\hat P_{i\to k[\rho[q]]}$, $k^*_{L[q]}\leq i\leq n$, and then couple the four processes. However, this time we need to handle the cases where  we couple a Bernoulli variable with mean zero and a Poisson variable with positive mean, or vice versa. We take care of these cases by choosing the appropriate $\hat P_{i\to k[\rho[q]]}$. Firstly, we observe that on the event $(\bigcap^3_{i=1} F_{\rho[q], i})\cap(\bigcap^3_{k=1} \mathcal{H}_{q,k})$, $P_{k\to k[\rho[q]]}$ is close enough to
\begin{equation}\label{tmean3}
    \hat P_{k\to k[\rho[q]]}:=\left(\frac{k[\rho[q]]}{k}\right)^\chi \frac{\zeta_q}{(\mu+1)k[\rho[q]]}\quad \leq 1
\end{equation}
for $k_{L[q]}+1\leq k\leq  n$ and $k\not \in V(B_q(G_n,k_0))$. On the other hand, $\lambda^{[\rho[q]]}_{k}\approx \hat P_{k\to k[\rho[q]]}$ for $\max\{M_{L[q]},k_{L[q]}+1\}\leq k\leq n$ because $\hat a_{L[q]}\approx (k[\rho[q]]/n)^\chi$ on the event $\mathcal{H}_{q,2}$. Hence for $\max\{M_{L[q]},k_{L[q]}+1\}\leq k\leq n$ where $k \not \in V(B_q(G_n,k_0))$, we can couple $Y_{k\to k[\rho[q]]}$ and $V_{k\to k[\rho[q]]}$ the same way as in Lemma \ref{poissonproc}.

When $k \in V(B_q(G_n,k_0))$, the Bernoulli variable $Y_{k\to k[\rho[q]]}$ has mean zero, and we couple it and a Bernoulli variable with mean (\ref{tmean3}), and then argue similarly as in Lemma \ref{poissonproc}. To bound the probability that $Y_{k\to k[\rho[q]]}\not =V_{k\to k[\rho[q]]}$ for any such $k$, we first use the fact that there is at most one type L vertex in $\partial \mathcal{B}_j$, and iterate over the radii $j=1,...,q$ to deduce that on the event $\mathcal{H}_{q,3}$, 
\begin{equation}\label{qlogvertices}
    \begin{split}
        &|\partial \mathcal{B}_{j}|<1+\sum^{j}_{h=1}(\log n)^{h/r}<1+j(\log n)^{j/r},\quad \text{$1\leq j\leq q$,}\\
        &|V(B_q(G_n,k_0))|< 1+\sum^q_{j=1}[1+j(\log n)^{j/r}]< 1+q+ q^2(\log n)^{q/r}.
    \end{split}
\end{equation} 
Because $\hat P_{k\to k[\rho[q]]}$ are sufficiently small and there are at most $O((\log n)^{q/r})$ such pairs of Bernoulli variables, we can use a union bound to show that this probability tends to zero as $n\to\infty$.

Let $h\in[k^*_{L[q]}, \max\{M_{L[q]},k_{L[q]}+1\}]$. We now consider the coupling of $Y_{h\to k[\rho[q]]}$ and $V_{h\to k[\rho[q]]}$. In Table \ref{tableofmeans} below, we give the possible combinations of the Bernoulli and Poisson means, and our choice of intermediate means. As indicated in the table, we choose $\hat P_{h\to k[\rho[q]]}>0$ whenever $P_{h\to k[\rho[q]]}>0$. When $M_{L[q]}\leq k_{L[q]}$, we couple $V_{h\to k[\rho[q]]}$ and a `Poisson variable' with mean zero; while when $M_{L[q]}\geq k_{L[q]}+2$, $V_{h\to k[\rho[q]]}:=0$ by construction, and it is coupled with a Poisson variable with mean $\hat P_{h\to k[\rho[q]]}>0$. However, the number of these pairs is small because $M_{L[q]}\approx k_{L[q]}+1$ when $\hat a_{L[q]}\approx (k_{L[q]}/n)^\chi$. Consequently, a union bound argument shows that the probability that any of these couplings fail tends to zero as $n\to\infty$.  

\begin{table}[ht]
\centering
\begin{tabular}{|l|l|l|l|}
\hline
 & $M_{L[q]}\leq k_{L[q]}$ & $M_{L[q]}\geq k_{L[q]}+2$ & $M_{L[q]}= k_{L[q]}+1$ \\ \hline
$P_{h\to k[\rho[q]]}$ & 0 & as in (\ref{tmean1}) & as in (\ref{tmean1}) \\ \hline
$\lambda^{[\rho[q]]}_h$ & as in (\ref{tmean2}) & 0 & as in (\ref{tmean2}) \\ \hline
$\hat P_{h\to k[\rho[q]]}$ & 0 & as in (\ref{tmean3}) & as in (\ref{tmean3})\\ \hline
\end{tabular}
\caption{\small Combinations of means for $k^*_{L[q]}\leq h \leq \max\{M_{L[q]},k_{L[q]}+1\}$. Note that $h=k_{L[q]}+1$ when $M_{L[q]}=k_{L[q]}+1$, and in that case, the coupling is similar to that of Lemma \ref{poissonproc}.}\label{tableofmeans}
\end{table}
 
We are now ready to prove Lemma \ref{gencoupling}.

\begin{proof}[Proof of Lemma \ref{gencoupling}]
We bound the right-hand side of 
\begin{align*}
    &\Prob_{\mathbf{x}}\bigg(\bigcap^3_{i=1} \mathcal{H}_{q,i} \cap \mathcal{H}_{q+1,1}\cap  \mathcal{K}^c_{\rho[q],1} \bigg)\\
    &\leq \Prob_{\mathbf{x}}\bigg(\bigcap^3_{i=1} \mathcal{H}_{q,i} \cap \bigcap^3_{k=1} F_{\rho[q],k}\cap \mathcal{H}_{q+1,1}\cap \mathcal{K}^c_{\rho[q],1} \bigg)+\Prob_{\mathbf{x}}\bigg( \bigcap^3_{i=1} \mathcal{H}_{q,i}\cap\bigg(\bigcap^3_{k=1} F_{\rho[q],k}\bigg)^c \bigg)\\
    &= \Prob_{\mathbf{x}}\bigg(\bigcap^3_{i=1} \mathcal{H}_{q,i} \cap \bigcap^3_{k=1} F_{\rho[q],k}\cap  \mathcal{H}_{q+1,1} \cap \mathcal{K}^c_{\rho[q],1} \bigg) + \sum^3_{k=1} \Prob_{\mathbf{x}}\bigg( \bigcap^3_{i=1} \mathcal{H}_{q,i}\cap \bigcap^{k-1}_{j=1} F_{\rho[q],j} \cap F^c_{\rho[q],k}\bigg),
\end{align*}
under an appropriate coupling. We first take care of the last three terms. Using the observations in Remark \ref{hoodsize} and (\ref{qlogvertices}), we can argue much the same way as for Lemma \ref{Sasymptotic} and \ref{betagamma} to show that there are constants $C:=C(q)$, $C'=C'(x_1,\gamma,\mu,q)$, $C^*_q=C^*_q(x_1,\mu)$ and $c^*=c^*(x_1,\mu,q)$ such that 
\begin{gather*}
    \Prob_{\mathbf{x}}\bigg(\bigcap^3_{i=1} \mathcal{H}_{q,i}\cap \bigcap^2_{k=1} F_{\rho[q],k} \cap F^c_{\rho[q],3}\bigg)\leq \Prob_{\mathbf{x}}\bigg( \bigcap^3_{i=1} \mathcal{H}_{q,i} \cap F^c_{\rho[q],3}\bigg)\leq C \kappa^4 n^{-\chi},\\
    \Prob_{\mathbf{x}}\bigg( \bigcap^3_{i=1} \mathcal{H}_{q,i}\cap F_{\rho[q],1} \cap F^c_{\rho[q],2}\bigg)\leq \Prob_{\mathbf{x}}\bigg( \bigcap^3_{i=1} \mathcal{H}_{q,i} \cap F^c_{\rho[q],2}\bigg)\leq C'n^{4\gamma-\tfrac{\chi}{3}},\\
    \Prob_{\mathbf{x}}\bigg( \bigcap^3_{i=1} \mathcal{H}_{q,i}\cap F^c_{\rho[q],1}\bigg)\leq c^* n^{-\tfrac{\chi}{6}}.
\end{gather*}

To bound the first probability, we give the appropriate coupling of the vertices $k_{L[q]}\in \partial \mathcal{B}_q$ and $L[q]\in \partial \mathfrak{B}_q$, starting with their type R neighbours. Assume that
\begin{gather*}
    ((k_{\bar z}, \hat k_{\bar z}, \hat a_{\bar z}), \bar z\in V(B_q(\mathcal{T}_{\mathbf{x},n},0))),\quad
    ((\theta_{\bar x}, \hat \tau_{\bar x}),\bar x \in V(B_{q-1}(\mathcal{T}_{\mathbf{x},n},0))),\\
    \left(\left(\mathcal{Z}_j[\rho[q]], \mathcal{\tilde Z}_{j}[\rho[q]]\right),j\in \mathcal{A}_{\rho[q]-1}\cup \mathcal{N}_{\rho[q]-1}\right)
\end{gather*}
are coupled such that $E_q:=\bigcap^3_{i=1}\mathcal{H}_{q,i}\cap \bigcap^3_{i=1}F_{\rho[q],i}$ occurs. We first show that on the event $E_q$, there is a coupling such that $\theta_{L[q]}=\hat \tau_{L[q]}$ and $k_{L[q],i}=\hat k_{L[q],i}$ for $2\leq i\leq 1+\hat \tau_{L[q]}$ with high probability. In view of Definition \ref{condpttree} and \ref{berproc01}, this follows readily from Lemma \ref{extnhood}. We now prove that under this coupling, the ages $(k_{L[q],i}/n)^\chi\approx \hat a_{L[q],i}$. Note that when $\mathbf{Y}^{(k[\rho[q]],n)}$ and $\mathbf{V}^{(k[\rho[q]],n)}$ are coupled, $k_{L[q],i}\geq M_{L[q]}$. Thus using $M_{L[q]}\geq n\hat a^{1/\chi}_{L[q]}$, a little calculation shows that 
\begin{align*}
    \mathbbm{1}\bigg[\bigcap^2_{i=1}\mathcal{H}_{q,i}\bigg] \left[\left(\frac{k}{n}\right)^\chi-\left(\frac{k-1}{n}\right)^\chi\right] \leq \frac{\widehat c(\log\log n)^{(1+q)(1-\chi)}}{n},\quad M_{L[q]}\leq k\leq n,
\end{align*}
where $\widehat c:=\widehat c(\mu)$ is constant. 

We proceed to couple the type L neighbours of $k_{L[q]}\in \partial \mathcal{B}_q$ and $L[q]\in \partial \mathfrak{B}_q$ on the event $E_q\cap \mathcal{H}_{q+1}$. Independently from all the variables that have been generated so far, let $U_{L[q],1}\sim\mathrm{U}[0,1]$. Set $\hat a_{L[q],1}:=\hat a_{L[q+1]}=U_{L[q],1} \hat a_{L[q]}$, so that $\hat a_{L[q],1}$ is the age of vertex $(L[q],1)\in \partial \mathfrak{B}_{q+1}$. We define $k_{L[q],1}$ to satisfy
\begin{align}\label{Lcoup}
    S_{k_{L[q],1}-1,n}[\rho[q]]\leq U_{L[q],1} S_{k_{L[q]}-1,n}[\rho[q]]< S_{k_{L[q],1},n}[\rho[q]].
\end{align}
Hence, in light of Definition \ref{condpttree} and Lemma \ref{tilted01}, we have $\hat k_{L[q],1}=k_{L[q],1}$. To show that $\hat a_{L[q],1}\approx (k_{L[q],1}/n)^\chi$, we first argue that we can substitute $S_{k,n}[\rho[q]]$ in (\ref{Lcoup}) with $(k/n)^\chi$ at a small enough cost. A straightforward computation shows that on the event $E_q \cap \mathcal{H}_{q+1,1}$, 
\begin{align*}
    U_{L[q],1}\geq (\log\log n)^{-(q+2)\chi}\quad\text{and}\quad k_{L[q]}\geq n(\log\log n)^{-(q+1)} - C_q n^{1-\beta_q/\chi},
\end{align*}
where $C_q$ is the constant in $\mathcal{H}_{q,2}$. Consequently, there is a constant $C:=C(x_1,\mu,q)$ such that
\begin{align*}
    S_{k_{L[q],1},n}[\rho[q]]&\geq U_{L[q],1} S_{k_{L[q]}-1,n}[\rho[q]]
    \geq (\log\log n)^{-(q+2)\chi} \left[\left(\frac{k_{L[q]}-1}{n}\right)^\chi - C^*_q n^{-\frac{\chi}{12}}
    \right]\\
    &\geq C (\log\log n)^{-(2q+3)\chi} \numberthis \label{cost1}
\end{align*}
on the event $E_q\cap \mathcal{H}_{q+1,1}$. This implies that $k_{L[q],1}=\Omega(n^\chi)$, and so $|S_{h,n}[\rho[q]]-(h/n)^\chi|\leq C^*_q n^{-\chi/12}$ for $h=k_{L[q],1}, k_{L[q],1}-1$. Additionally, a direct calculation yields
\begin{align*}
    \left|\left(S_{k_{L[q]}-1,n}[\rho[q]]/\hat a_{L[q]}\right)-1\right|=Cn^{-\beta_q}(\log \log n)^{(q+1)\chi},
\end{align*}
 where $C:=C(x_1,\mu,q)$ is a constant. Replacing $(S_{k_{L[q]},n}[\rho[q]]/\hat a_{L[q]})$, $S_{k_{L[q],1},n}[\rho[q]]$ and $S_{k_{L[q],1}-1,n}[\rho[q]]$ in (\ref{Lcoup}) for one, $(k_{L[q],1}/n)^\chi$ and $((k_{L[q],1}-1)/n)^\chi$ at the costs above, then recalling $\beta_{q+1}<\beta_q$, we conclude that, on the event $E_q\cap \mathcal{H}_{q+1,1}$, there is a constant $\widehat C:=\widehat C(x_1,\mu,q)$ such that $|\hat a_{L[q],1}-(k_{L[q],1}/n)^\chi|\leq \widehat Cn^{-\beta_{q+1}}$.

The proof is complete once we bound $\Prob_{\mathbf{x}}\left(\bigcap^3_{i=1} \mathcal{H}_{q,i} \cap \bigcap^3_{k=1} F_{\rho[q],k}\cap  \mathcal{H}_{q+1,1} \cap \mathcal{K}^c_{\rho[q],1} \right)$ under the coupling above. For the event $\mathcal{K}_{\rho[q],1}$ (and $\mathcal{H}_{q+1,2}$), pick $C_{q+1}:=\widehat C\vee \widehat c$. By Lemma \ref{extnhood}, there is a positive constant $C:=C(x_1,\mu,\kappa,q)$ such that
\begin{align*}
    &\Prob_{\mathbf{x}}\bigg(\bigcap^3_{i=1} \mathcal{H}_{q,i} \cap \bigcap^3_{k=1} F_{\rho[q],k}\cap \mathcal{H}_{q+1,1}\cap \mathcal{K}^c_{\rho[q],1} \bigg)\\
    &\qquad \leq \Prob_{\mathbf{x}}\bigg(\mathbf{Y}^{(k[\rho[q]],n)}\not = \mathbf{V}^{(k[\rho[q]],n)}, \bigcap^3_{i=1} F_{\rho[q], i}\cap \bigcap^3_{k=1} \mathcal{H}_{q,k}\bigg)\\
    &\qquad \leq  C n^{-d}(\log \log n)^{q+1}(\log n)^{q/r},
\end{align*}
where $d=\min\{\beta_q,\gamma,1-\chi\}$. 
\end{proof}

The following analog of Lemma \ref{regularity} shows that under the graph coupling, $\mathcal{K}_{q,2}=\{\hat \tau_{L[q]}<(\log n)^{1/r}\}$ occurs with high probability. We omit the proof as it is similar to that of Lemma \ref{regularity}.

\begin{lemma}\label{regt}
Retaining the notations and the assumption in Lemma \ref{verygoodevents}, let $\mathcal{K}_{\rho[q],i}$, $i=1,2$ be as in (\ref{kr1}). Then given $r<\infty$ and $n$ large enough, there is a coupling of $(G_n,k_0)$ and $(\mathcal{T}_{\mathbf{x},n},0)$,  $p>\max\{r-1,7\}$ and a positive constant $C:=C(\kappa,p)$ such that
\begin{equation*}
    \Prob_{\mathbf{x}}\left(\bigcap^3_{i=1}\mathcal{H}_{q,i}\cap  \mathcal{K}_{q,1}\cap\mathcal{K}^c_{q,2}\right)\leq C (\log n)^{-p/r}(\log \log n)^{p(q+1)/(\mu+1)}.
\end{equation*}
\end{lemma}

In the following we apply Lemma \ref{gencoupling} and \ref{regt} to prove Lemma \ref{verygoodevents}. 

\begin{proof}[Proof of Lemma \ref{verygoodevents}]
We may assume that there is at least one type R vertex in $\partial \mathcal{B}_q$, as otherwise the proof is similar to that of Lemma \ref{r1goodevents}. We begin by stating the type R analogs of the events $\mathcal{K}_{\rho[q],1}$ and $\mathcal{K}_{\rho[q],2}$, and the results parallel to Lemma \ref{gencoupling} and \ref{regt}, as they will be useful for bounding $\Prob_{\mathbf{x}}\left(\left(\bigcap^{3}_{i=1}\mathcal{H}_{q+1,i}\right)^c\right)$ later.

Recursively, if $(G_n,k_0)$ and $(\mathcal{T}_{\mathbf{x},n},0)$ are already coupled such that for some exploration step $t>\rho[q]$ and $k[t]\in \partial \mathcal{B}_q$, the event 
\begin{align*}
   \mathcal{H}_{q+1,1}\cap \bigg(\bigcap^3_{i=1}\mathcal{H}_{q,i}\bigg)\cap \bigcap_{\{u<t:k[u]\in \partial \mathcal{B}_q\}} (\mathcal{K}_{u,1}\cap\mathcal{K}_{u,2}) =:\mathcal{H}_{q+1,1}\cap \mathcal{J}_{q,t}
\end{align*}
has occurred, let $((\mathcal{Z}_j[t], \mathcal{\tilde Z}_{j-1}[t]),j \in \mathcal{A}_{t-1}\cup  \mathcal{N}_{t-1})$ and $\left(S_{k,n}[t],1\leq k\leq n\right)$ be as in (\ref{newgamma}), (\ref{newtildegamma}) and (\ref{news}). We use these variables to construct the distribution of the type R neighbours of $k[t]=k_{\bar v}$, and to generate the points of the mixed Poisson process with intensity (\ref{coupleppp}), where $\zeta_q$ and $\hat a_{L[q]}$ are now replaced with $\mathcal{Z}_{k_{\bar v}}[t]\sim \mathrm{Gamma}(x_{k_{\bar v}},1)$ and $\hat a_{\bar v}$. Denote these points by $(\hat a_{\bar v,i},1\leq i\leq \hat \tau_{\bar v})$. We define a coupling for the pair $k_{\bar v}\in \partial \mathcal{B}_q$ and $\bar v\in \partial \mathfrak{B}_q$, and for the coupling we define the events analogous to $\mathcal{K}_{\rho[q],1}$ and $\mathcal{K}_{\rho[q],2}$:
\begin{align*}
    \mathcal{K}_{t,1}&= \bigg\{\theta_{\bar v}=\hat \tau_{\bar v}, \text{and for $1\leq i\leq \hat \tau_{\bar v}$, }\left|\hat a_{\bar v,i}-\left(\frac{k_{\bar v,i}}{n}\right)^\chi\right| \leq C_{q+1}n^{-\beta_{q+1}}\text{ and }
    \hat k_{\bar v,i} = k_{\bar v,i} \bigg\},\\
     \mathcal{K}_{t,2}&=\{\hat \tau_{\bar v}<(\log n)^{1/r}\},
\end{align*}
where $C_{q+1}$ is the constant in $\mathcal{H}_{q+1,2}$. On the event $\mathcal{J}_{q,t}$, (\ref{qlogvertices}) implies that the number of discovered vertices up to exploration time $t$ can be bounded as
\begin{align*}
    |\mathcal{A}_{t-1}\cup \mathcal{P}_{t-1}|\leq 2+q+(q+1)^2(\log n)^{(q+1)/r}.
\end{align*}
Hence, proceeding similarly as for Lemma \ref{gencoupling} and \ref{extnhood}, we get that there is a coupling of $(G_n,k_0)$ and $(\mathcal{T}_{\mathbf{x},n},0)$, and positive constants $C:=C(x_1,\mu,\kappa,q)$ and $c:=c(\kappa,p)$ such that
\begin{equation}\label{Rb2}
    \begin{split}
    &\Prob_{\mathbf{x}}\left( \mathcal{K}^c_{t,1}\cap \mathcal{H}_{q+1,1} \cap \mathcal{J}_{q,t}\right)
    \leq C n^{-d}(\log \log n)^{q+1} (\log n)^{(q+1)/r},\\
    &\Prob_{\mathbf{x}}\left(\mathcal{K}^c_{t,2}\cap \mathcal{K}_{t,1}\cap \mathcal{J}_{q,t}\right)
    \leq c (\log n)^{-\frac{p}{r}}(\log \log n)^{\frac{p(q+1)}{\mu+1}},
    \end{split}
\end{equation}
where $d=\min\{\chi/3-4\gamma,1-\chi,\gamma,\beta_q\}$.

To bound $\Prob_{\mathbf{x}}\left(\left(\bigcap^{3}_{i=1}\mathcal{H}_{q+1,i}\right)^c\right)$ using Lemma \ref{gencoupling}, \ref{regt} and (\ref{Rb2}), we note that
\begin{align*}
    &\Prob_{\mathbf{x}}\bigg(\bigg(\bigcap^{3}_{i=1}\mathcal{H}_{q+1,i}\bigg)^c\bigg)\\
    &\leq \Prob_{\mathbf{x}}\bigg(\bigg(\bigcap^{3}_{i=1}\mathcal{H}_{q+1,i}\bigg)^c \cap \bigcap^{3}_{j=1}\mathcal{H}_{q,j} \bigg) + \Prob_{\mathbf{x}}\bigg(\bigg(\bigcap^{3}_{j=1}\mathcal{H}_{q,j}\bigg)^c\bigg)\\
    &=\Prob_{\mathbf{x}}\bigg(\bigg(\bigcap^{3}_{i=2}\mathcal{H}_{q+1,i}\bigg)^c \cap \mathcal{H}_{q+1,1}\cap \bigcap^{3}_{j=1}\mathcal{H}_{q,j} \bigg) + \Prob_\mathbf{x}\bigg(\bigcap^{3}_{j=1}\mathcal{H}_{q,j}\cap \mathcal{H}^c_{q+1,1}\bigg)+ \Prob_{\mathbf{x}}\bigg(\bigg(\bigcap^{3}_{j=1}\mathcal{H}_{q,j}\bigg)^c\bigg)\\
    &=
   \Prob_{\mathbf{x}}\bigg(\bigg(\bigcap^{3}_{j=1}\mathcal{H}_{q,j}\bigg)^c\bigg)+\Prob_\mathbf{x}\bigg(\bigcap^{3}_{j=1}\mathcal{H}_{q,j}\cap \mathcal{H}^c_{q+1,1}\bigg)\\
   &\hspace{4cm} + \Prob_{\mathbf{x}}\bigg(\bigg(\bigcap_{\{s:k[s]\in \partial\mathcal{B}_q \}}(\mathcal{K}_{s,1}\cap\mathcal{K}_{s,2})\bigg)^c\cap\bigcap^{3}_{j=1}\mathcal{H}_{q,j}\cap \mathcal{H}_{q+1,1}\bigg),\numberthis\label{complementofH}
\end{align*}
where the last equality follows from 
\begin{equation*}
   \bigg(\bigcap^3_{j=2}\mathcal{H}_{q+1,j}\bigg)^c\cap \bigcap^3_{l=1}\mathcal{H}_{q,l}= \bigg(\bigcap_{\{s:k[s]\in \partial\mathcal{B}_q \}}(\mathcal{K}_{s,1}\cap\mathcal{K}_{s,2})\bigg)^c \cap \bigcap^3_{l=1}\mathcal{H}_{q,l}.
\end{equation*}
Thus the lemma is proved once we show that each probability on the right-hand side of (\ref{complementofH}) is of order at most $(\log\log n)^{-\chi}$. By assumption, there is a constant $C:=C(x_1,\mu,\kappa,q)$ such that $\Prob_{\mathbf{x}}\left(\left(\bigcap^{3}_{j=1}\mathcal{H}_{q,j}\right)^c\right) \leq C (\log\log n)^{-\chi}$. For the second probability, recall that $\hat a_{L[q],1}=U_{L[q],1}\hat a_{L[q]}$, where $U_{L[q],1}\sim\mathrm{U}[0,1]$ is independent of $\hat a_{L[q]}$. Hence,
\begin{align*}
    \Prob_{\mathbf{x}}\bigg(\bigcap^{3}_{i=1}\mathcal{H}_{q,i}\cap \mathcal{H}^c_{q+1,1}\bigg)
    &\leq \Prob_{\mathbf{x}}\left(\mathcal{H}^c_{q+1,1}\cap \mathcal{H}_{q,1} \right)\\
    &=\Prob_{\mathbf{x}}\left(U_{L[q],1}\hat a_{L[q]} \leq (\log\log n)^{-\chi(q+2)}, \hat a_{L[q]}>(\log\log n)^{-\chi(q+1)}\right)\\
    &\leq \Prob_{\mathbf{x}}(U_{L[q],1} \leq (\log\log n)^{-\chi})\\
    &= (\log \log n)^{-\chi}, 
\end{align*}
Finally, for the third probability we observe
\begin{align*}
    &\Prob_{\mathbf{x}}\bigg(\bigg(\bigcap_{\{s:k[s]\in \partial\mathcal{B}_q \}}(\mathcal{K}_{s,1}\cap\mathcal{K}_{s,2})\bigg )^c\cap\bigcap^{3}_{j=1}\mathcal{H}_{q,j}\cap \mathcal{H}_{q+1,1}\bigg)\\
    &\quad =\sum_{\{s:k[s]\in \partial \mathcal{B}_q\}} \left[\Prob_{\mathbf{x}}\left(\mathcal{K}^c_{s,2}\cap \mathcal{K}_{s,1}\cap\mathcal{J}_{q,s}\cap \mathcal{H}_{q+1,1}\right)+\Prob_{\mathbf{x}}\left( \mathcal{K}^c_{s,1}\cap \mathcal{J}_{q,s}\cap \mathcal{H}_{q+1,1}\right)\right];
\end{align*}
noting that by (\ref{qlogvertices}), $|\partial \mathcal{B}_q|\leq 1+q(\log n)^{q/r}$ on the event $\mathcal{H}_{q,3}$. Bounding the summands using Lemma \ref{gencoupling}, \ref{regt} and (\ref{Rb2}), it follows that there is a constant $C:=C(x_1,\mu,\kappa,q)$ such that the probability above is bounded by $C(\log\log n)^{-\chi}$. 
\end{proof}

\subsection{Proof of Theorem \ref{bigthm}}
Equipped with Corollary \ref{maincor}, we can prove Theorem \ref{bigthm}.

\begin{proof}
Following from Definition \ref{bs}, it is enough to establish (\ref{maindtv}) of Theorem \ref{bigthm}. Let $G_n\sim \mathrm{PA}(\pi)_n$, $(\mathcal{T}_{\mathbf{X},n},0)$ be the intermediate P\'olya point tree randomised over $\mathbf{X}$, and $(\mathcal{T},0)$ be the $\pi$-P\'olya point tree introduced in Definition \ref{pipolyapointgraph}. By the triangle inequality for the total variation distance, for any $r<\infty$ we have
\begin{align*}
    &d_{\mathrm{TV}}\left(\mathcal{L}\left((B_r(G_n,k_0),k_0)\right), \mathcal{L}\left((B_r(\mathcal{T},0),0)\right)\right)\\
    &\quad \leq d_{\mathrm{TV}}\left(\mathcal{L}\left((B_r(G_n,k_0),k_0)\right), \mathcal{L}\left((B_r(\mathcal{T}_{\mathbf{X},n},0),0)\right)\right)\\
    &\hspace{5cm}+ d_{\mathrm{TV}}\left(\mathcal{L}\left((B_r(\mathcal{T}_{\mathbf{X},n},0),0)\right), \mathcal{L}\left((B_r(\mathcal{T},0),0)\right)\right).
\end{align*}
Let $A_n:=A_{2/3,n}$ be as in (\ref{eventX}). Applying Jensen's inequality to the total variation distance, it can be seen that the above is bounded by
\begin{align*}
    &\E [d_{\mathrm{TV}}\left(\mathcal{L}\left((B_r(G_n,k_0),k_0)|\mathbf{X}\right),\mathcal{L}\left((B_r(\mathcal{T}_{\mathbf{X},n},0),0)|\mathbf{X}\right)\right)]\\
    &\hspace{3cm}+ d_{\mathrm{TV}}\left(\mathcal{L}\left((B_r(\mathcal{T}_{\mathbf{X},n},0),0)\right), \mathcal{L}\left((B_r(\mathcal{T},0),0)\right)\right)\\
    &\leq \E [\mathbbm{1}[\mathbf{x}\in A_n]d_{\mathrm{TV}}\left(\mathcal{L}\left((B_r(G_n,k_0),k_0)|\mathbf{X}=\mathbf{x}\right),\mathcal{L}\left((B_r(\mathcal{T}_{\mathbf{X},n},0),0)|\mathbf{X}=\mathbf{x}\right)\right)] \\
    &\hspace{3cm}+\Prob(\mathbf{x}\in A^c_n) + d_{\mathrm{TV}}\left(\mathcal{L}\left((B_r(\mathcal{T}_{\mathbf{X},n},0),0)\right), \mathcal{L}\left((B_r(\mathcal{T},0),0)\right)\right). 
\end{align*}
We prove that each term above is of order at most $(\log\log n)^{-\chi}$, starting from the expectation. Let $\mathcal{H}_{r,i}$, $i=1,2,3$ be as in (\ref{hqi}). For $G_n\sim \mathrm{Seq}(\mathbf{x})_n$ and $\mathbf{x}\in A_n$, Corollary \ref{maincor} implies that there is a coupling of $(G_n,k_0)$ and $(\mathcal{T}_{\mathbf{x},n},0)$, and a positive constant $C:=C(x_1,\mu,\kappa,r)$ such that 
\begin{align*}
    \Prob_{\mathbf{x}}\left(B_r(\mathcal{T}_{\mathbf{x},n},0),0\not \cong (B_r(G_n,k_0),k_0)\right)
    \leq \Prob_{\mathbf{x}}\bigg(\bigg(\bigcap^{3}_{i=1}\mathcal{H}_{r,i}\bigg)^c\bigg) \leq C (\log \log n)^{-\chi}. \numberthis \label{loglogbd}
\end{align*}
Hence applying definition (\ref{coupinter}) and the last display yields the desired upper bound. 

The bound on $\Prob(\mathbf{x}\in A^c_n)$ follows immediately from Lemma \ref{goodsum} (with $\alpha=2/3$). To handle the last term, we couple $(B_r(\mathcal{T},0),0)$ and $(B_r(\mathcal{T}_{\mathbf{X},n},0),0)$. For this coupling, denote by $\mathcal{E}$ the event that the PA labels in $B_r(\mathcal{T}_{\mathbf{X},n},0)$ are distinct, that is, $\hat k_{\bar u}\not= \hat k_{\bar v}$ for any $\bar u\not =\bar v$, $\bar u,\bar v\in V(B_r(\mathcal{T}_{\mathbf{X},n},0))$. We first construct $B_r(\mathcal{T}_{\mathbf{X},n},0)$, then on the event $\mathcal{E}$, we set $B_r(\mathcal{T},0)$ as $B_r(\mathcal{T}_{\mathbf{X},n},0)$, inheriting the Ulam-Harris labels, ages, types and fitness from the latter; and if the PA labels are not distinct, we generate $B_r(\mathcal{T},0)$ independently from $B_r(\mathcal{T}_{\mathbf{X},n},0)$. For any set $D \subseteq \mathcal{G}$, with $\mathcal{G}$ being the set of connected, rooted finite graphs, the triangle inequality yields
\begin{align*} 
    &|\Prob((B_r(\mathcal{T}_{\mathbf{X},n},0),0)\in D)-\Prob((B_r(\mathcal{T},0),0)\in D)|\\
    &\quad \leq |\Prob((B_r(\mathcal{T}_{\mathbf{X},n},0),0)\in D,\mathcal{E})-\Prob((B_r(\mathcal{T},0),0)\in D,\mathcal{E})| \\
    &\hspace{3cm} +|\Prob((B_r(\mathcal{T}_{\mathbf{X},n},0),0)\in D,\mathcal{E}^c)-\Prob((B_r(\mathcal{T},0),0)\in D,\mathcal{E}^c)|.\numberthis \label{twoabs}
\end{align*}
Under the coupling, $\Prob((B_r(\mathcal{T}_{\mathbf{X},n},0),0)\in D,\mathcal{E})=\Prob((B_r(\mathcal{T},0),0)\in D,\mathcal{E})$; and from the definition of $\mathcal{H}_{r,2}$, we can use (\ref{loglogbd}) and Lemma \ref{goodsum} to bound the second term in (\ref{twoabs}) by $\Prob(\mathcal{E}^c)\leq \Prob\left(\left(\bigcap^3_{i=1}\mathcal{H}_{r,1}\right)^c\right)\leq C(\log\log n)^{-\chi}$. Since the resulting bound does not depend on $D$, it follows from definition (\ref{tvddef}) of the total variation distance that 
\begin{equation*}
    d_{\mathrm{TV}}\left(\mathcal{L}\left((B_r(\mathcal{T}_{\mathbf{X},n},0),0)\right), \mathcal{L}\left((B_r(\mathcal{T},0),0)\right)\right)\leq C(\log\log n)^{-\chi},
\end{equation*}
and the theorem is proved.
\end{proof}

\section{Supplementary proofs for the local weak limit theorem}\label{ssup}
\subsection{Proof of Theorem \ref{linebreaking}}\label{pflinebreaking}
We use the urn embedding method to prove Theorem \ref{linebreaking} (see \cite{pekoz2017} and \cite{delphin2019} for example). Let $G_n\sim\mathrm{Seq(\mathbf{x}})_n$. Additionally, let $M_k(n):=\sum^k_{j=1}(x_j+W_{j,n})$, $U_k(n)=M_k(n)-M_{k-1}(n)$ and $M_0(n)=0$, where $W_{j,n}$ is the in-degree of vertex $j$ in $G_n$. In words, $M_k(n)$ is the total weight of the first $k$ vertices in $G_n$, while $U_k(n)$ is the weight of vertex~$k$ after $n$ completed attachment steps. If  $k>n$, then we set $U_k(n)=0$. 

Furthermore, denote by Polya$(b,w;n)$ the law of the number of white balls after the $n$th draw in a classical P\'olya urn initially with $w$ white balls and $b$ black balls. The following lemma is an easy modification of \cite[Lemma 2]{pekoz2017} that relates $U_k(n)$ to the number of white balls in a classical P\'olya urn. 
\begin{lemma}\label{urnembed}
Retaining the notations above, let $T_j:=\sum^j_{h=1}x_h$. Then given $n\geq 2$, 
\begin{equation}\label{polya1}
    U_{n-1}(n)\sim \mathrm{Polya}(T_{n-2}+n-2,x_{n-1};1),
\end{equation}
and conditional on $M_k(n)$ and the events $(\{U_j(n)=U_j(n-1)\},k+1\leq j\leq n-1)$, 
\begin{equation}\label{polya}
    U_k(n)\sim \mathrm{Polya}(T_{k-1}+k-1,x_k;M_k(n)-T_k-k+1).
\end{equation}
\end{lemma}

\begin{proof}
To prove (\ref{polya1}), note that when adding vertex $n$ to the existing graph $G_{n-1}$, the probability that vertex $n$ sends an outgoing edge to $n-1$ is $x_{n-1}/(T_{n-1}+n-2)$. This implies that $U_{n-1}(n)$ evolves like Polya$(T_{n-2}+n-2,x_{n-1};1)$.

For (\ref{polya}), observe that if one of the vertices $1,...,k$ is chosen when adding vertex $n$ to $G_{n-1}$, then by a straightforward computation using the definition of the conditional probability and the preferential attachment rule, we can show that the probability that $U_k(n)=U_k(n-1)+1$ is $U_k(n-1)/M_k(n-1)$. In particular, this implies that $U_k(n)$ behaves like Polya$(b,w;m)$, where $b=T_{k-1}+k-1$, $w=x_k$, and $m=M_k(n)-M_k(k)=M_k(n)-T_k-k+1$ being the number of times the vertices $1,...,k$ are picked after the $k$th attachment step. 
\end{proof}

Before using Lemma \ref{urnembed} to prove Theorem \ref{linebreaking}, we recall a result on the classical P\'olya urn initially with $a$ white balls and $b$ balls. The almost sure limit of the proportion of the white balls exists as the number of draws tends to infinity. Denote this limit by $\xi$, and it is well-known that $\xi\sim \mathrm{Beta}(a,b)$. By de Finetti's theorem (\textcite[Theorem 4.7.9, p.\ 220]{durrett}), we have that conditional on $\xi$, each draw is independent, and the probability of choosing a white ball is~$\xi$.

\begin{proof}[Proof of Theorem \ref{linebreaking}]
It is enough to consider the attachment steps $3\leq m \leq n$, because the first two steps of the graph constructions are deterministic. To add vertex $m$ to $G_{m-1}$, we consider vertices $1,...,m-1$ in decreasing order. The event that vertex $m$ sends an outgoing edge to vertex $m-1$ is $U_{m-1}(m)=x_{m-1}+1$. By Lemma \ref{urnembed}, $U_{m-1}(m)\sim \mathrm{Polya}(T_{m-1}+m-2,x_{m-1};1)$. So it follows from de Finetti's theorem that conditional on $B^{(x)}_{m-1}$, $U_{m-1}(m)=x_{m-1}+1$ (resp. $U_{m-1}(m)=x_{m-1}$) with probability $B^{(x)}_{m-1}$ (resp. $1-B^{(x)}_{m-1}$). 

Fix $1\leq k \leq m-2$, we now consider the event that vertex $m$ sends an outgoing edge to vertex $k$, which is exactly $U_k(m)=U_k(m-1)+1$. Given that vertex $m$ attaches to one of the vertices $\{1,...,k\}$, $U_k(m)\sim \mathrm{Polya}(T_{k-1}+k-1,x_k;M_k(m)-M_k(k))$ by Lemma~\ref{urnembed}. Using de Finetti's theorem, we deduce that conditional on $B^{(x)}_k$ and the events $(\{U_j(m)=U_j(m-1)\}, k+1\leq j\leq m-1)$, $U_k(m)=U_{k}(m-1)+1$ (resp.\ $U_k(m)=U_{k}(m-1)$) with probability $B^{(x)}_k$ (resp.\ $1-B^{(x)}_k$). We emphasize that $B^{(x)}_{k}$ only depends on $(x_j, 1\leq j\leq k)$ and not $M_k(m)$, since $M_k(m)-M_k(k)$ is the number of draws from the P\'olya urn.

The theorem is proved once we show that the conditionally on $(B^{(x)}_j, 1\leq j\leq n)$, the probability that vertex $m$ sends an outgoing edge to vertex $k$ is given by $\{S^{(x)}_{k,n}-S^{(x)}_{k-1,n}\}/S^{(x)}_{m-1,n}$, where $S^{(x)}_{h,n}:=\prod^n_{j=h+1}(1-B^{(x)}_j)$. Observe that
\begin{align*}
    &\Prob_{\mathbf{x}}\bigg(U_k(m)=U_k(m-1)+1, \bigcap^{m-1}_{i=k+1}\{U_i(m)=U_i(m-1)\}\bigg|G_{m-1}\bigg)\\
    &=\Prob_{\mathbf{x}}\bigg(U_k(m)=U_k(m-1)+1\bigg|\bigcap^{m-1}_{i=k+1}\{U_i(m)=U_i(m-1)\}, G_{m-1}\bigg)\\
    &\quad \times \prod^{m-2}_{h=k+1}\Prob_{\mathbf{x}}\bigg(U_h(m)=U_h(m-1)\bigg|\bigcap^{m-1}_{j=h+1}\{U_j(m)=U_j(m-1)\}, G_{m-1}\bigg)\\
    &\qquad \times \Prob_{\mathbf{x}}(U_{m-1}(m)=U_{m-1}(m-1)|G_{m-1}),
\end{align*}
and so vertex $m$ attaches to vertex $k$ with conditional probability $B^{(x)}_k \prod^{m-1}_{j=k+1}(1-B^{(x)}_j)$. Since
\begin{equation*}
    S^{(x)}_{k,n}-S^{(x)}_{k-1,n}=B^{(x)}_k\prod^n_{j=k+1}(1-B^{(x)}_j)\quad\text{and}\quad \frac{S^{(x)}_{k,n}-S^{(x)}_{k-1,n}}{S^{(x)}_{m-1,n}}=B^{(x)}_k\prod^{m-1}_{j=k+1}(1-B^{(x)}_j),
\end{equation*}
the proof is completed.
\end{proof}

\subsection{Proofs of Lemma \ref{goodsum}, \ref{Sasymptotic} and \ref{betagamma}}\label{pfssasymp}

\begin{proof}[Proof of Lemma \ref{goodsum}]
Given $p>2$, choose $1/2+1/p<\alpha<1$. Let $A_{\alpha,n}$ be as in (\ref{eventX}), and $C_p$ be the positive constant given in Lemma \ref{petrov} below, which bounds the moment of a sum of variables in terms of the moments of the summands. Denote $T^*_m:=\sum^m_{i=2}X_i$. Then 
\begin{align*}
    \Prob(\mathbf{x} \in A^c_{\alpha,n})&=\Prob\bigg(\bigcup^\infty_{j=\ceil{\phi(n)}} \{|T^*_j-(j-1)\mu|>j^{\alpha}\}\bigg)\\
    &\leq \sum^\infty_{j=\ceil{\phi(n)}}\Prob(|T^*_j-(j-1)\mu|>j^\alpha)\quad \text{by a union bound,}\\
    &\leq \sum^\infty_{j=\ceil{\phi(n)}}\E(|T^*_j-(j-1)\mu|^{p})j^{-\alpha p}\quad \text{by the Chebyshev's inequality,}\\
    &\leq C_{p}\E(|X_2-\mu|^{p})\sum^\infty_{j=\ceil{\phi(n)}}j^{-p(\alpha-1/2)}\quad \text{by Lemma \ref{petrov},}\\
    &\leq C_{p}\E(|X_2-\mu|^{p})\int^\infty_{\ceil{\phi(n)}-1} y^{-p(\alpha-1/2)}dy\\
    &= C_p [p(\alpha-1/2)-1]^{-1}\E(|X_2-\mu|^{p})(\ceil{\phi(n)}-1)^{1-p(\alpha-1/2)},
\end{align*}
where $p(\alpha-1/2)>1$ and $\E(|X_2-\mu|^{p})<\infty$. The lemma follows from $\phi(n)=\Omega(n^\chi)$.
\end{proof}

The next lemma can be found in \cite[Item 16, p.\ 60]{petrov}, where it is attributed to \cite{jogdeo69}. 

\begin{lemma}\label{petrov}
Let $Y_1,...,Y_n$ be independent random variables such that for $i=1,...,n$, $\E Y_i=0$ and $\E |Y_i|^p<\infty$ for some $p\geq 2$. Let $W_n:=\sum^n_{j=1}Y_j$, then
\begin{equation*}
    \E|W_n|^p\leq C_p n^{p/2-1}\sum^{n}_{i=1}\E|Y_i|^p,
\end{equation*}
where 
\begin{equation*}
    C_p:=\frac{1}{2}p(p-1)\max (1, 2^{p-3})\left(1+\frac{2}{p}K^{(p-2)/2m}_{2m}\right),
\end{equation*}
and the integer $m$ satifies the condition $2m\leq p\leq 2m+2$ and 
\begin{equation*}
    K_{2m}=\sum^m_{k=1}\frac{k^{2m-1}}{(k-1)!}.
\end{equation*}
\end{lemma}

Keeping the notations $T_i:=\sum^i_{h=1}x_h$ and $\phi(n)=\Omega(n^\chi)$, we now prove Lemma \ref{Sasymptotic} under the assumption $\mathbf{x}\in A_{\alpha,n}$. The proof is done in several steps. The first step is to give a slight variation of the moment formula in the proof of Proposition 3 of \textcite{delphin2019}.

\begin{lemma}
Let $B^{(x)}_j$, $S^{(x)}_{k,n}$ be as in Definition \ref{urnrep} and $T_i$ be as above. Then for $1\leq k < n$ and a positive integer $p$,
\begin{equation}\label{pth}
    \E_{\mathbf{x}}\left[\left(S^{(x)}_{k,n}\right)^p\right]= \left[\prod^{p-1}_{h=0}\frac{T_k+k+h}{T_n+n-1+h}\right] \prod^{p-1}_{j=0}\prod^{n-1}_{i=k+1}\left(1+\frac{1}{T_{i}+i-1+j}\right).
\end{equation}
\end{lemma}

\begin{proof}
Since $(B^{(x)}_j,1\leq j\leq n)$ are independent beta random variables, we use the moment formula of the beta distribution to show that for $p\geq 1$,
\begin{align*}
    \E_{\mathbf{x}}\left[\left(S^{(x)}_{k,n}\right)^p\right]&=\prod^n_{i=k+1}\E\left[\left(1-B^{(x)}_i\right)^p\right]=\prod^n_{i=k+1}\prod^{p-1}_{j=0}\frac{T_{i-1}+i-1+j}{T_i+i-1+j},\\
    &= \prod^{p-1}_{j=0} \left\{\frac{(T_k+k+j)}{(T_n+n-1+j)}\frac{(T_n+n-1+j)}{(T_k+k+j)}\prod^n_{i=k+1}\frac{T_{i-1}+i-1+j}{T_i+i-1+j}\right\}.
\end{align*}
Noting that $T_k+k+j$ and $T_n+n-1+j$ in the second product above cancel with $(T_n+n-1+j)/(T_k+k+j)$, we can rewrite the final term as
\begin{align*}
    \E_{\mathbf{x}}\left[\left(S^{(x)}_{k,n}\right)^p\right]&=\left[\prod^{p-1}_{h=0}\frac{T_k+k+h}{T_n+n-1+h}\right]\prod^{n-1}_{i=k+1} \prod^{p-1}_{j=0}\frac{T_{i}+i+j}{T_{i}+i-1+j}\\
    &=\left[\prod^{p-1}_{h=0}\frac{T_k+k+h}{T_n+n-1+h}\right] \prod^{p-1}_{j=0}\prod^{n-1}_{i=k+1}\left(1+\frac{1}{T_{i}+i-1+j}\right),
\end{align*}
hence concluding the proof.
\end{proof}

Note that taking $k=1$ in (\ref{pth}) recovers the original formula of \cite{delphin2019}, where $T_i$ here is $A_i$ in \cite{delphin2019}. In the second step, we obtain an estimate for $\E_{\mathbf{x}}[S^{(x)}_{k,n}]$ when $\mathbf{x}\in A_{\alpha,n}$ and $k\geq \ceil{\phi(n)}$.

\begin{lemma}\label{dctc}
Given $1/2<\alpha<1$ and a positive integer $n$, assume that $\mathbf{x}\in A_{\alpha,n}$. Then there is a positive constant $C:=C(x_1,\mu,\alpha)$ such that for all $\ceil{\phi(n)} \leq k\leq n$,
\begin{equation}\label{bigc}
   \left|\E_{\mathbf{x}}[S^{(x)}_{k,n}]- \left(\frac{k}{n}\right)^{\chi }\right|\leq C n^{\chi (\alpha-1)}.
\end{equation}
\end{lemma}

\begin{proof}
We first prove the upper bound for $\E_{\mathbf{x}}[S^{(x)}_{k,n}]$, using techniques appearing in the proof of \textcite[Lemma 4.4]{pekoz2019}. Applying the formula (\ref{pth}) (taking $p=1$), for $\mathbf{x}\in A_{\alpha,n}$ and $k\geq \ceil{\phi(n)}$ we obtain,
\begin{equation}\label{prodbound}
    \E_{\mathbf{x}}[S^{(x)}_{k,n}]\leq \frac{k(\mu+1)+k^\alpha+b}{n(\mu+1)-n^\alpha+b-1} \prod^{n-1}_{i=k+1} \left(1+\frac{1}{i\mu-i^{\alpha}+i+b-1}\right). 
\end{equation}
where $b:=x_1-\mu$. We rewrite the first term on the right-hand side of (\ref{prodbound}) as follows.
\begin{align*}
   \left(\frac{k}{n}\right)\frac{\mu+1+k^{-1+\alpha}+k^{-1}b}{\mu+1-n^{-1+\alpha}+n^{-1}(b-1)}
    &=  \left(\frac{k}{n}\right)\left(1+\frac{k^{\alpha-1}+n^{\alpha-1}+(k^{-1}-n^{-1})b+n^{-1}}{\mu+1-n^{\alpha-1}+n^{-1}(b-1)}\right)\\
    & \leq \frac{k}{n}(1+\bar Ck^{\alpha-1})\\
    &\leq \frac{k}{n}(1+\bar Cn^{\chi(\alpha-1)}),
\end{align*}
 where $\bar C:=\bar C(x_1,\mu,\alpha)$ is a positive constant. To bound the product term on the right-hand side of~(\ref{prodbound}), we take logarithm and bound
 \begin{equation*}
    \left|\sum^{n-1}_{i=k+1}\log\left(1+\frac{1}{i(\mu+1)-i^{\alpha}+b-1}\right)-\frac{1}{\mu+1}\log\left(\frac{n}{k}\right)\right|.
\end{equation*}
By the triangle inequality, we have
\begin{align*}
     &\left|\sum^{n-1}_{i=k+1} \log \left(1+\frac{1}{i\mu-i^{\alpha}+i+b-1}\right)-\frac{1}{\mu+1}\log\left(\frac{n}{k}\right)\right|\\
     & \quad\leq \left|\sum^{n-1}_{i=k+1} \log \left(1+\frac{1}{i\mu-i^{\alpha}+i+b-1}\right)-\frac{1}{i(\mu+1)-i^\alpha+b-1}\right| \numberthis \label{nathan} \\
     &\qquad + \left|\sum^{n-1}_{j=k+1}\frac{1}{j(\mu+1)-j^\alpha+b-1} - \log\left(\frac{n}{k}\right)\right|.\numberthis \label{ross}
\end{align*}
In order to bound (\ref{nathan}), we use $\log(1+y)=y+\sum_{j\geq 2}(-1)^{j+1}y^j/j$ for $y$ near zero. Letting $y_i=(i(\mu+1)-i^{\alpha}+b-1)^{-1}$, this implies that for $k\geq \ceil{\phi(n)}$ and $n$ large enough, (\ref{nathan}) is bounded by $\sum^n_{i=k+1} y^2_i$. Furthermore, by an integral comparison, we obtain $\sum^n_{i=k+1}y^2_i=O(n^{-\chi})$. For (\ref{ross}), we have
\begin{align*}
   & \left|\sum^{n-1}_{i=k+1}\frac{1}{i(\mu+1)-i^\alpha+b-1}-\frac{1}{(\mu+1)}\log\left(\frac{n}{k}\right)\right|\\
   &\qquad= \left|\sum^{n-1}_{i=k+1}\left(\frac{1}{i(\mu+1)-i^\alpha+b-1}-\frac{1}{(\mu+1)i}\right)+O(k^{-1})\right|\\
    &\qquad \leq \sum^{n-1}_{i=k+1}\left|\frac{i^{\alpha}-b+1}{i(\mu+1)(i(\mu+1)-i^\alpha+b-1)}\right|+O(k^{-1})\\
    &\qquad\leq C'\sum^{n-1}_{i=k+1}i^{-2+\alpha}+ O(k^{-1})\\
    & \qquad \leq C'(1-\alpha)^{-1}[\phi(n)]^{\alpha-1}+ O(n^{-\chi}),
\end{align*}
where $C':=C'(x_1,\mu,\alpha)$ is a constant. Combining the bounds above, a little calculation shows that there are positive constants $\bar C:=\bar C(x_1,\mu,\alpha)$, $\widetilde C:=\widetilde C(x_1,\mu,\alpha)$ and $\kappa_n:=\widetilde C n^{\chi(\alpha-1)}$ such that for $\mathbf{x}\in A_{\alpha,n}$ and $\ceil{\phi(n)}\leq k\leq n$,
\begin{align*}
    \E_{\mathbf{x}}[S^{(x)}_{k,n}]\leq (k/n)^\chi (1+\bar Cn^{\chi(\alpha-1)}) e^{\kappa_n}.
\end{align*}
Since $e^{x}=1+x+O(x^2)$ for $x$ near zero, for $n$ large enough, there is a positive constant $C:= C(x_1,\mu,\alpha)$ such that for $\ceil{\phi(n)}\leq k\leq n$,
\begin{align*}
    \E_{\mathbf{x}}[S^{(x)}_{k,n}]\leq  (k/n)^\chi(1+  C n^{\chi(\alpha-1)})\leq (k/n)^\chi +  Cn^{\chi(\alpha-1)},
\end{align*}
hence proving the desired upper bound. The lower bound can be proved by first noting that for $\ceil{\phi(n)}\leq k\leq n$,
\small
\begin{equation*}
    \E_{\mathbf{x}}[S^{(x)}_{k,n}]\geq \frac{k(\mu+1)-k^\alpha+b}{n(\mu+1)+n^\alpha+b-1}\prod^n_{i=k+1} \left(1+\frac{1}{i\mu+i^{\alpha}+i+b-1}\right). 
\end{equation*}
\normalsize
Repeating the calculations above, we get that when $\mathbf{x}\in A_{\alpha,n}$, there is a positive constant $c:=c(x_1,\mu,\alpha)$ such that
\begin{align*}
    \E_{\mathbf{x}}[S^{(x)}_{k,n}]
    \geq (k/n)^\chi-cn^{\chi(\alpha-1)},
\end{align*}
which completes the proof of the lemma.
\end{proof}

With Lemma \ref{dctc} and a martingale argument, we can prove Lemma \ref{Sasymptotic}. 

\begin{proof}[Proof of Lemma \ref{Sasymptotic}]
Let $\hat \delta_n=Cn^{\chi(\alpha-1)/4}$, where $C:=C(x_1,\mu,\alpha)$ is the positive constant in (\ref{bigc}) of Lemma \ref{dctc}. Denote by $D_{\hat \delta_n,n,{\mathbf{x}}}$ the event
\begin{equation*}
    \left\{\max\limits_{K \leq k \leq n} \left|S^{(x)}_{k,n}-\left(\frac{k}{n}\right)^{\chi}\right| \geq 2\hat \delta_n\right\},
\end{equation*}
where we write $K:=\ceil{\phi(n)}$ to shorten formulas. The lemma follows from bounding $\Prob_{\mathbf{x}}(D_{\hat \delta_n,n,{\mathbf{x}}})$ under the assumption $\mathbf{x}\in A_{\alpha,n}$. By the triangle inequality, we have
\begin{align*}
    \Prob_{\mathbf{x}}\left(D_{\hat \delta_n,n,{\mathbf{x}}}\right) 
    &\leq \Prob_{\mathbf{x}}\left(\max\limits_{K \leq k \leq n}\left|S^{(x)}_{k,n}-\E_{\mathbf{x}}[S^{(x)}_{k,n}]\right|+\max\limits_{K \leq j \leq n}\left|\E_{\mathbf{x}}[S^{(x)}_{j,n}]-\left(\frac{j}{n}\right)^{\chi}\right|\geq  2\hat \delta_n  \right).
\end{align*}
Applying Lemma \ref{dctc} to bound the difference between $S^{(x)}_{k,n}$ and $\E_{\mathbf{x}}[S^{(x)}_{k,n}]$, we obtain
\begin{align*}
    \Prob_{\mathbf{x}}\left(D_{\hat \delta_n,n,\mathbf{x}}\right) \leq \Prob_{\mathbf{x}}\left(\max\limits_{K \leq k \leq n}\left|S^{(x)}_{k,n}-\E_{\mathbf{x}}[S^{(x)}_{k,n}]\right|\geq \hat \delta_n\right)
    \leq \Prob_{\mathbf{x}}\left(\max\limits_{K \leq k \leq n}\left|S^{(x)}_{k,n}(\E_{\mathbf{x}}[S^{(x)}_{k,n}])^{-1}-1\right|\geq \hat \delta_n\right),
\end{align*}
where the final inequality is due to $\E_{\mathbf{x}}[S^{(x)}_{k,n}]\leq 1$. We proceed to bound the right-hand side of the above by using martingale techniques. Since $\E_{\mathbf{x}}[S^{(x)}_{k,n}]=\prod^n_{j=k+1}\E(1-B^{(x)}_j)$, we construct a martingale as follows. Define $ M^{(x)}_0:=1$ and for $j=1,...,n-K$, let
\begin{equation*}
    M^{(x)}_{j}:=\prod^{n}_{i=n-j+1}\frac{1-B^{(x)}_i}{\E[1-B^{(x)}_i]}=\frac{S^{(x)}_{n-j,n}}{\E[S^{(x)}_{n-j,n}]}.
\end{equation*}
Let $\mathcal{F}^{(x)}_{j}$ be the $\sigma$-algebra generated by $(B^{(x)}_i,n-j+1\leq i \leq n)$ for $1\leq j\leq n-K$, with $\mathcal{F}^{(x)}_0=\varnothing$. It follows that $((M^{(x)}_{j}, \mathcal{F}^{(x)}_{j}),0\leq j\leq n-K)$ is a martingale and $\E[M^{(x)}_{j}]=1$. Noting that $(M^{(x)}_j-1)^2$ is a submartingale, Doob's inequality \cite[Theorem 4.4.2, p.\ 204]{durrett} yields
\begin{align*}
   \Prob_{\mathbf{x}}\left(D_{\hat \delta_n,n,{\mathbf{x}}}\right)\leq 
    \Prob_{\mathbf{x}}\left(\max\limits_{0 \leq j\leq n-K}\left|M^{(x)}_{j}-1\right|\geq\hat \delta_n \right) \leq \hat \delta_n^{-2}\mathrm{Var}_{\mathbf{x}}(M^{(x)}_{n-K}). \numberthis \label{condpb}
\end{align*}
We then use the formulas for the first and second moments of the beta distribution to bound the variance in (\ref{condpb}):
\begin{align*}
    \mathrm{Var}_{\mathbf{x}}(M^{(x)}_{n-K})&=\E_{\mathbf{x}}[(M^{(x)}_{n-K})^2]-1
     =\bigg[\prod^n_{j=K+1}\frac{(T_{j-1}+j)}{(T_{j-1}+j-1)}\frac{(T_j+j-1)}{(T_{j}+j)}\bigg]-1\\
    & = \prod^n_{j=K+1}\left(1+\frac{T_j-T_{j-1}}{(T_j+j)(T_{j-1}+j-1)}\right)-1. 
\end{align*}
In the display below, let $C'=C'(x_1,\mu,\alpha)$ be a positive constant that may vary from line to line. As $|\sum^j_{h=2}x_h - (j-1)\mu| \leq  j^\alpha $ for all $K+1\leq j\leq n$ when $\mathbf{x}\in A_{\alpha,n}$, we obtain
\begin{align*}
    \mathrm{Var}_{\mathbf{x}}(M^{(x)}_{n-K})&\leq  \prod^n_{j=K+1}\left(1+\frac{\mu +j^\alpha+(j-1)^\alpha}{\{(\mu+1)j-\mu-j^\alpha+x_1 \}\{(\mu+1) j-2\mu-j^\alpha+x_1-1\}}\right)-1\\
    &\leq \prod^n_{j=K+1}\left(1+\frac{\mu +O(j^{\alpha})}{\{(\mu+1)j-\mu-j^\alpha+x_1 \}\{(\mu+1) j-2\mu-j^\alpha+x_1-1\}}\right)-1\\
    &\leq \prod^n_{j=K+1} (1+ C'j^{\alpha-2})-1 \\
    &\leq C' K^{\alpha-1} \numberthis \label{doobb}
\end{align*}
Applying (\ref{doobb}) to (\ref{condpb}) completes the proof.
\end{proof}

We conclude this subsection with the proof of Lemma \ref{betagamma}.

\begin{proof}[Proof of Lemma \ref{betagamma}]
Let $T_m:=\sum^m_{k=1}x_k$ and $Y_j\sim \mathrm{Gamma}(T_j+j-1,1)$. We start by proving (\ref{bgc}). Let $E_{\varepsilon,j,{\mathbf{x}}}$ be as in (\ref{epsj}), we have
\begin{align*}
    \Prob_{\mathbf{x}}(E^c_{\varepsilon,j,{\mathbf{x}}})&= \Prob_{\mathbf{x}}\left(\left|\frac{\mathcal{Z}_{j}}{\mathcal{Z}_{j}+\mathcal{\tilde Z}_{j-1}}-\frac{\mathcal{Z}_{j}}{(\mu+1)j}\right|\geq\frac{\mathcal{Z}_{j}}{(\mu+1)j}\varepsilon\right)\\
    &=\Prob_{\mathbf{x}}\left(\left|\frac{(\mu+1)j}{\mathcal{Z}_{j}+\mathcal{\tilde Z}_{j-1}}-1\right|\geq\varepsilon\right)\\
    &\leq \Prob_{\mathbf{x}}\left(\left|\frac{\mathcal{Z}_{j}+\mathcal{\tilde Z}_{j-1}}{(\mu+1)j}-1\right|\geq\frac{\varepsilon}{1+\varepsilon}\right)\\
    &=\Prob_{\mathbf{x}}\left(\left|\frac{Y_j}{(\mu+1)j}-1\right|\geq\frac{\varepsilon}{1+\varepsilon}\right)\\
    &\leq \left(\frac{1+\varepsilon}{\varepsilon}\right)^4 \E_{\mathbf{x}}\left[\left(\frac{Y_j}{(\mu+1)j}-1\right)^4\right];
\end{align*}
and so (\ref{bgc}) follows from bounding the moment in the last display, and then applying a union bound. We bound the moment above under the assumption $\mathbf{x}\in A_{\alpha,n}$. Let $a_j:=T_j+j-1$. By the moment formula for the standard gamma distribution, 
\begin{align*}
    \E_{\mathbf{x}}\left[\left(\frac{Y_j}{(\mu+1)j}-1\right)^4\right]
    &=\frac{\E_{\mathbf{x}}(Y^{4}_j)}{(\mu+1)^4j^4}-\frac{4\E_{\mathbf{x}}(Y^{3}_j)}{(\mu+1)^3j^3}+\frac{6\E_{\mathbf{x}}(Y^{2}_j)}{(\mu+1)^2 j^2 }-\frac{4\E_{\mathbf{x}}(Y_j)}{(\mu+1)j}+1\\
    &=\frac{\prod^3_{k=0}(a_j+k)}{(\mu+1)^4j^4}-\frac{4\prod^2_{k=0}(a_j+k)}{(\mu+1)^3j^3}+\frac{6\prod^1_{k=0}(a_j+k)}{(\mu+1)^2j^2}-\frac{4 a_j}{(\mu+1)j}+1.
\end{align*}
Noting that $|a_j-(\mu+1)j|\leq j^\alpha+x_1+\mu+1$ for $j\geq \phi(n)$, a direct calculation shows that there is a positive constant $C:=C(x_1, \mu,\alpha)$ such that
\begin{equation}\label{gammamom}
    \E_{\mathbf{x}}\left[\left(\frac{Y_j}{(\mu+1)j}-1\right)^4\right]\leq Cj^{4\alpha-4}.
\end{equation}
We now prove (\ref{bgc}) using (\ref{gammamom}). Let $C:=C(x_1,\alpha,\mu)$ be a positive constant that may vary from line to line in the subsequent formulas. Then,
\begin{align*}
    \Prob_{\mathbf{x}}\bigg(\bigcup^n_{j= \ceil{\phi(n)}} E^c_{\varepsilon,j,\mathbf{x}}\bigg)&\leq  \sum^n_{j= \ceil{\phi(n)}}\Prob_{\mathbf{x}}(E^c_{\varepsilon,j,\mathbf{x}}) 
     \leq C \left(\frac{1+\varepsilon}{\varepsilon}\right)^4  \sum^n_{j=\ceil{\phi(n)}} j^{4\alpha-4} \\
    &\leq C \left(\frac{1+\varepsilon}{\varepsilon}\right)^4 \int^\infty_{\ceil{\phi(n)}-1} y^{4\alpha-4} dy 
    \quad \leq C(1+\varepsilon)^4\varepsilon^{-4}n^{\chi(4\alpha-3)},
\end{align*}
as required. Next, we use a union bound and Chebyshev's inequality to prove (\ref{woboundedass}) as follows:
\begin{align*}
    \Prob_{\mathbf{x}}\bigg(\bigcup^n_{j=\ceil{\phi(n)}}\{\mathcal{Z}_{j}\geq j^{1/2}\}\bigg)
    & \leq \sum^n_{j= \ceil{\phi(n)}} \Prob_{\mathbf{x}}(\mathcal{Z}_{j}\geq j^{1/2})
    \leq  \sum^n_{j= \ceil{\phi(n)}} \E_{\mathbf{x}} (\mathcal{Z}^4_{j}) j^{-2}\\
    &= \sum^n_{j= \ceil{\phi(n)}} j^{-2} \prod^{3}_{\ell=0} (x_j+\ell). 
\end{align*}
If we further assume $x_2 \in (0, \kappa]$, then there is a positive number $C'$ such that
\begin{align*}
   \Prob_{\mathbf{x}}\bigg(\bigcup^n_{j=\ceil{\phi(n)}}\{\mathcal{Z}_{j}\geq j^{1/2}\}\bigg)\leq C' \kappa^4 \sum^n_{j= \ceil{\phi(n)}} j^{-2} \leq C' \kappa^4 \int^{\infty}_{\phi(n)-1} y^{-2} dy \leq C' \kappa^4 n^{-\chi},
\end{align*}
hence proving (\ref{boundedwhp}). 
\end{proof}

\subsection{Proof of Lemma \ref{poissonproc}}\label{pfpoissonproc}
In preparation, we use $\hat P_{k\to k_0}$ in (\ref{mean2}) to construct a Bernoulli point process and a discretised mixed Poisson process, which shall appear in the intermediate coupling steps. Recall that $k_0:=k[1]$ and $\zeta_0 := \mathcal{Z}_{k_0}[1]$. Furthermore, define
\begin{equation}\label{collectionrv}
    \Xi_{\mathbf{x}}:=(U_0, (\mathcal{Z}_j[1],\mathcal{\tilde Z}_j[1]),2\leq j\leq n)),
\end{equation}
and observe that the event $\bigcap^3_{i=1}F_{1,i} \cap \mathcal{H}_{1,0}$, defined in (\ref{h10}) and (\ref{3f}), is measurable with respect to $\Xi_{\mathbf{x}}$. 

\begin{defn}\label{interber}
Given $\hat a_0$, $k_0$ and $\zeta_0$, let $\hat Y_{k\to k_0}$, $k_0+1\leq k\leq n$, be conditionally independent Bernoulli variables, each with parameter $\hat P_{k\to k_0}$. We define a Bernoulli point process by the random vector
 \begin{equation*}
    \mathbf{\hat Y}^{(k[1],n)}:=(\hat Y_{(k_0+1)\to k_0},\hat Y_{(k_0+2)\to k_0},...,\hat Y_{n\to k_0}).
\end{equation*}
\end{defn}

\begin{defn}\label{interpoi}
Given $\hat a_0$, $k_0$ and $  \zeta_0$, let $\hat V_{k\to k_0}$, $k_0+1\leq k\leq n$, be conditionally independent Poisson random variables, each with parameter $\hat P_{k\to k_0}$. We define a mixed discretised Poisson point process by the random vector
 \begin{equation*}
    \mathbf{\hat V}^{(k[1],n)}:=(\hat V_{(k_0+1)\to k_0},\hat V_{(k_0+2)\to k_0},...,\hat V_{n\to k_0}). 
\end{equation*}
\end{defn}

We also require a simple result that turns the problem of coupling two random vectors into the problem of coupling two random variables.

\begin{lemma}\label{tvdpf}
Given a positive integer $d$, let $\mathbf{V}=(V_1,...,V_d)$ and $\mathbf{W}=(W_1,...,W_d)$ be vectors of independent random variables. Then there is a coupling of the random vectors such that
\begin{equation*}
    \Prob(\mathbf{V}\not = \mathbf{W}) \leq \sum^d_{i=1} \Prob( V_i\not= W_i).
\end{equation*}
\end{lemma}

\begin{proof}
For $i=1,..,d$, let $(\tilde V_i, \tilde W_i)$ be a coupling of $(V_i, W_i)$, where for $i\not = j$, $(\tilde V_i,\tilde W_i)$ is independent of $(\tilde V_j,\tilde W_j)$. Denote $\mathbf{\tilde V}=(\tilde V_1,...,\tilde V_d)$ and $\mathbf{\tilde W}=(\tilde W_1,...,\tilde W_d)$. By a union bound,
\begin{equation*}
      \Prob(\mathbf{\tilde V}\not = \mathbf{\tilde W}) \leq \Prob(\cup^d_{i=1} \{ \tilde V_i \not = \tilde W_i \})\leq \sum^d_{i=1} \Prob(\tilde V_i \not = \tilde W_i). \qedhere
\end{equation*}
\end{proof}

The proof of Lemma \ref{poissonproc} consists of two main components. The first is to use Lemma~\ref{tvdpf} and standard techniques to couple $(\mathbf{Y}^{(k[1],n)}, \mathbf{\hat Y}^{(k[1],n)})$, $(\mathbf{\hat Y}^{(k[1],n)},\mathbf{\hat V}^{(k[1],n)})$ and $(\mathbf{\hat V}^{(k[1],n)},\mathbf{V}^{(k[1],n)})$ under the event $\bigcap^3_{i=1}F_{1,i} \cap \mathcal{H}_{1,0}$. These results are given in the next three lemmas. The second is to combine these lemmas, and use a union bound argument. 

\begin{lemma}\label{firstcoupling}
Let $\mathbf{Y}^{(k[1],n)}$, $\mathbf{\hat Y}^{(k[1],n)}$, $\mathcal{H}_{1,0}$, $F_{1,i}$, $i=1,2,3$ and $\Xi_\mathbf{x}$ be as in Definition  \ref{berproc01}, \ref{interber}, (\ref{h0}), (\ref{3f}) and (\ref{collectionrv}). There is a coupling of the random vectors and a positive constant $C:=C(x_1,\mu)$ such that on the event $\bigcap^3_{i=1}F_{1,i} \cap \mathcal{H}_{1,0}$,
\begin{equation}\label{bos1}
    \Prob_\mathbf{x}\left(\mathbf{Y}^{(k[1],n)}\not = \mathbf{\hat Y}^{(k[1],n)}\big|\Xi_{\mathbf{x}}\right)
     \leq C \zeta_0 n^{-\gamma} (\log \log n)^{1-\chi}.
\end{equation}
\end{lemma}

\begin{proof}
We first show that on the event $\bigcap^3_{i=1}F_{1,i} \cap \mathcal{H}_{1,0}$, there is a positive constant $C:=C(x_1,\mu)$ such that 
\begin{equation}\label{impbd}
    (1-Cn^{-\gamma})\hat P_{k\to k_0}\leq P_{k\to k_0} \leq (1+Cn^{-\gamma})\hat P_{k\to k_0},
\end{equation}
and then couple the random vectors. We only prove the upper bound in (\ref{impbd}), as the lower bound follows from a similar calculation. Choose $n$ large enough so that $C^*n^{-\chi/12}(\log\log n)^\chi <1/2$, where $C^*$ is the constant in the event $F_{1,1}$. Since $k_0>n(\log\log n)^{-1}$ on the event $\mathcal{H}_{1,0}$, we get that on the event $F_{1,1}$,
\begin{align*}
     \frac{S_{k_0,n}[1]}{S_{k,n}[1]} &\leq \left\{\left(\frac{k_0}{n}\right)^\chi + C^*n^{-\chi/12}\right\}\left\{\left(\frac{k}{n}\right)^\chi - C^*n^{-\chi/12}\right\}^{-1}\\
    &\leq \left(\frac{n}{k}\right)^\chi \left\{\left(\frac{k_0}{n}\right)^\chi + C^*n^{-\chi/12}\right\}\left\{1 - C^*n^{-\chi/12}(\log\log n)^\chi\right\}^{-1}\\
    &\leq \left\{\left(\frac{k_0}{k}\right)^\chi + C^*n^{-\chi/12}(\log \log n)^\chi\right\} \sum_{l\geq 0}(-1)^l [-C^*n^{-\chi/12}(\log \log n)^\chi]^\ell\\
    &\leq \left\{\left(\frac{k_0}{k}\right)^\chi + C^*n^{-\chi/12}(\log \log n)^\chi\right\}\{1 + 2C^*n^{-\chi/12}(\log\log n)^\chi\},
\end{align*}
where we have used $(n/k)^\chi < (n/k_0)^\chi <(\log \log n)^\chi$ and the generalised binomial series (\cite[equation (5.13), p.\ 163]{concrete}). Thus, there is a positive constant $C':=C'(x_1,\mu)$ such that
\begin{align*}
    \frac{S_{k_0,n}[1]}{S_{k,n}[1]} \leq \left(\frac{k_0}{k}\right)^\chi+C'n^{-\frac{\chi}{12}}(\log \log n)^{\chi}\quad \text{for $k_0<k\leq n$}.
\end{align*}
Hence, on the event $\bigcap^3_{i=1}F_{1,i} \cap \mathcal{H}_{1,0}$, we can bound $P_{k\to k_0}$ in terms of $\hat P_{k\to k_0}$:
\begin{align*}
    P_{k\to k_0}&\leq \left[\frac{  \zeta_0}{(\mu+1)k_0}+\frac{  \zeta_0n^{-\gamma}}{(\mu+1)k_0}\right]\left[\left(\frac{k_0}{k}\right)^\chi+C'n^{-\frac{\chi}{12}}(\log \log n)^{\chi}\right]\\
    &=  \hat P_{k\to k_0}\left[1 + n^{-\gamma} +
    \left(\frac{k}{k_0}\right)^\chi C'(\log \log n)^{\chi} (n^{-\gamma-\frac{\chi}{12}} + n^{-\frac{\chi}{12}})\right].
\end{align*}
Using $(k/k_0)^\chi \leq (n/k_0)^\chi \leq (\log \log n)^\chi$ and $0<\gamma<\chi/12$, for large enough $n$ we have
\begin{align*}
     P_{k\to k_0}&\leq \hat P_{k\to k_0} (1+ n^{-\gamma} + C'n^{-\gamma-\frac{\chi}{12}}(\log \log n)^{3\chi} + C' n^{-\frac{\chi}{12}}(\log \log n)^{3\chi})\\
    &= \hat P_{k\to k_0}(1+3C'n^{-\gamma}).
\end{align*}
For the coupling, let $U_k$, $k_0+1\leq k\leq n$ be independent standard uniform variables. Define
\begin{equation*}
    Y'_{k\to k_0}=\mathbbm{1}[U_k\leq~ P_{k\to k_0}]\quad\text{and}\quad \hat Y'_{k\to k_0}=\mathbbm{1}[U_k\leq \hat P_{k\to k_0}].
\end{equation*}
Then on the event $\bigcap^3_{i=1}F_{1,i} \cap \mathcal{H}_{1,0}$, we obtain
\begin{equation*}
    \Prob_\mathbf{x}\left(Y'_{k\to k_0}\not =\hat Y'_{k\to k_0}\big|\Xi_{\mathbf{x}}\right)
    \leq \Prob\left(U_k \leq |P_{k\to k_0}-\hat P_{k\to k_0}|\big|\Xi_{\mathbf{x}}\right) \leq Cn^{-\gamma}\hat P_{k\to k_0}. 
\end{equation*}
By Lemma \ref{tvdpf}, we have that on the event $\bigcap^3_{i=1}F_{1,i} \cap \mathcal{H}_{1,0}$,
\begin{equation*}
    \Prob_\mathbf{x}\left(\mathbf{Y}^{(k[1],n)}\not = \mathbf{\hat Y}^{(k[1],n)}\big|\Xi_{\mathbf{x}}\right)\leq Cn^{-\gamma} \sum^n_{j=k_0+1} \hat P_{j\to k_0}. 
\end{equation*}
To bound the sum above, we use $k_0>n(\log \log n)^{-1}$ and an integral comparison to get
\begin{align*}
    \mathbbm{1}\bigg[\bigcap^3_{i=1}F_{1,i} \cap \mathcal{H}_{1,0}\bigg] \sum^n_{j=k_0+1} \hat P_{j\to k_0}&= \mathbbm{1}\bigg[\bigcap^3_{i=1}F_{1,i} \cap \mathcal{H}_{1,0}\bigg]\frac{   \zeta_0}{(\mu+1)k^{1-\chi}_0}\sum^n_{k=k_0+1} k^{-\chi} \\
    &\leq \frac{   \zeta_0}{(\mu+1)}\left(\frac{\log \log n }{n}\right)^{1-\chi} \int^n_{n(\log \log n)^{-1}} y^{-\chi} dy\\ 
     &\leq   \zeta_0 (\log \log n)^{1-\chi}. \numberthis \label{sumbd}
\end{align*}
Combining the last two inequalities gives (\ref{bos1}).
\end{proof}

\begin{lemma}\label{secondcoupling}
Let $\mathbf{\hat Y}^{(k[1],n)}$, $\mathbf{\hat V}^{(k[1],n)}$, $\mathcal{H}_{1,0}$, $F_{1,i}$, $i=1,2,3$ and $\Xi_\mathbf{x}$ be as in Definition \ref{interber}, \ref{interpoi}, (\ref{h0}), (\ref{3f}) and (\ref{collectionrv}). Then there is a coupling of the random vectors such that on the event $\bigcap^3_{i=1}F_{1,i} \cap \mathcal{H}_{1,0}$,
\begin{equation*}
    \Prob_\mathbf{x}\left(\mathbf{\hat Y}^{(k[1],n)}\not = \mathbf{\hat V}^{(k[1],n)}\big|\Xi_{\mathbf{x}}\right)\leq  \frac{\zeta^2_0(\log\log n)^{2-\chi}}{n}.
\end{equation*}
\end{lemma}

\begin{proof}
By the standard Poisson-Bernoulli coupling \cite[equation (1.11), p.\ 5]{lindvall2002lectures}, we have
\begin{equation*}
    \Prob_\mathbf{x}\left(\hat Y_{k\to k_0}\not = \hat V_{k\to k_0}\big|\Xi_{\mathbf{x}}\right) \leq \hat P^2_{k\to k_0}\quad \text{for $k_0+1\leq k\leq n$.}
\end{equation*}
Lemma \ref{tvdpf} and $(k_0/k)^\chi \leq 1$ imply 
\begin{equation*}
     \Prob_\mathbf{x}\left(\mathbf{\hat Y}^{(k[1],n)}\not = \mathbf{\hat V}^{(k[1],n)}\big|\Xi_{\mathbf{x}}\right)\leq \sum^n_{j=k_0+1} \hat P^2_{j\to k_0} \leq \frac{  \zeta_0}{(\mu+1)k_0} \sum^n_{j=k_0+1} \hat P_{j\to k_0},
\end{equation*}
so the lemma follows from applying (\ref{sumbd}) to the sum above, and noting $k_0>n(\log \log n)^{-1}$ on the event $\mathcal{H}_{1,0}$. 
\end{proof}

\begin{lemma}\label{thirdcoupling}
Let $\mathbf{\hat V}^{(k[1],n)}$, $\mathbf{\hat V}^{(k[1],n)}$, $\mathcal{H}_{1,0}$, $F_{1,i}$, $i=1,2,3$ and $\Xi_\mathbf{x}$ be as in Definition \ref{dispoi}, \ref{interpoi}, (\ref{h0}), (\ref{3f}) and (\ref{collectionrv}). Then there is a coupling of the random vectors and a positive constant $C:=C(\mu)$ such that on the event $\bigcap^3_{i=1}F_{1,i} \cap \mathcal{H}_{1,0}$,
\begin{equation*}
    \Prob_\mathbf{x}\left(\mathbf{V}^{(k[1],n)}\not = \mathbf{\hat V}^{(k[1],n)}\big|\Xi_{\mathbf{x}}\right)\leq  \frac{C\log \log n}{n}  \zeta_0.
\end{equation*}
\end{lemma}

\begin{proof}
We construct a monotone coupling of $V_{k\to k_0}$ and $\hat V_{k\to k_0}$. Let $\nu_j:=\lambda^{[1]}_{j}\wedge \hat P_{j\to k_0}$ and
\begin{align*}
    V'_{j\to k_0}\sim \mathrm{Poi}(\nu_j),\quad 
    V''_{j\to k_0}\sim \mathrm{Poi}(|\lambda^{[1]}_{j}-\hat P_{j\to k_0}|),\quad \hat V'_{j\to k_0}= V'_{j\to k_0}+V''_{j\to k_0}, 
\end{align*}
where $V''_{j\to k_0}$ is conditionally independent of $V'_{j\to k_0}$. Then
\begin{equation*}
    \Prob_\mathbf{x}\left(\hat V'_{k\to k_0}\not =V'_{k\to k_0} \big|\Xi_{\mathbf{x}}\right)
    =  \Prob(V''_{k\to k_0}\geq 1\big|\Xi_{\mathbf{x}})\leq |\lambda^{[1]}_k - \hat P_{k\to k_0}|,
\end{equation*}
where the last inequality follows from $1-e^{-x}\leq x$. By Lemma \ref{tvdpf}, 
\begin{equation}\label{bos3}
   \Prob_\mathbf{x}\left(\mathbf{V}^{(k[1],n)}\not = \mathbf{\hat V}^{(k[1],n)}\big|\Xi_{\mathbf{x}}\right)\leq \sum^n_{j=k_0+1}|\lambda^{[1]}_{j}-\hat P_{j\to k_0}|.
\end{equation}
We now bound the sum in (\ref{bos3}) on the event $\mathcal{H}_{1,0}$. We first show that we can swap $a^{-1/\mu}_0$ in $\lambda^{[1]}_k$ for $(k_0/n)^{-\chi/\mu}$ at a small cost, and then proceed to bound the difference between $\lambda^{[1]}_k$ and $\hat P_{k\to k_0}$. Recalling that $\hat a_0=U^\chi_0$ and $k_0=\ceil{nU_0}$, we have
\begin{align*}
    \hat a^{-1/\mu}_0-\left(\frac{k_0}{n}\right)^{-\chi/\mu}\leq U^{-\chi/\mu}_0 \bigg[1-\left(1+\frac{1}{n}\right)^{-\chi/\mu}\bigg].
\end{align*}
Using $\chi/\mu=1-\chi$ and $(1+n^{-1})^{\chi-1}=\sum_{i\geq 0}\binom{\chi-1}{i}(1/n)^i$, we get that there is a constant $C_\mu$ such that 
\begin{align}\label{agediff}
    \hat a^{-1/\mu}_0-(k_0/n)^{-\chi/\mu}\leq C_\mu n^{-1} U^{-\chi/\mu}_0. 
\end{align}
Noting that $U_0\geq (\log \log n)^{-1}$ on the event $\mathcal{H}_{1,0}$, it follows that
\begin{equation}\label{bda0}
     \mathbbm{1}[\mathcal{H}_{1,0}]\left(\hat a^{-1/\mu}_0-(k_0/n)^{-\chi/\mu}\right)\leq C_\mu n^{-1} (\log \log n)^{1-\chi} =:\eta_n.
\end{equation}
For $k=k_0+2,...,n$, we use (\ref{bda0}) to compute
\begin{align*}
    \lambda^{[1]}_k &= \int^{(k/n)^\chi}_{((k-1)/n)^\chi}  \frac{\zeta_0}{\mu \hat a^{1/\mu}_0} y^{1/\mu-1} dy=  \zeta_0\hat a^{-1/\mu}_0 \left[\left(\frac{k}{n}\right)^{1-\chi} - \left(\frac{k-1}{n}\right)^{1-\chi}\right]\\
    &\leq \zeta_0\left(\frac{k}{n}\right)^{1-\chi} \left[1 - \left(1-\frac{1}{k}\right)^{1-\chi}\right]\left[\left(\frac{k_0}{n}\right)^{\chi-1}+ \eta_n\right].
\end{align*}
Since $(1-1/k)^\chi = \sum_{j\geq 0} \binom{1-\chi}{j}(-k)^{-j}=1-[(\mu+1)k]^{-1}+O(k^{-2})$,
\begin{align*}
     \lambda^{[1]}_k &\leq \zeta_0 \left[\frac{1}{(\mu+1)k}+O(k^{-2})\right]\left[\left(\frac{k}{k_0}\right)^{1-\chi}+ \eta_n\left(\frac{k}{n}\right)^{1-\chi}\right];
\end{align*}
and expanding the terms above we obtain
\begin{align*}
     \lambda^{[1]}_k &\leq \hat P_{k\to k_0} + \frac{\eta_n\zeta_0}{n^{1-\chi}}\frac{1}{k^\chi(\mu+1)} + c\zeta_0 k^{-1-\chi}[\eta_n n^{\chi-1}+ k^{\chi-1}_0] k^{-1-\chi},
\end{align*}
where $c:=c(\mu)$ is a constant. In the subsequent calculations, we allow the constants $c:=c(\mu)$ and $c':=c'(\mu)$ to vary from term to term. Repeating the calculation above for a lower bound on $\lambda^{[1]}_{k}$, we deduce that on the event $\mathcal{H}_{1,0}$, 
\begin{equation} \label{meandiff}
    |\lambda^{[1]}_k - \hat P_{k\to k_0}|\leq \frac{c \zeta_0\eta_n}{ n^{1-\chi}}  k^{-\chi}+c'\zeta_0 k^{\chi-1}_0 k^{-1-\chi}.
\end{equation}
On the event $\mathcal{H}_{1,0}$, $cn^{\chi-1}\eta_n\zeta_0 \sum^n_{j=k_0+2}k^{-\chi}\leq c(\mu+1)\zeta_0\eta_n$ because
\begin{align*}
    \mathbbm{1}[\mathcal{H}_{1,0}] \sum^n_{k=k_0+2} k^{-\chi}\leq  \sum^n_{k=\ceil{n(\log \log n)^{-1}}+2} k^{-\chi}\leq \int^n_{n(\log \log n)^{-1}+1} y^{-\chi} dy\leq (\mu+1)n^{1-\chi};
\end{align*}
and similarly, $c'\zeta_0k^{\chi-1}_0\sum^n_{k=k_0+1}k^{-1-\chi}\leq c'\chi^{-1}\zeta_0 n^{-1}(\log\log n)$. Hence,
\begin{align*}
  \mathbbm{1}[ \mathcal{H}_{1,0}] \sum^n_{j=k_0+2}|\lambda^{[1]}_{j}-\hat P_{j\to k_0}|\leq \frac{c \log \log n}{n}  \zeta_0.
\end{align*}
Finally, we can use (\ref{bda0}) and a similar calculation to show that
\begin{align*}
   \mathbbm{1}[ \mathcal{H}_{1,0}] |\lambda^{[1]}_{k_0+1}-\hat P_{k_0+1\to k_0}|\leq \frac{c' \log \log n}{n} \zeta_0.
\end{align*}
Thus, applying the last two displays to (\ref{bos3}) gives the desired result.
\end{proof}

We now use Lemma \ref{firstcoupling}, \ref{secondcoupling} and \ref{thirdcoupling} to prove Lemma \ref{poissonproc}.

\begin{proof}[Proof of Lemma \ref{poissonproc}]
We show that on the event $\bigcap^3_{i=1}F_{1,i} \cap \mathcal{H}_{1,0}$, there is a coupling of $(\mathbf{Y}^{(k[1],n)}, \mathbf{V}^{(k[1],n)})$, and positive constants $C:=C(x_1,\mu)$ and $c:=c(\mu)$ such that
\begin{align*}
   &\Prob_{\mathbf{x}}(\mathbf{Y}^{(k[1],n)}\not = \mathbf{V}^{(k[1],n)}\big|\Xi_\mathbf{x})\\
    &\qquad \leq \frac{C\zeta_0(\log \log n)^{1-\chi}}{n^\gamma}+ \frac{c\zeta_0\log \log n}{n} +  \frac{\zeta^2_0(\log\log n)^{2-\chi}}{n}; \numberthis \label{toterr}
\end{align*}
and the lemma follows from taking expectation with respect to $\zeta_0$ on $(0,\infty)$, since $\E_{\mathbf{x}} [\zeta_0|U_0]=x_{k_0}$, $\E_{\mathbf{x}} [\zeta^2_0|U_0]=x_{k_0}(x_{k_0}+1)\leq \kappa(\kappa+1)$. Using $\mathbf{\hat Y}^{(k[1],n)}$ and $\mathbf{\hat V}^{(k[1],n)}$, we construct a coupling such that the conditional laws $\mathcal{L}(\mathbf{\hat Y}^{(k[1],n)}|\mathbf{Y}^{(k[1],n)})$,
\begin{gather*}
    \mathcal{L}\left(\mathbf{\hat V}^{(k[1],n)}| \mathbf{\hat Y}^{(k[1],n)}, \mathbf{Y}^{(k[1],n)}\right)=\mathcal{L}\left(\mathbf{\hat V}^{(k[1],n)}|\mathbf{\hat Y}^{(k[1],n)}\right),\\
    \mathcal{L}\left(\mathbf{V}^{(k[1],n)}|\mathbf{\hat V}^{(k[1],n)}, \mathbf{\hat Y}^{(k[1],n)}, \mathbf{Y}^{(k[1],n)}\right) =  \mathcal{L}\left(\mathbf{V}^{(k[1],n)}|\mathbf{\hat V}^{(k[1],n)}\right)
\end{gather*}
are as in Lemma \ref{firstcoupling}, \ref{secondcoupling} and \ref{thirdcoupling}. Note that given $\mathbf{\hat V}^{(k[1],n)}$, $\mathbf{V}^{(k[1],n)}$ is independent of the other two random vectors; while given $\mathbf{\hat Y}^{(k[1],n)}$, $\mathbf{\hat V}^{(k[1],n)}$ is independent of $\mathbf{Y}^{(k[1],n)}$. To prove (\ref{toterr}) using the probability bounds given in these lemmas, we note that
\begin{align*}
    &\Prob_{\mathbf{x}}\left(\mathbf{Y}^{(k[1],n)}\not = \mathbf{V}^{(k[1],n)}\big|\Xi_\mathbf{x}\right)\\
    &\leq \Prob_\mathbf{x}\left(\{\mathbf{Y}^{(k[1],n)}\not = \mathbf{\hat Y}^{(k[1],n)}\}\cup \{\mathbf{\hat Y}^{(k[1],n)}\not = \mathbf{\hat V}^{(k[1],n)}\}\cup\{ \mathbf{\hat V}^{(k[1],n)}\not = \mathbf{V}^{(k[1],n)}\}\big|\Xi_\mathbf{x}\right)\\
    &\leq \Prob_{\mathbf{x}}\left(\mathbf{Y}^{(k[1],n)}\not = \mathbf{\hat Y}^{(k[1],n)}\big|\Xi_\mathbf{x}\right)
    +\Prob_{\mathbf{x}}\left(\mathbf{\hat Y}^{(k[1],n)}\not = \mathbf{\hat V}^{(k[1],n)}\big|\Xi_\mathbf{x}\right)\\
   &\hspace{7cm} +\Prob_{\mathbf{x}}\left(\mathbf{\hat V}^{(k[1],n)}\not = \mathbf{V}^{(k[1],n)}\big|\Xi_\mathbf{x}\right),\numberthis \label{vectorineq} 
\end{align*}
where the last inequality is due to a union bound.
\end{proof}

\subsection{Proof of Lemma \ref{regularity}}\label{pfuniformd}
\begin{proof}
Recall that $  \zeta_0\sim \mathrm{Gamma}(x_{k_0},1)$ and $\hat \tau_0\sim\mathrm{Po}(  \zeta_0(\hat a^{-1/\mu}_0-1))$, where $\hat \tau_0$ is the number of type R vertices in $\partial \mathfrak{B}_1$ that are attached to the root $0\in V((\mathcal{T}_{\mathbf{x},n},0))$. On the event $\mathcal{H}_{1,0}=\{\hat a_0>(\log \log n)^{-\chi}\}$, $\hat \tau_0$ is stochastically dominated by $\tilde \tau_0\sim \mathrm{Po}(\zeta_0(\log\log n)^{1-\chi})$. Hence,
\begin{align}\label{regularbd1}
    \Prob_{\mathbf{x}}\bigg(\bigcap^2_{i=0} \mathcal{H}_{1,i}\cap \mathcal{H}^c_{1,3}\bigg)
    \leq \Prob_{\mathbf{x}}(\mathcal{H}_{1,0},\tilde \tau_0\geq (\log n)^{1/r})
    \leq \E_{\mathbf{x}}\left[\Prob_{\mathbf{x}}(\tilde \tau_0\geq (\log n)^{1/r}| \zeta_0,U_0)\right].
\end{align}
To apply Chebyshev's inequality, let $(Y)_k:=Y(Y-1)\cdots(Y-k+1)$ for non-negative integer $k$. By \cite[equation (6.10), p.\ 262]{concrete}, $Y^p = \sum^p_{k=0}\stirling{p}{k}(Y)_k$, where $\stirling{p}{k}$ is the Stirling number of the second kind (with $\stirling{0}{0}=1$ and $\stirling{p}{0}=0$ for positive integer $p$). If $Y\sim \mathrm{Po}(\theta)$, then $\E[(Y)_k]=\theta^k$. For such $Y$ and $\theta\geq 1$, taking expectation on both sides of identity gives
\begin{equation*}
    \E Y^p = \sum^p_{k=0}\stirling{p}{k}\E[(Y)_k]\leq C_p \theta^p, 
\end{equation*}
where $C_p$ is the sum of the Stirling numbers. Choose $p>\max\{r-1,7\}$. By Chebyshev's inequality and the moment bound above,
\begin{align*}
   \Prob(\tilde \tau_0\geq (\log n)^{1/r}| \zeta_0,U_0)
    \leq (\log n)^{-p/r} \E[\tilde \tau^p_0| \zeta_0,U_0]
    = C_p \zeta^p_0(\log n)^{-p/r} (\log\log n)^{p(1-\chi)}.  
\end{align*}
Applying the last display to (\ref{regularbd1}), and noting that there is a positive constant $c_p$ such that
\begin{align*}
    \E_{\mathbf{x}}[\zeta^p_0] =\E_{\mathbf{x}}\bigg[\prod^{p-1}_{\ell=0}(x_{k_0}+\ell)\bigg]\leq c_p \kappa^p
\end{align*}
because $x_i\leq \kappa$ for $i\geq 2$ proves the lemma.
\end{proof}

\subsection{Proof of Lemma \ref{extnhood}}\label{pfextnhood}
We first recall some notations introduced in Section \ref{secgenr} as they frequently appear in the proof below. The random variable $\rho[q]$ is the time we probe the type L vertex in $\partial \mathcal{B}_q$, and $L[q]=(0,1,...,1)$, $|L[q]|=q+1$, so that $k[\rho[q]]=k_{L[q]}$. Moreover, $M_{L[q]}:=\min\{k: (k/n)^\chi\geq a_{L[q]}\}$, $k^*_{L[q]}:=\min\{k_{L[q]}+1, M_{L[q]}\}$ and $\zeta_q := \mathcal{Z}_{k[\rho[q]]}{[\rho[q]]}$. 

To prepare for the intermediate coupling steps, we construct a Bernoulli point process and a discretised mixed Poisson process using the means $\hat P_{k\to k[\rho[q]]}$ given in (\ref{tmean3}) and Table~\ref{tableofmeans}.

\begin{defn}\label{interber01}
Given $k_{L[q]}$, $\hat a_{L[q]}$ and $\zeta_q$, let $\hat Y_{j\to k_{L[q]}}$, $k^*_{L[q]}\leq j\leq n$, be conditionally independent Bernoulli variables with parameter $\hat P_{j\to k_{L[q]}}$ given in (\ref{tmean3}) and Table \ref{tableofmeans}. We define this Bernoulli point process by the random vector
\begin{equation*}
    \mathbf{\hat Y}^{(k_{L[q]},n)}:=\left(\hat Y_{k^*_{L[q]}\to k_{L[q]}},\hat Y_{(k^*_{L[q]}+1)\to k_{L[q]}},...,\hat Y_{n\to k_{L[q]}}\right).
\end{equation*}
\end{defn}

\begin{defn}\label{interpoi01}
Given $k_{L[q]}$, $\hat a_{L[q]}$ and $\zeta_q$, let $\hat V_{j \to k_{L[q]}}$, $k^*_{L[q]}\leq j\leq n$, be conditionally independent Poisson random variables, each with parameters $\hat P_{j\to k[t]}$ given in (\ref{tmean3}) and in Table \ref{tableofmeans}. We define this discretised mixed Poisson point process with the random vector
\begin{equation*}
    \mathbf{\hat V}^{(k_{L[q]},n)}:=\left(\hat V_{k^*_{L[q]}\to k_{L[q]}},\hat V_{(k^*_{L[q]}+1)\to k_{L[q]}},...,\hat V_{n\to k_{L[q]}}\right).
\end{equation*}
\end{defn}

As we assume that $(G_n,k_0)$ and $(\mathcal{T}_{\mathbf{x},n},0)$ are already coupled such that $B_q(G_n,k_0),k_0)\cong (B_r(\mathcal{T}_{\mathbf{x},n},0),0)$, it is enough to condition on the following collection of random variables in the sequel, 
\begin{align}\label{collection}
   \left((k_{\bar z}, \hat k_{\bar z}, \hat a_{\bar z})_{k_{\bar z}\in \mathcal{A}_{\rho[q]-1}\cup \mathcal{P}_{\rho[q]-1}}, 
   (\theta_{\bar x}, \hat \tau_{\bar x})_{k_{\bar x}\in \mathcal{P}_{\rho[q]-1}}, (\mathcal{\tilde Z}_i[\rho[q]], \mathcal{Z}_i[\rho[q]])_{2\leq i\leq n, i\not \in \mathcal{P}_{\rho[q]-1}}\right),
\end{align}
where $\mathcal{P}_{\rho[q]-1}=V(B_{q-1}(G_n,k_0))$ and $\mathcal{A}_{\rho[q]-1}=\partial \mathfrak{B}_q$. Denote this collection by $\Xi_\mathbf{x}$, and define
\begin{equation}\label{eventJ}
    J:=\bigcap^3_{i=1} F_{\rho[q], i}\cap \bigg(\bigcap^{3}_{k=1} \mathcal{H}_{q,k} \bigg),
\end{equation}
where $\mathcal{H}_{q,k}$, $k=1,2,3$ and $F_{\rho[q], i}$, $i=1,2,3$ are as in (\ref{hqi}) and (\ref{3fL}). To prove Lemma \ref{extnhood}, we also observe that on the event $\mathcal{H}_{q,1}\cap \mathcal{H}_{q,2}$, 
\begin{align}\label{vertlab}
    k_{L[q]}\geq n(\log\log n)^{-(q+1)} - C_qn^{1-\beta_q/\chi} \quad\text{and}\quad M_{L[q]}\geq n(\log\log n)^{-(q+1)},
\end{align}
where $C_q:=C_q(x_1,\mu)$ is the positive constant in $\mathcal{H}_{q,2}$. 

We are now ready to prove the lemma. As in the case of Lemma \ref{poissonproc}, we first couple the pairs
\begin{equation*}
    \left(\mathbf{Y}^{(k[t],n)},\mathbf{\hat Y}^{(k[\rho[q]],n)}\right),\text{ }\left(\mathbf{\hat Y}^{(k[\rho[q]],n)},\mathbf{\hat V}^{(k[\rho[q]],n)}\right)\text{ and }\left(\mathbf{\hat V}^{(k[\rho[q]],n)},\mathbf{V}^{(k[\rho[q]],n)}\right)
\end{equation*}
on the event $J$, and then apply a union bound argument. 

\begin{lemma}\label{rfirstcoup}
Let $\mathbf{Y}^{(k[\rho[q]],n)}$, $\mathbf{\hat Y}^{(k[\rho[q]],n)}$, $\Xi_{\mathbf{x}}$ and $J$ be as in (\ref{adjberpoi}), Definition \ref{interber01}, (\ref{eventJ}) and (\ref{collection}). Then there is coupling of the random vectors, and positive constants $C:=C(x_1,\mu,q)$ and $c:=c(x_1,\mu)$ such that on the event $J$, 
\begin{align*}
    &\Prob_{\mathbf{x}}\left(\mathbf{Y}^{(k[\rho[q]],n)}\not = \mathbf{\hat Y}^{(k[\rho[q]],n)}\big|\Xi_{\mathbf{x}}\right)\\
    &\hspace{2cm}\leq C\zeta_q n^{-\gamma}(\log \log n)^{(1-\chi)(q+1)} +\frac{c\zeta_q(\log n)^{q/r}(\log \log n)^{(q+1)(1-\chi)}}{(\mu+1)n^{1-\chi}}.
\end{align*}
\end{lemma}

\begin{proof} 
We consider the scenarios $M_{L[q]}\leq k[\rho[q]]$ and $M_{L[q]}\geq k[\rho[q]]+1$ separately, starting from $M_{L[q]}\leq k[\rho[q]]$. For $M_{L[q]}\leq j\leq k[\rho[q]]$, $Y_{j \to k[\rho[q]]}=\hat Y_{j \to k[\rho[q]]}=0$ since $P_{j \to k[\rho[q]]}=\hat P_{j \to k[\rho[q]]}=0$; whereas for $k[\rho[q]]+1\leq h\leq n$, we couple $Y_{h\to k[\rho[q]]}$ and $\hat Y_{h\to k[\rho[q]]}$ as follows. Firstly, we use a similar calculation as for (\ref{impbd}) to show that there is a positive constant $C:=C(x_1,\mu,q)$ such that 
\begin{equation*}
    (1-Cn^{-\gamma})\hat P_{k\to k[\rho[q]]}\leq P_{k\to k[\rho[q]]} \leq (1+Cn^{-\gamma})\hat P_{k\to k[\rho[q]]},\quad k\in \mathcal{N}_{\rho[q]-1}\cap \{k_{\rho[q]}+1,...,n\},
\end{equation*}
on the event $J$. Next, let $U_k$, $k[\rho[q]]+1\leq k\leq n$ be independent standard uniform variables, and define
\begin{equation*}
    Y'_{k\to k[\rho[q]]}=\mathbbm{1}\left[U_k\leq P_{k\to k[\rho[q]]}\right]\quad\text{and}\quad\hat Y'_{k\to k[\rho[q]]}=\mathbbm{1}\left[U_k\leq \hat P_{k\to k[\rho[q]]}\right].
\end{equation*}
So on the event $J$, we have 
\begin{equation*}
  \Prob_\mathbf{x}\left(Y_{k\to k[\rho[q]]}\not=\hat Y_{k\to k[\rho[q]]}\big|\Xi_{\mathbf{x}}\right)\leq \Prob\left(U_k\leq |\hat P_{k\to k[\rho[q]]}-P_{k\to k[\rho[q]]}|\big|\Xi_{\mathbf{x}}\right)\leq C n^{-\gamma} \hat P_{j\to k[\rho[q]]}
\end{equation*}
for $k \in \mathcal{N}_{\rho[q]-1}$; and  
\begin{equation*}
    \Prob_\mathbf{x}\left(Y_{k\to k[\rho[q]]}\not=\hat Y_{k\to k[\rho[q]]}\big|\Xi_{\mathbf{x}}\right)\leq \Prob(U_k\leq \hat P_{k\to k[\rho[q]]}\big|\Xi_{\mathbf{x}}) \leq \hat P_{j\to k[\rho[q]]}
\end{equation*}
for $k \in \mathcal{A}_{\rho[q]-1}\cup \mathcal{P}_{\rho[q]-1}$. Using Lemma \ref{tvdpf}, we obtain
\begin{align*}
    &\Prob_{\mathbf{x}}\left(\mathbf{Y}^{(k[\rho[q]],n)}\not = \mathbf{\hat Y}^{(k[\rho[q]],n)}\big|\Xi_{\mathbf{x}}\right) \\
    &\qquad \leq Cn^{-\gamma} \sum^n_{\substack{j=k[\rho[q]]+1;\\j \in \mathcal{N}_{\rho[q]-1}}}\hat P_{j\to k[\rho[q]]}+ \sum^n_{\substack{k=k[\rho[q]]+1;\\k\in \mathcal{A}_{\rho[q]-1}\cup \mathcal{P}_{\rho[q]-1}}}\hat P_{k\to k[\rho[q]]}.\numberthis \label{rfirstb}
\end{align*}
In what follows we allow $C:=C(x_1,\mu,q)$ to vary from line to line. For the first sum above, we use $n(\log\log n)^{-(q+1)}\leq M_{L[q]}\leq k[\rho[q]]$ to obtain
\begin{align*}
   Cn^{-\gamma} \sum^n_{\substack{j=k[\rho[q]]+1;\\j\in \mathcal{N}_{\rho[q]-1}}}\hat P_{j\to k[\rho[q]]} &\leq \frac{Cn^{-\gamma}\zeta_q}{(\mu+1)\{k[\rho[q]]\}^{1-\chi}}\sum^n_{k=k[\rho[q]]+1} k^{-\chi}\\
    &\leq \frac{Cn^{-\gamma}\zeta_q}{(\mu+1)\{k[\rho[q]]\}^{1-\chi}}\int^n_{k[\rho[q]]}y^{-\chi}dy\\
    &\leq C\zeta_qn^{-\gamma}(\log \log n)^{(1-\chi)(q+1)}. \numberthis \label{rfirstb1}
\end{align*}
To bound the second sum in (\ref{rfirstb}), we note that (\ref{qlogvertices}) implies that
\begin{equation*}
    |\mathcal{A}_{\rho[q]-1}|+|\mathcal{P}_{\rho[q]-1}|=|V(B_q(G_n,k_0))|\leq 1+q+q^2(\log n)^{q/r}
\end{equation*}
on the event $\mathcal{H}_{q,3}$. So combining the last display, the lower bound on $k[\rho[q]]$ and $h^{-\chi}\leq 1$ for $h\geq 1$, we get that on the event $J$, there is a constant $c:=c(x_1,\mu,q)$ such that
\begin{align*}
   \sum_{h\in \mathcal{A}_{\rho[q]-1}\cup \mathcal{P}_{\rho[q]-1}}\hat P_{h\to k[\rho[q]]}&=\frac{\zeta_q}{(\mu+1)[k[\rho[q]]]^{1-\chi}}\sum_{h\in \mathcal{A}_{\rho[q]-1}\cup \mathcal{P}_{\rho[q]-1}}h^{-\chi} 
   \\&
   \leq \frac{c\zeta_q(\log \log n)^{(q+1)(1-\chi)}}{n^{1-\chi}}(\log n)^{q/r}.\numberthis \label{rfirstb2}
\end{align*}
Hence, applying (\ref{rfirstb1}) and (\ref{rfirstb2}) to (\ref{rfirstb}) shows that the lemma holds whenever $M_{L[q]}\leq k[\rho[q]]$. The coupling argument for $M_{L[q]}\geq k[\rho[q]]+1$ is exactly the same as above, thus omitted. 
\end{proof}

\begin{lemma}\label{rsecondcoup}
Let $\mathbf{\hat Y}^{(k[\rho[q]],n)}$, $\mathbf{\hat V}^{(k[\rho[q]],n)}$, $\Xi_{\mathbf{x}}$ and $J$ be as in  Definition \ref{interber01} and \ref{interpoi01}, (\ref{collection}) and (\ref{eventJ}). Then there is coupling of the random vectors, and a positive constant $C:=C(x_1,\mu,q)$ such that on the event $J$,  
\begin{equation*}
   \Prob_\mathbf{x}\left(\mathbf{\hat Y}^{(k[\rho[q]],n)}\not= \mathbf{\hat V}^{(k[\rho[q]],n)}\big|\Xi_{\mathbf{x}}\right) \leq \frac{C\zeta^{2}_{q}(\log \log n)^{(2-\chi)(q+1)}}{n}.
\end{equation*}
\end{lemma}

\begin{proof}
We again consider $M_{L[q]}\leq k[\rho[q]]$ and $M_{L[q]}\geq k[\rho[q]]+1$ separately, starting from $M_{L[q]}\leq k[\rho[q]]$. For $M_{L[q]}\leq j\leq k[\rho[q]]$, $\hat Y_{j \to k[\rho[q]]}=\hat V_{j \to k[\rho[q]]}=0$ because $P_{j \to k[\rho[q]]}=\hat P_{j \to k[\rho[q]]}=0$. Hence, we only need to couple $\hat Y_{h \to k[\rho[q]]}$ and $\hat V_{h \to k[\rho[q]]}$ for $k[\rho[q]]+1\leq h\leq n$. By the standard Poisson-Bernoulli coupling \cite[equation (1.11), p.\ 5]{lindvall2002lectures},
\begin{equation*}
   \Prob_\mathbf{x}(\hat Y_{k\to k[\rho[q]]}\not=\hat V_{k\to k[\rho[q]]}|\Xi_{\mathbf{x}})\leq \hat P^2_{k\to k[\rho[q]]}, \quad k[\rho[q]]+1\leq k\leq n.
\end{equation*}
Lemma \ref{tvdpf} and $[k[\rho[q]]/k]^\chi\leq 1$ yield
\begin{equation*}
     \Prob_\mathbf{x}\left(\mathbf{\hat Y}^{(k[\rho[q]],n)}\not= \mathbf{\hat V}^{(k[\rho[q]],n)}\big|\Xi_{\mathbf{x}}\right) \leq \sum^n_{k=k[\rho[q]]+1}\hat P^2_{k\to k[\rho[q]]} \leq \frac{\zeta_q}{(\mu+1)k[\rho[q]]} \sum^n_{k=k[\rho[q]]+1}\hat P_{k\to k[\rho[q]]},
\end{equation*}
so bounding the sum by (\ref{rfirstb1}) proves the lemma for $M_{L[q]}\leq k[\rho[q]]$. The arguments for the case $M_{L[q]}\geq k[\rho[q]]+1$ are similar, thus omitted.
\end{proof}
\begin{lemma}\label{rthirdcoup}
Let $\mathbf{V}^{(k[\rho[q]],n)}$, $\mathbf{\hat V}^{(k[\rho[q]],n)}$, $\Xi_{\mathbf{x}}$ and $J$ be as in Definition \ref{pp01}, \ref{interpoi01}, (\ref{collection}) and (\ref{eventJ}). Then there is a coupling of the random vectors, and a positive constant $C:=C(x_1,\mu,q)$ such that on the event $J$,
\begin{align*}
   \Prob_\mathbf{x}\left(\mathbf{V}^{(k[\rho[q]],n)}\not= \mathbf{\hat V}^{(k[\rho[q]],n)}\big|\Xi_{\mathbf{x}}\right)
  \leq C\zeta_q n^{-\beta_q}(\log \log n)^{q+1}.
\end{align*}
\end{lemma}

\begin{proof}
We first construct a coupling of $V_{k\to k[\rho[q]]}$ and $\hat V_{k\to k[\rho[q]]}$ that is applicable to both cases $M_{L[q]}\leq k_{L[q]}$ and $M_{L[q]}\geq k_{L[q]}+1$. Then, we bound the probability that the coupling fails for each case. For $k^*_{L[q]} +1\leq j\leq n$, let $\nu_j=~\lambda^{[\rho[q]]}_{j}\wedge \hat P_{j\to k[\rho[q]]}$,
\begin{align*}
    V'_{j\to k[\rho[q]]}\sim \mathrm{Poi}(\nu_j), \quad V''_{j\to k[\rho[q]]}\sim \mathrm{Poi}(|\lambda^{[\rho[q]]}_{j}-\hat P_{j\to k[\rho[q]]}|),
    \quad \hat V'_{j\to k[\rho[q]]}= V'_{j\to k[\rho[q]]}+V''_{j\to k[\rho[q]]},
\end{align*}
where $V''_{j\to k[\rho[q]]}$ is conditionally independent of $V'_{j\to k[\rho[q]]}$ and we use the convention that zero is a Poisson variable with mean zero. Then,
\begin{equation*}
   \Prob_\mathbf{x}(\hat V_{k\to k[\rho[q]]}\not=V_{k\to k[\rho[q]]}|\Xi_{\mathbf{x}}) \leq \Prob(V''_{k\to k[\rho[q]]}\geq 1\big|\Xi_{\mathbf{x}})\leq \left|\lambda^{[\rho[q]]}_k - \hat P_{k\to k[\rho[q]]}\right|,
\end{equation*}
So by Lemma \ref{tvdpf} and comparing the means in Table \ref{tableofmeans}, 
\begin{align*}
  &\Prob_\mathbf{x}\left(\mathbf{V}^{(k[\rho[q]],n)}\not= \mathbf{\hat V}^{(k[\rho[q]],n)}\big|\Xi_{\mathbf{x}}\right)\\
  &\leq
   \begin{cases}\numberthis \label{annoyingsum}
       \sum^{k[\rho[q]]}_{h=M_{L[q]}} \lambda^{[\rho[q]]}_h + \sum^n_{j=k[\rho[q]]+1}\left|\lambda^{[\rho[q]]}_{j}-\hat P_{j\to k[\rho[q]]}\right|, \quad M_{L[q]}\leq k_{L[q]},\\
       \sum^n_{j=k_{L[q]}+1}\left|\lambda^{[\rho[q]]}_{j}-\hat P_{j\to k[\rho[q]]}\right|,\qquad M_{L[q]}= k_{L[q]}+1,\\
       \sum^{M_{L[q]}}_{h=k[\rho[q]]+1} \hat P_{h\to k[\rho[q]]} + \sum^n_{j=M_{L[q]}}\left|\lambda^{[\rho[q]]}_{j}-\hat P_{j\to k[\rho[q]]}\right|, \quad M_{L[q]}\geq k_{L[q]}+2.
   \end{cases}
\end{align*}
We first handle the sum of absolute mean differences for all three cases in (\ref{annoyingsum}). Note that for $M_{L[q]}\leq k_{L[q]}+1$,
\begin{equation*}
    \sum^n_{j=k[\rho[q]]+1}\left|\lambda^{[\rho[q]]}_{j}-\hat P_{j\to k[\rho[q]]}\right|\leq \sum^n_{j=M_{L[q]}}\left|\lambda^{[\rho[q]]}_{j}-\hat P_{j\to k[\rho[q]]}\right|,
\end{equation*}
so it is enough to consider the sum from $M_{L[q]}$ to $n$ in all three cases. Because $|\hat a_{L(q)}-(k_{L[q]}/n)^{\chi}|\leq C_q n^{-\beta_q}$ on the event $\mathcal{H}_{q,2}$, a computation similar to that of (\ref{bda0}) gives
\begin{equation}\label{coupledages}
    \bigg|\hat a^{-\frac{1}{\mu}}_{L[q]} - \{k_{L[q]}/n\}^{-\frac{\chi}{\mu}}\bigg|\leq c_\mu n^{-\beta_q}(\log \log n)^{1+q}=:\eta_n, 
\end{equation}
where $c_\mu:=c_\mu(q,x_1,\mu)$ is a constant. Using (\ref{coupledages}) and the second inequality in (\ref{vertlab}), we can repeat the calculation of (\ref{meandiff}) to obtain
\begin{equation}\label{lambdaPbound}
    \left|\lambda^{[\rho[q]]}_{j}-\hat P_{j\to k[\rho[q]]}\right|\leq \frac{c\zeta_q\eta_n}{n^{1-\chi}} j^{-\chi} + c'\zeta_q(\log\log n)^{(1-\chi)(q+1)} j^{-2},
\end{equation}
where $c:=c(x_1,\mu)$ and $c':=c'(\mu)$. Using (\ref{vertlab}) and integral comparisons, a straightforward computation shows that on the event $J$, $\sum^n_{j=M_{L[q]}} j^{-\chi}\leq (\mu+1)n^{1-\chi}$. The second term of (\ref{lambdaPbound}) can be bounded in the same way, such that on the event $J$, $\sum^n_{j=M_{L[q]}} j^{-2}\leq n^{-1}(\log\log n)^{q+1}$. Combining the bounds, we get that for all three cases in (\ref{annoyingsum}), there is a constant $C:=C(x_1,\mu,q)$ such that
\begin{equation}\label{LPsum}
    \sum^n_{j=M_{L[q]}} \left|\lambda^{[\rho[q]]}_{j}-\hat P_{j\to k[\rho[q]]}\right| \leq C\zeta_q n^{-\beta_q}(\log\log n)^{q+1}.
\end{equation}
Next, we bound the remaining sums appearing in (\ref{annoyingsum}). When $M_{L[q]} \leq k_{L[q]}$, 
\begin{align*}
     \sum^{k[\rho[q]]}_{h=M_{L[q]}} \lambda^{[\rho[q]]}_h &= \int^{\left(\frac{k[\rho[q]]}{n}\right)^\chi}_{\hat a_{L[q]}} \frac{\zeta_q}{\mu \hat a^{1/\mu}_{L[q]}} y^{\frac{1}{\mu}-1} dy =\zeta_q \left[ \left(\frac{k_{L[q]}}{n}\right)^{\frac{\chi}{\mu}}\hat a^{-\frac{1}{\mu}}_{L[q]} -1 \right].
\end{align*}
Choose $n$ large enough such that $C_q n^{-\beta_q}(\log\log n)^{\chi(q+1)}< 1$. We use $(k_{L[q]}/n)^\chi \leq \hat a_{L[q]}\ + C_q n^{-\beta_q}$ and the first inequality in (\ref{vertlab}) to calculate
\begin{align*}
   \left(\frac{k_{L[q]}}{n}\right)^{\frac{\chi}{\mu}}\hat a^{-\frac{1}{\mu}}_{L[q]} -1 &\leq \hat a^{-\frac{1}{\mu}}_{L[q]}\left[\hat a_{L[q]} + C_q n^{-\beta_q}\right]^{\frac{1}{\mu}} -1 
  \leq \{1+C_q n^{-\beta_q}(\log\log n)^{(q+1)\chi}\}^{\frac{1}{\mu}}-1\\
    &\leq C n^{-\beta_q}(\log \log n)^{(q+1)\chi},
\end{align*}
where $C:=C(x_1,\mu,q)$ is a constant. Hence, on the event $J$,
\begin{equation}\label{Lsum}
   \sum^{k[\rho[q]]}_{h=M_{L[q]}} \lambda^{[\rho[q]]}_h\leq C\zeta_q n^{-\beta_q}(\log \log n)^{(q+1)\chi},
\end{equation}
When $M_{L[q]}\geq k_{L[q]}+2$, we need an upper bound on $M_{L[q]}$ to bound $\sum^{M_{L[q]}}_{h=k_{L[q]}+1} \hat P_{h\to k[\rho[q]]}$. Since $((M_{L[q]}-1)/n)^\chi\leq a_{L[q]}$ by definition, we have $M_{L[q]}\leq n\hat a^{1/\chi}_{L[q]} + 1$. Let $a_n:=\floor{na^{1/\chi}_{L[q]}-C_qn^{1-\beta_q/\chi}}$ and $b_n := na^{1/\chi}_{L[q]} + 1$, so that $k_{L[q]}\geq a_n$ and $M_{L[q]}\leq b_n$ on the event $J$. In addition, pick $n$ large enough so that $e_n:=(\log\log n)^{q+1}(C_q n^{-\beta_q/\chi}+n^{-1})<1$. On the event $J$, we use $\hat a^{-1/\chi}_{L[q]}\leq (\log\log n)^{q+1}$ and the first inequality in (\ref{vertlab}) to obtain
\begin{align*}
    &\sum^{M_{L[q]}}_{h=k_{L[q]}+1} \hat P_{h\to k[\rho[q]]}
    \quad \leq \frac{ \zeta_q}{(\mu+1)(k_{L[q]})^{1-\chi}}\int^{b_n}_{a_n} y^{-\chi} dy\\
    & \quad\leq \frac{ \zeta_q n^{1-\chi}\hat a^{1/\mu}_{L[q]}}{(k_{L[q]})^{1-\chi}}\left[\left(1+n^{-1}\hat a^{-1/\chi}_{L[q]}\right)^{1-\chi}-\left(1-C_q\hat a^{-1/\chi}_{L[q]}n^{-\beta_q/\chi}-n^{-1}\hat a^{-1/\chi}_{L[q]}\right)^{1-\chi}\right]\\
    & \quad\leq C\zeta_q (\log\log n)^{q+1}\left[\{1+n^{-1}(\log\log n))^{q+1}\}^{1-\chi}-\{1-e_n\}^{1-\chi}\right]\\
    &\quad \leq C' \zeta_q n^{-\beta_q/\chi}(\log\log n)^{2(q+1)}, \numberthis \label{Psum}
\end{align*}
where $C:=C(x_1,\mu)$ and $C':=C'(x_1,\mu)$. The proof then follows from applying (\ref{LPsum}), (\ref{Lsum}) and (\ref{Psum}) to (\ref{annoyingsum}).
\end{proof}

Next, we apply Lemma \ref{rfirstcoup}, \ref{rsecondcoup} and \ref{rthirdcoup} to prove Lemma \ref{extnhood}.

\begin{proof}[Proof of Lemma \ref{extnhood}]
We show that on the event $J$, there is a coupling of $\mathbf{Y}^{(k[\rho[q]],n)}$ and $\mathbf{V}^{(k[\rho[q]],n)}$, and positive constants $C:=C(x_1,\mu,q)$, $c:=c(x_1,\mu)$, $C':=C'(x_1,\mu,q)$ and $C'':=C''(x_1,\mu,q)$ such that
\begin{align*}
    &\Prob_\mathbf{x}\left(\mathbf{Y}^{(k[\rho[q]],n)}\not= \mathbf{V}^{(k[\rho[q]],n)}\big|\Xi_{\mathbf{x}}\right)\\
    &\hspace{1cm}\leq \frac{C\zeta_q(\log \log n)^{(1-\chi)(q+1)}}{n^\gamma}
    +\frac{c\zeta_q(\log \log n)^{(q+1)(1-\chi)}}{(\mu+1)n^{1-\chi}}(\log n)^{q/r}\\
    &\hspace{4cm} + \frac{C'\zeta^{2}_{q}(\log \log n)^{(2-\chi)(q+1)}}{(\mu+1)n}
    + \frac{C''\zeta_q(\log \log n)^{q+1}}{n^{\beta_q}}. \numberthis \label{Lvectorbd}
\end{align*} 
The lemma follows from taking expectation with respect to $\Xi_{\mathbf{x}}$, since on the event $J$, $\E_{\mathbf{x}}\{ \zeta_q|k_{L[q]}\}\leq \kappa+1$ and $\E_{\mathbf{x}} \{\zeta_q^2|k_{L[q]}]\} \leq  (\kappa+2)(\kappa+1)$. The coupling can be constructed exactly as for $\mathbf{Y}^{(k[1],n)}$ and $\mathbf{V}^{(k[1],n)}$, this time using Lemma \ref{rfirstcoup}, \ref{rsecondcoup} and \ref{rthirdcoup}. The bound in (\ref{Lvectorbd}) follows readily from a computation similar to (\ref{vectorineq}).
\end{proof}


\section{Application to degree statistics: proofs }\label{sapp}
In this section we prove the results appearing in Section \ref{subsecapp}. We only prove Theorem \ref{uniformvert} and Proposition \ref{pmfuni}, because Theorem \ref{ancestors} is an immediate consequence of Theorem \ref{bigthm}, and the proof of Proposition \ref{ancestorpowerlaw} is similar to that of Proposition \ref{pmfuni}.

For Theorem \ref{uniformvert}, the key to improving the rate of convergence is to choose a suitable threshold $n^\psi$, and then for each fixed vertex $\ceil{n^\psi}\leq j\leq n$ in $G_n\sim \mathrm{Seq}(\mathbf{x})_n$ and $\mathbf{x}\in A_{\alpha,n}$, construct a coupling as in Lemma \ref{firstcoupling}, \ref{secondcoupling} and \ref{thirdcoupling}. The convergence rate then follows from randomising over the choices of uniform vertex $k_0$, and taking expectation with respect to $\mathbf{X}$. More intuitively, this is because the uniformly chosen vertex only needs to be large enough so that it has a small degree with high probability, in contrast to having a small $r$-neighbourhood for all $r<\infty$, as required in the proof of Theorem \ref{bigthm}.

\begin{proof}[Proof of Theorem \ref{uniformvert}]
Given $p>4$, choose $\alpha$ such that $1/2+1/p<\alpha<3/4$, and let $A_{\alpha,n}$ be as in (\ref{eventX}). In preparation for the coupling, define $\psi:=\max\{1-(1-\alpha)/8, \chi\}$ and $\gamma':=\min\{\psi, \chi(3-4\alpha)/4\}$. Let $\Xi_{\mathbf{x}}:=((\mathcal{Z}_j[1], \mathcal{Z}_j[1]), 2\leq j\leq n)$, $(S_{k,n}[1], 1\leq k\leq n)$ and the events $F_{1,i}$, $i=1,2,3$ be as in (\ref{newgamma}), (\ref{newtildegamma}), (\ref{news}) and (\ref{3f}). Furthermore, let $U_0\sim\mathrm{U}[0,1]$, $\hat a_0=U^\chi_0$ and $k_0=\ceil{nU_0}$. Define $\mathbf{Y}^{(k_0,n)}$, $\mathbf{V}^{(k_0,n)}$, $\mathbf{\hat Y}^{(k_0,n)}$ and $\mathbf{\hat V}^{(k_0,n)}$ as in Definition \ref{berproc01}, \ref{dispoi}, \ref{interber} and \ref{interpoi}.

Fixing vertex $\ceil{n^\psi}\leq j\leq n$, assume that $\mathbf{x}\in A_{\alpha,n}$ and $U_0\in ((j-1)/n, j/n]$, where the latter implies $k_0=j$. We start by coupling $\mathbf{Y}^{(j,n)}$ and $\mathbf{\hat Y}^{(j,n)}$. On the event $\bigcap^3_{i=1}F_{1,i}$, arguing the same way as for (\ref{impbd}), this time using $j>n^\psi$ and $\gamma'<\psi$, there is a positive constant $c:=c(x_1,\mu)$ such that
\begin{equation*}
   (1-cn^{-\gamma'})\hat P_{h\to j}  \leq  P_{h\to j}\leq (1+cn^{-\gamma'})\hat P_{h\to j},
\end{equation*}
where $P_{h\to j}$ and $\hat P_{h\to j}$ are as in (\ref{tmean1}) and (\ref{mean2}), with $k_0$ replaced by $j$. Hence, using the same coupling as that of Lemma \ref{firstcoupling}, on the event $\bigcap^3_{i=1}F_{1,i}$,
\begin{align} \label{improvedcoup1}
   \Prob_{\mathbf{x}}\left(\mathbf{Y}^{(j,n)}\not = \mathbf{\hat Y}^{(j,n)}|\Xi_\mathbf{x}\right)=\frac{cn^{-\gamma'}\mathcal{Z}_j[1]}{(\mu+1)j^{1-\chi}} \sum^n_{h=j+1} h^{-\chi}
   \leq \frac{cn^{1-\chi-\gamma'}\mathcal{Z}_j[1]}{j^{1-\chi}}.
\end{align}

The coupling of $\mathbf{\hat Y}^{(j,n)}$ and $\mathbf{\hat V}^{(j,n)}$ is entirely similar to that in Lemma \ref{secondcoupling}, yielding 
\begin{align}\label{improvedcoup2}
    \Prob_{\mathbf{x}}\left(\mathbf{\hat Y}^{(j,n)}\not = \mathbf{\hat V}^{(j,n)}|\Xi_\mathbf{x}\right) \leq \frac{\mathcal{Z}^2_j[1]n^{1-\chi}}{(\mu+1)j^{2-\chi}}.
\end{align}
Next, we couple $\mathbf{\hat V}^{(j,n)}$ and $\mathbf{ V}^{(j,n)}$ as in the proof of Lemma \ref{thirdcoupling}. To bound the sum of the absolute differences of the Poisson means, we apply (\ref{agediff}) and $j\geq n^\psi$ to obtain $\hat a^{-1/\mu}_0-(j/n)^{-\chi/\mu}\leq c'' n^{-\chi} j^{\chi-1}$, where $c'':=c(\mu)$ is a constant. Using the bound, we can continue much the same way for obtaining (\ref{meandiff}). Adding the absolute differences over $j+1,...,n$, we deduce that there are constants  $c':=c'(\mu)$ and $c'':=c''(\mu)$ such that
\begin{align}\label{improvedcoup3}
    \Prob_{\mathbf{x}}\left(\mathbf{V}^{(j,n)}\not =\mathbf{\hat V}^{(j,n)} |\Xi_\mathbf{x}\right) \leq \mathcal{Z}_j[1] \{c'n^{-\chi^2} j^{\chi-1} + c''j^{-2}\}.
\end{align}
Applying the arguments for proving Lemma \ref{poissonproc}, this time with the bounds (\ref{improvedcoup1}), (\ref{improvedcoup2}) and (\ref{improvedcoup3}), it follows that
\begin{align}\label{totalerror}
    \Prob_\mathbf{x}\bigg(\mathbf{Y}^{(j,n)}\not =\mathbf{V}^{(j,n)}, \bigcap^3_{i=1} F_{1,i} \bigg)\leq \frac{cn^{1-\chi-\gamma'}x_j}{j^{1-\chi}} +  \frac{x_j(x_j+1)n^{1-\chi}}{(\mu+1)j^{2-\chi}} + \frac{c'x_jn^{-\chi^2}}{j^{1-\chi}}+\frac{c''x_j}{j^2}.
\end{align}
To conclude the proof, we note that
 \begin{align}\label{kb}
      \Prob\left(\mathbf{Y}^{(k_0,n)}\not =\mathbf{V}^{(k_0,n)} \right)\leq \E\bigg[\mathbbm{1}[\mathbf{x}\in A_{\alpha, n}] \Prob_\mathbf{x}\left(\mathbf{Y}^{(k_0,n)}\not =\mathbf{V}^{(k_0,n)} \right)\bigg] + \Prob(\mathbf{x}\in A^c_{\alpha, n}),
 \end{align}
where for the first term, we write
\begin{align*}
    &\Prob_\mathbf{x}\left(\mathbf{Y}^{(k_0,n)}\not =\mathbf{V}^{(k_0,n)} \right)
    =\frac{1}{n}\sum^n_{j=1}\Prob_\mathbf{x}\left(\mathbf{Y}^{(j,n)}\not =\mathbf{V}^{(j,n)} \right)\\
    &\qquad\leq \frac{\ceil{n^\psi}-1}{n} + \frac{1}{n}\sum^n_{j=\ceil{n^\psi}}\Prob_\mathbf{x}\left(\mathbf{Y}^{(j,n)}\not =\mathbf{V}^{(j,n)} \right)\\
    &\qquad\leq n^{\psi-1} +\frac{1}{n} \sum^n_{j=\ceil{n^\psi}}\Prob_\mathbf{x}\bigg(\mathbf{Y}^{(j,n)}\not =\mathbf{V}^{(j,n)}, \bigcap^3_{i=1} F_{1,i} \bigg) + \Prob_\mathbf{x}\bigg(\bigg(\bigcap^3_{k=1} F_{1,k}\bigg)^c\bigg)\\
    &\qquad\leq n^{\psi-1} +\frac{1}{n}\sum^n_{j=\ceil{n^\psi}}\Prob_\mathbf{x}\bigg(\mathbf{Y}^{(j,n)}\not =\mathbf{V}^{(j,n)}, \bigcap^3_{i=1} F_{1,i} \bigg)+ \sum^3_{k=1}\Prob_{\mathbf{x}}( F^c_{1,k} ).
\end{align*}
Since $\mathbf{x}\in A_{\alpha, n}$, we bound $\Prob_{\mathbf{x}}(F^c_{1,k})$ using Lemma \ref{Sasymptotic}, and (\ref{bgc}) and (\ref{woboundedass}) of Lemma \ref{betagamma}. Applying the last display and (\ref{totalerror}) to (\ref{kb}), bounding $\Prob(\mathbf{x}\in A^c_{\alpha, n})$ by Lemma \ref{goodsum}, and then taking expectation with respect to $\mathbf{X}$, we obtain
\begin{align*}
    \Prob\left(\mathbf{Y}^{(k_0,n)}\not =\mathbf{V}^{(k_0,n)} \right)
    \leq C'n^{-b} + \sum^n_{j=\ceil{n^\psi}}\left\{ \frac{cn^{-\chi-\gamma'}\mu }{j^{1-\chi}} +  \frac{(\E X^2_2+\mu)n^{-\chi}}{(\mu+1)j^{2-\chi}} + \frac{c'\mu n^{-1-\chi^2}}{j^{1-\chi}}+\frac{c''n^{-1}}{j^2}\right\},
\end{align*}
where $b=\min\{\chi[p(\alpha-1/2)-1], \chi(1-\alpha)/2, 4\gamma'-\chi(4\alpha-3),\chi,\psi-1\}$ and $C':=C'(x_1,\mu,p)$. By an integral comparison, we get that the sum above is bounded by $C'n^{-\min\{\gamma',1-\chi+\chi^2\}}$, where $C':=C'(x_1,\mu)$ is a constant. Choosing $d=\min\{b,\gamma',1-\chi+\chi^2\}$ and $C=2\max\{C',C''\}$ concludes the proof.
\end{proof}


\begin{proof}[Proof of Proposition \ref{pmfuni}]
The probability mass function given in (\ref{degdistn}) is an exercise of integration. The steps are similar to the case where $X_1=0$ and $X_i=1$ almost surely for $i\geq 2$, and the details can be found in \textcite[Lemma 5.2]{Berger12asymptoticbehavior}. To prove (\ref{powerlaw}), we use the argument of \textcite{lodewijks2020phase} for proving their Theorem 2.6(i). \cite[Theorem 1, equation (5)]{jameson2013} implies that $k^{\mu+2}\Gamma(x+k-1)/\Gamma(x+\mu+k+1)\leq 1$ for all $x>0$ and $k\geq 1$. Hence,
\begin{align*}
    f_{\mu,k}(x)=k^{\mu+2}\frac{\Gamma(x+k-1)}{\Gamma(x+\mu+k+1)}\frac{\Gamma(x+\mu+1)}{\Gamma(x)}
\end{align*}
is dominated by $\Gamma(x+\mu+1)/\Gamma(x)$. Thus, if $\E X^{\mu+1}_2<\infty$, the dominated convergence theorem (\textcite[Theorem 1.5.8, p.\ 24]{durrett}) implies that 
\begin{equation*}
     \lim_{k\to\infty} k^{\mu+2} p_\pi(k)= (\mu+1)\int^\infty_0 \lim_{k\to\infty} f_{\mu,k}(x) d\pi(x),
\end{equation*}
and so the claim follows from $\lim_{k\to\infty} k^{\mu+2}\Gamma(x+k-1)/\Gamma(x+k+\mu+1)=1$.
\end{proof}

\section{The embellished preferential attachment graphs}\label{embmodel}
In this section we study the embellished $(\mathbf{x},n)$-sequential model, that is, the $(\mathbf{x},n)$-sequential model conditional on a finite collection of edges. Importantly, this embellished $(\mathbf{x},n)$-sequential model has an urn representation that enables us to apply the Bernoulli-Poisson coupling arguments to all vertices in the local neighbourhood of the uniform vertex in the random graph with law Seq$(\mathbf{x})_n$.

\subsection{The attachment rules}
Given $\mathbf{x}$ and $n$, let $G'_n\sim\mathrm{Seq}(\mathbf{x})_n$. To specify the edges that we condition on, let $Q_{j\to k}$ be a zero-one variable that takes value one if and only if vertex $j$ sends an outgoing edge to $k$ in $G'_n$, and $W_{k,n}$ be the in-degree of vertex $k$ in $G'_n$, so that $W_{k,n} := \sum^n_{j=k+1} Q_{j\to k}$, and $W_{k,n}=0$ if $k\geq n$. In view of the preferential attachment rules, if $Q_{k\to i}=1$ for some $1\leq i<k$, then $Q_{k\to l}=0$ for any $l\not =i$. Furthermore, let $\mathcal{V}$ be a strict subset of $V(G'_n)$, and given $\mathcal{V}$, let $\mathcal{E}$ be a set of edges where at least one end of an edge in this set is a vertex in $\mathcal{V}$. Furthermore, for each $u\in \mathcal{V}$ where $u>1$, $Q_{u\to k}=1$ for some $k<u$ and $Q_{u\to j}=0$ for all $j\not=k$, so that vertex $k$ is the recipient of the only outgoing edge from vertex $u$. We investigate how conditioning on the event 
\begin{equation*}
   \mathcal{I}:=\mathcal{I}(\mathcal{V},\mathcal{E}):= \bigcap_{\substack{\{h,\ell\}\not \in \mathcal{E};\\ h \in \mathcal{V},h<\ell\leq n } } \{Q_{\ell \to h}=0\} \cap \bigcap_{\substack{\{i,k\}\in \mathcal{E};\\ 1\leq i<k \leq n}}\{ Q_{k\to i}=1\}
\end{equation*}
changes the attachment rules of $G'_n$. Note that $\Prob_{\mathbf{x}}(\mathcal{I})>0$ because $\mathcal{E}$ does not allow vertex $u\in \mathcal{V}$ to send two outgoing edges, and the outgoing edge from $u\in \mathcal{V}$ is always directed towards some vertex $k<u$. Moreover, observe that the attachment steps involving vertex $u\in \mathcal{V}$ are deterministic on the event $\mathcal{I}$. 

To prepare for the subsequent arguments, denote by $\mathcal{V}^*:=\{ v\not \in \mathcal{V}: \{u,v\}\in \mathcal{E}\}$ the vertices that are not in $\mathcal{V}$, but are the endpoints of at least one edge in $\mathcal{E}$. To exclude trivial cases, from now on we assume that $\mathcal{V}\cup\mathcal{V}^*$ is a strict subset of $V(G_n)$. Moreover, let $v_s:=v_s(\mathcal{V})$ (resp.\ $v^*_s:=v^*_s(\mathcal{V}^*)$) be the vertex in $\mathcal{V}$ (resp.\ $\mathcal{V}^*$) that has the smallest vertex label. On the event $\mathcal{I}$, $v^*_s$ must be a recipient of at least one incoming edge from the vertices in $\mathcal{V}$, and it does not send an outgoing edge to a vertex in $\mathcal{V}$. If $v^*_s$ sends an outgoing edge to vertex $v \in \mathcal{V}$, then by the definition of $\mathcal{V}^*$, there is some vertex $u<v^*_s$ in $\mathcal{V}^*$ that receives the edge emanating from $v$. Because $\Prob_{\mathbf{x}}(\mathcal{I})>0$, $u$ and $v^*_s$ cannot both be the recipients of the incoming edges from $v$.

We impose the following assumption on the collection of edges, which greatly simplifies the upcoming computations. We refer to it as ($\triangle$). 
\begin{center}
\fbox{\begin{varwidth}{\dimexpr\textwidth-2\fboxsep-2\fboxrule\relax}
The outgoing edge sent by vertex $v\in \mathcal{V}\setminus \{v_s\}$ is received by another vertex $u\in \mathcal{V}$.
\end{varwidth}}
\end{center}
A moment's thought reveals that under this assumption, $v_s$ is the only vertex in $\mathcal{V}$ that sends an outgoing edge to $v^*_s$, and any vertex $w\in \mathcal{V}^*\setminus \{v^*_s\}$ must be sending an outgoing edge to a vertex in $\mathcal{V}$. By Lemma \ref{onlyedge}, $\mathcal{V}$ and $\mathcal{V^*}$ correspond to $\mathcal{P}_{t-1}$ and $\mathcal{A}_{t-1}$ for any $t\geq 2$ under the assumption, where $\mathcal{P}_{t-1}$ and $\mathcal{A}_{t-1}$ are respectively the sets of probed and active vertices in the breadth-first search of the $(\mathbf{x},n)$-sequential model. It follows that $v_s$ and $v^*_s$ correspond to $k_s[t]$ and $k^*_s[t]$ respectively. To visualise these observations, it is instructive to refer to Figure \ref{graphlabelling}, with the labels there replaced with the vertex labels of $G'_n$, and the edges directed. Additionally, the event $\mathcal{I}$ can be written as 
\begin{align*}
   \mathcal{I}&=\{Q_{v_s\to v^*_s}=1\} \cap \bigcap_{\substack{\{\ell,h\}\not \in \mathcal{E};\\  h\in \mathcal{V},h<\ell\leq n} } \{Q_{\ell \to h}=0\} \cap \bigcap_{\substack{\{i,k\}\in \mathcal{E}\setminus\{v^*_s,v_s\};\\ v_s\leq i<k\leq n}}\{ Q_{k\to i}=1\}\\
   &=: \{Q_{v_s\to v^*_s}=1\} \cap \mathcal{J}.\numberthis\label{eventI}
\end{align*}

We now show that conditioning on $\mathcal{I}$, $G'_n$ has a modified preferential attachment rule. Firstly, note that conditioning on $\mathcal{I}$ does not change the attachment rule for the first $v^*_s-1$ steps for constructing the $G'_n$. Next, we study how conditioning on $\mathcal{I}$ changes the rules for constructing $G'_i$ from $G'_{i-1}$ for $v^*_s<i<v_s$. Under the assumption $(\triangle)$, the edges in $\mathcal{E}\setminus \{v^*_s,v_s\}$ do not affect these attachment steps, because these edges are born after step $v_s$. Hence it is enough to consider how conditioning on the edge $\{v^*_s,v_s\}$ changes the rules of these attachment steps. The upcoming lemma is a variation of \textcite[Lemma 3.5]{ross2013}. We state the lemma in slightly greater generality, but is clearly applicable to our case by taking $\ell=v^*_s$ and $s=v_s$ in what follows. Choose any two positive integers $\ell$ and $s$, where $\ell<s<n$. The lemma below implies that given $Q_{s \to \ell}=1$ and $G'_{m-1}$, where $\ell < m < s$, we attach vertex $m$ to $\ell$ according to the same preferential attachment rule, but also include the edge $\{s,\ell\}$ in the vertex weight of $\ell$. In other words, we can think of the initial attractiveness of vertex $\ell$ as $x_{\ell}+1$ instead of $x_{\ell}$. 

\begin{lemma}\label{sizebias}
Let $Q_{m\to j}$, $G'_{m-1}$ and $W_{\ell,m}$ be as above and $T_k:=\sum^k_{i=1} x_i$. Then for $j,\ell<m<s$, 
\begin{equation}\label{sb}
    \Prob_{\mathbf{x}}(Q_{m\to j}=1|G'_{m-1},Q_{s\to \ell}=1)=\frac{W_{j,m-1}+x_{j}+\mathbbm{1}[j=\ell]}{T_{m-1}+m-1}.
\end{equation}
\end{lemma}

\begin{proof}
We adapt the proof of \cite[Lemma 3.5]{ross2013}. By the definition of conditional probability, we have
\begin{align*}
    &\Prob_{\mathbf{x}}(Q_{m\to j}=1|G'_{m-1},Q_{s\to \ell}=1)\\
    &\qquad =\frac{\Prob_{\mathbf{x}}(Q_{m\to j}=1|G'_{m-1})\Prob_{\mathbf{x}}(Q_{s\to \ell}=1|G'_{m-1}, Q_{m\to j}=1)}{\Prob_{\mathbf{x}}(Q_{s\to \ell}=1|G'_{m-1})}\numberthis \label{condprob}.
\end{align*}
We compute the probabilities above as follows.
\begin{equation*}
    \Prob_{\mathbf{x}}(Q_{m\to j}=1|G'_{m-1})=\frac{W_{j, m-1}+x_{j}}{T_{m-1}+m-2}.
\end{equation*}
This in turn implies that 
\begin{equation*}
    \Prob_{\mathbf{x}}(Q_{s\to \ell}=1|G'_{m-1})=\frac{\E_{\mathbf{x}}(W_{\ell,s-1}+x_{\ell}|G'_{m-1})}{T_{s-1}+s-2},
\end{equation*}
and
\begin{equation*}
    \Prob_{\mathbf{x}}(Q_{s\to \ell}=1|G'_{m-1}, Q_{m\to j=1} )=\frac{\E_{\mathbf{x}}(W_{\ell,s-1}+x_{\ell}|G'_{m-1}, Q_{m\to j}=1)}{T_{s-1}+s-2}.
\end{equation*}
Moreover, Lemma \ref{morimom} at the end of this subsection implies that for $\ell<m<s$,
\begin{equation*}
    \E_{\mathbf{x}}(W_{\ell,s-1}+x_{\ell}|G'_{m-1})=(W_{\ell,m-1}+x_{\ell})\prod^{s-2}_{h=m-1}\frac{T_h+h}{T_h+h-1}. 
\end{equation*}
Note that
\begin{equation*}
    \E_{\mathbf{x}}(W_{\ell,m}+x_{\ell}|G'_{m-1}, Q_{m\to j}=1)=W_{\ell,m-1}+x_{\ell}+\mathbbm{1}[j=\ell],
\end{equation*}
and so by another application of Lemma \ref{morimom},
\begin{equation*}
    \E_{\mathbf{x}}(W_{\ell,s-1}+x_{\ell}|G'_{m-1}, Q_{m\to j}=1)=(W_{\ell,m-1}+x_{\ell}+\mathbbm{1}[j=\ell])\prod^{s-2}_{h=m}\frac{T_h+h}{T_h+h-1}.
\end{equation*}
Applying these results to (\ref{condprob}) and simplifying yields
\begin{align}\label{sb1}
    \Prob_{\mathbf{x}}(Q_{m\to j}=1|G'_{m-1},Q_{s\to \ell}=1)=\frac{(W_{j,m-1}+x_j)(W_{\ell,m-1}+\mathbbm{1}[j=\ell]+x_\ell)}{(T_{m-1}+m-1)(W_{\ell,m-1}+x_\ell)},
\end{align}
and (\ref{sb1}) is equal to (\ref{sb}) by considering the cases $\ell=j$ and $\ell\not =j$ separately.
\end{proof}

Given that $m > v_s$ and $m\not \in \mathcal{V}^*\cup \mathcal{V}$, we proceed to prove that on the event $\mathcal{J}$, we attach vertex $m$ to $j\in \{1,...,m-1\}$, $j\not \in \mathcal{V}$ with probability proportional to the weight of vertex $j$. To precisely state the lemma, denote the sets of vertices and edges in $\mathcal{V}$ and $\mathcal{E}$ that are born before vertex $m$ as
\begin{align*}
    \mathcal{V}_{m}=\{u\in \mathcal{V}:u<m\}\quad \text{and}\quad
    \mathcal{E}_{m}=\{\{h,k\}\in \mathcal{E}: \max\{h,k\}<m\}.
\end{align*}
When compared to the breadth-first search of the $(\mathbf{x},n)$-sequential model, $\mathcal{V}_{m}$ and $ \mathcal{E}_{m}$ respectively correspond to $\mathcal{P}_{t,m}$ and $\mathcal{E}_{t,m}$ in (\ref{eptm}) for any $t\geq 2$. 

\begin{lemma}\label{genexcond}
Retaining the notations in Lemma \ref{sizebias}, let $\mathcal{V}$ and $\mathcal{E}$ be such that $(\triangle)$ holds, and $\mathcal{J}$ be as in (\ref{eventI}). We have for $j<m$, where $j\not \in \mathcal{V}$, $m \in \{v_s+1,...,n\}\setminus (\mathcal{V}^*\cup \mathcal{V})$,
\begin{equation}\label{geneq}
    \Prob(Q_{m\to j}=1|G'_{m-1},\mathcal{J})=\frac{W_{j,m-1}+x_j}{T_{m-1}+m-1-\sum_{k\in \mathcal{V}_{m}}x_k-|\mathcal{E}_{m}|}.
\end{equation}
\end{lemma}

Before proving the lemma, we emphasize that we exclude the attachment steps of the vertices in $\mathcal{V}^*\setminus\{v^*_s\}$ in the lemma, because under the assumption $(\triangle)$, vertex $i \in \mathcal{V}^*\setminus\{v^*_s\}$ necessarily sends an outgoing edge to a vertex in $\mathcal{V}$. Furthermore, the edge count in the normalising constant in (\ref{geneq}) is $m-1-|\mathcal{E}_m|$ instead of $m-2-|\mathcal{E}_m|$, as we need to include $\{v_s,v^*_s\}\in \mathcal{E}_m$. Let $m \in \{v_s+1,...,n\}\setminus (\mathcal{V}^*\cup \mathcal{V})$. When attaching vertex $m$ to vertex $j \in \mathcal{V}^*\setminus \{v^*_s\}$, $(\triangle)$ ensures that vertex $j$ does not receive any incoming edges from the vertices of $\mathcal{V}$, as otherwise these edges have a size-biasing effect on the initial attractiveness of vertex $j$ and in that case equation (\ref{geneq}) no longer holds.

\begin{proof}[Proof of Lemma \ref{genexcond}]
We use the definition of the conditional probability again, this time to rewrite the the left-hand side of (\ref{geneq}) in terms of probabilities that condition on the events occurring before step $m$. Firstly, let $\mathcal{J}:=\mathcal{J}_{< m}\cap \mathcal{J}_{\geq m}$, where for $m \in \{v_s+1,...,n\}\setminus (\mathcal{V}^*\cup \mathcal{V})$, 
\begin{align*}
    \mathcal{J}_{<m} := \bigcap_{\substack{ v_s\leq k<\ell<m;\\
    \{\ell,k\}\in \mathcal{E}_{m}\setminus \{v^*_s, v_s\} }}\{Q_{\ell\to k}=1\}\cap \bigcap_{\substack{v_s\leq i<h< m;\\ i\in \mathcal{V},\{i,h\}\not \in \mathcal{E}_m} }\{Q_{h\to i}=0\}
\end{align*}
and 
\begin{align*}
    \mathcal{J}_{>m}:=\bigcap_{\substack{ (m\vee k)< \ell\leq n;\\\{\ell,k\}\in \mathcal{E}\setminus \mathcal{E}_m }}\{Q_{\ell\to k}=1\}\cap \bigcap_{\substack{ m\leq h\leq n,i<h;\\ i \in \mathcal{V},\{i,h\}\not \in \mathcal{E}\setminus \mathcal{E}_m }} \{Q_{h\to i}=0\}. 
\end{align*}
Hence we have the following expression:
\begin{align*}
    &\Prob_{\mathbf{x}}(Q_{m\to j}=1|\mathcal{J},  G'_{m-1})\\
    &\qquad=\frac{\Prob_{\mathbf{x}}(Q_{m\to j}=1|\mathcal{J}_{< m}, G'_{m-1})\Prob_{\mathbf{x}}(\mathcal{J}_{\geq m}|Q_{m\to j}=1, \mathcal{J}_{< m}, G'_{m-1})}{\Prob_{\mathbf{x}}(\mathcal{J}_{\geq m}| \mathcal{J}_{< m}, G'_{m-1})}.
\end{align*}
It is straightforward to see that for $j\not \in \mathcal{V}$,
\begin{equation*}
    \Prob_{\mathbf{x}}(Q_{m\to j}=1|\mathcal{J}_{< m}, G'_{m-1})=\frac{W_{j,m-1}+x_j}{T_{m-1}+m-2}.
\end{equation*}
Using the attachment rules of the $(\mathbf{x},n)$-sequential model, and simplifying, we can compute
\begin{align*}
    \frac{\Prob_{\mathbf{x}}(\mathcal{J}_{\geq m}| \mathcal{J}_{< m}, G'_{m-1})}{\Prob_{\mathbf{x}}(\mathcal{J}_{\geq m}|Q_{m\to j}=1, \mathcal{J}_{< m}, G'_{m-1})}=\frac{T_{m-1}+m-1-\sum_{k\in \mathcal{V}_m}x_k-|\mathcal{E}_m|}{T_{m-1}+m-2}.
\end{align*}
Combining the equations above completes the proof of the lemma.
\end{proof}

The following lemma is applied in the proof of Lemma \ref{sizebias}, which is a slight modification of Lemma 4.1 in \cite{pekoz2016} and Theorem 2.1 of \cite{mori}.

\begin{lemma}\label{morimom}
Retaining the notations in Lemma \ref{sizebias}, let $k$, $\ell$ and $m$ be positive integers such that $k\leq \ell\leq  m$, then 
\begin{align*}
    \E_{\mathbf{x}}(W_{k,m}+x_k|G'_{\ell})=(W_{k,\ell} + x_{k}) \prod^{m-1}_{j=\ell}\frac{T_j+j}{T_j+j-1}.
\end{align*}
\end{lemma}

\begin{proof}
Note that at attachment step $i>k$, $W_{k,i}$ either increases by exactly one or stays the same, and $W_{k,i+1}=W_{k,i}+1$ with probability proportional to $W_{k,i}+x_k$. Hence,  
\begin{align*}
     &\E_{\mathbf{x}}(W_{k,m}+x_k|G'_{m-1})\\
     &=\frac{W_{k,m-1}+x_k}{T_{m-1}+m-2}(W_{k,m-1}+1+x_k)+\left(1-\frac{W_{k,m-1}+x_k}{T_{m-1}+m-2}\right)(W_{k,m-1}+x_k)\\
     &=\frac{T_{m-1}+m-1}{T_{m-1}+m-2}(W_{k,m-1}+x_k). 
\end{align*}
The lemma follows from iterating the result above. 
\end{proof}

\subsection{Construction of the embellished model}\label{embx}
Assuming that $\mathcal{V}$ and $\mathcal{E}$ are such that ($\triangle$) holds, $\mathcal{V}\cup\mathcal{V}^*\subset V(G_n)$, and that $G_n{(\mathcal{I})}$ has the distribution of $\mathrm{Seq}(\mathbf{x})_n$ conditional on $\mathcal{I}$. Lemma \ref{sizebias} and \ref{genexcond} imply that we can construct the random variables $W_{k,n}{(\mathcal{I})}:=(W_{k,n}|\mathcal{I})$ as follows. Initially, we generate $G_{v^*_s-1}{(\mathcal{I})}$ using the usual attachment rule. At step $v^*_s$, add the vertices $v^*_s$ and $v_s$ to $G_{v^*_s-1}{(\mathcal{I})}$, such that $v^*_s$ receives an incoming edge from $v_s$. Then $v^*_s$ sends an outgoing edge to vertex $j\in \{1,...,v^*_s-1\}$, which is chosen with probability proportional to $W_{j,v^*_s-1}{(\mathcal{I})}+x_j$. After the attachment step, assign vertex $v^*_s$ with initial attractiveness $x_{v^*_s}$, and set $W_{v^*_s, v^*_s}{(\mathcal{I})}=1$. At step $v^*_s<m<v_s$, vertex $m$ sends an outgoing edge to vertex $j\in \{1,...,m-1\}$, with probability \begin{equation*}
    \frac{W_{j,m-1}{(\mathcal{I})}+x_j}{T_{m-1}+m-1},
\end{equation*}
and we equip vertex $m$ with initial attractiveness $x_m$. The step $v_s$ is completed by endowing vertex $v_s$ with initial attractiveness $x_{v_s}$. At step $m>v_s$ where $m\not\in \mathcal{V}^*\cup \mathcal{V}$, vertex $m$ sends an outgoing edge to vertex $j\in \{1,...,m-1\}\setminus \mathcal{V}_m$ with probability
\begin{equation*}
    \frac{W_{j,m-1}{(\mathcal{I})}+x_j}{T_{m-1}+m-2-\sum_{k\in \mathcal{V}_m}x_k-(|\mathcal{E}_{m}|-1)}.
\end{equation*}
At steps $m=\max\{h,i\}$, where $\{h,i\}\in \mathcal{E}\setminus \{v^*_s, v_s\}$, vertex $m$ sends an outgoing edge to vertex $\min\{h,i\}$, and the initial attractiveness of vertex $m$ is given by $x_m$.

\subsection{An urn representation of the embellished model}\label{altdefembmodel}
Using an urn argument analogous to that of Theorem \ref{linebreaking}, below we give an alternative definition of $G_n{(\mathcal{I})}$. To simplify notations, we drop $\mathbf{x}$ in the definitions of the variables below, which depend on the sequence of initial attractiveness. 

\begin{defn}[$(\mathbf{x},\mathcal{I},n)$-P\'olya urn tree]\label{tiltedpug}
Given the sequence $\mathbf{x}$, and the sets of vertices and edges $\mathcal{V}$ and $\mathcal{E}$ such that ($\triangle$) holds, and $\mathcal{V}\cup\mathcal{V}^*\subset V(G_n)$, let $\mathcal{I}$ be as in (\ref{eventI}), and $B_j(\mathcal{I})$ be conditionally independent random variables such that $B_1{(\mathcal{I})}:=1$, $B_j(\mathcal{I}):=0$ if $j\in \mathcal{V}$, and for $j\not \in \mathcal{V}$,
\begin{align*}
    B_j(\mathcal{I})\sim
    \begin{cases}
    &\mathrm{Beta}(x_j+\mathbbm{1}[j= v^*_s], T_{j-1}+j-1),\quad 2\leq j\leq v^*_s; \\
    &\mathrm{Beta}(x_j, T_{j-1}+j),\quad v^*_s<j<v_s;\\
    &\mathrm{Beta}(x_j, T_{j-1}+j-\sum_{k\in \mathcal{V}_j }x_k-|\mathcal{E}_{j}|),\quad v_s<j\leq n
    \end{cases}
\end{align*}
Furthermore, let $S_{0,n}{(\mathcal{I})}:=0$, $S_{n,n}{(\mathcal{I})}:=1$ and 
\begin{equation*}
    S_{k,n}(\mathcal{I}):=\prod^n_{i=k+1}(1-B_i(\mathcal{I}))\quad \text{for $1\leq j\leq n-1$.}
\end{equation*}
Starting with $n$ vertices and the edges in $\mathcal{E}$, we connect the vertices as follows. For $1\leq j\leq n$, let $I_j=[S_{j-1,n}{(\mathcal{I})},S_{j,n}{(\mathcal{I})})$ for $j\not \in \mathcal{V}$. Conditionally on $(S_{i,n}{(\mathcal{I})},1\leq i\leq n-1)$, we generate the variables $U_k\sim \mathrm{U}[0,S_{k-1,n}{(\mathcal{I})}]$ for $k\in \{2,...,n\}\setminus (\mathcal{V}\cup\mathcal{V}^*\setminus\{v^*_s\})$. If $j<k$ and $U_k\in I_j$, we attach vertex $k$ to vertex $j$, and we say that the resulting graph is an $(\mathbf{x},\mathcal{I},n)$-P\'olya urn graph. 
\end{defn}

\begin{thm}\label{condpolyaurnrep} 
Assume that $\mathcal{V}$ and $\mathcal{E}$ are such that $(\triangle)$ holds and $\mathcal{V}\cup\mathcal{V}^*\subset V(G_n)$. Let $\tilde G_n{(\mathcal{I})}$ be an $(\mathbf{x},\mathcal{I},n)$-P\'olya urn tree. Then $\tilde G_n{(\mathcal{I})}$ has the distribution of $\mathrm{Seq}(\mathbf{x})_n$ conditional on $\mathcal{I}$.
\end{thm}


The proof of Theorem \ref{condpolyaurnrep} is similar to that of Theorem \ref{linebreaking}. For the proof, we need notations analogous to those appearing in Lemma \ref{urnembed}. Let $G_{n}{(\mathcal{I})}$ and $W_{i,n}(\mathcal{I})$ be as in Section \ref{embx}. Define
\begin{equation*}
    M'_k(m)=\sum_{i\in \{1,...,k\}\setminus \mathcal{V}_k} (x_i+W_{i,m}(\mathcal{I})), \qquad U'_j(m)=M'_j(m)-M'_{j-1}(m), \quad j\not \in \mathcal{V},
\end{equation*}
and for $k \in \mathcal{V}$, $U'_k(m)=U'_k(m-1)+1$ if and only if $\{k,m\}\in \mathcal{E}$. Hence, $M'_k(m)$ is the total weight of the vertices $\{1,...,k\}\setminus\mathcal{V}_k$ after $m$ attachment steps, and $U'_j(l)$ is the weight of vertex~$j$ after $l$ steps. 

For the discussion and the proof below, assume that $m\not \in \mathcal{V}\cup \mathcal{V}^*\setminus \{v^*_s\}$ and $k\not \in \mathcal{V}$. It is clear from the construction of $G_m(\mathcal{I})$ that we can obtain an analog of Lemma~\ref{urnembed} for $U'_k(m)$. The differences here are that we have to use the adjusted total vertex weights, and vertex $m$ can only attach to a vertex that is not in $\mathcal{V}$. If $m-1\not \in \mathcal{V}$, arguing as for Lemma \ref{urnembed}, we have $U'_{m-1}(m)\sim \mathrm{Polya}(b,w;1)$, where $w=x_{m-1}+\mathbbm{1}[m=v^*_s+1]$ and
\begin{align*}
    &b=T_{m-2}+m-2\quad\text{if $m\leq v^*_s$,}\\
    &b=T_{m-2}+m-2+\mathbbm{1}[m\not =v^*_s+1]\quad\text{if $v^*_s<m< v_s$,}\\
    &b=T_{m-2}+m-2-\sum_{j\in\mathcal{V}_{m-1}}x_j-(|\mathcal{E}_{m-1}|-1)\quad \text{if $v_s<m\leq n$ and $m-1\not \in \mathcal{V}$;}
\end{align*}
with $b$ being the total weight of the vertices $\{1,...,m-1\}\setminus \mathcal{V}_{m-1}$. Let $k+1<m\leq n$. Conditional on $M'_k(m)$ and $(\{U'_j(m)=U'_j(m-1)\},j\in\{k+1,..., m-1\} \setminus \mathcal{V}_{m-1})$, the event that vertex $m$ does not attach to any vertex in $\{k+1,..., m-1\} \setminus \mathcal{V}_{m-1}$, $U'_k(m)\sim \mathrm{Polya}(b',w';q')$, where  $q'=M'_k(m)-b'-w'$ is the number of draws in a classical P\'olya urn. More specifically, $w'=x_k+\mathbbm{1}[k=v^*_s]$ and
\begin{align*}
   &b'=T_{k-1}+k-1,\quad q'=M'_k(m)-T_k-k+1\quad m\leq v^*_s;\\
   &b'=T_{k-1}+k-1+\mathbbm{1}[k>v^*_s],\quad q'=M'_k(m)-T_k-k+1-\mathbbm{1}[k\geq v^*_s],\quad v^*_s<m<v_s;
\end{align*}
and for $m>v_s$,
\begin{align*}
    &b'=T_{k-1}+k-1-\sum_{j\in\mathcal{V}_k}x_j-|\mathcal{E}_k|+\mathbbm{1}[k>v^*_s],
    \\
    &q'= M'_k(m)- T_{k}-k+1+\sum_{j\in\mathcal{V}_k}x_j+|\mathcal{E}_k|-\mathbbm{1}[k\geq v^*_s].
\end{align*}

\begin{proof}[Proof of Theorem \ref{condpolyaurnrep}]
We only consider the attachment step $m\not \in \mathcal{V}\cup \mathcal{V}^*\setminus \{v^*_s\}$, because step $h\in \mathcal{V}\cup \mathcal{V}^*\setminus \{v^*_s\}$ is deterministic under the assumption ($\triangle$). To prove the theorem, we replace $U_k(n)$ in the proof of Theorem \ref{linebreaking} by $U'_k(n)$ for $k\not \in \mathcal{V}$ and argue similarly. Note that for $j\in \mathcal{V}$ and $m\not \in \mathcal{V}\cup \mathcal{V}^*\setminus \{v^*_s\}$, $U'_j(m)=U'_j(m-1)$ is reflected by $B_j(\mathcal{I})=0$ in the construction of the $(\mathbf{x},\mathcal{I},n)$-P\'olya urn graph.
\end{proof}

With these preparations, it is now possible to prove Lemma \ref{condurnconseq} and \ref{tilted01}.

\begin{proof}[Proof of Lemma \ref{condurnconseq} and \ref{tilted01}]
Apply Theorem \ref{linebreaking} for $t=1$; while for $t\geq 2$, take $\mathcal{V}$, $\mathcal{V}^*$, $\mathcal{E}$ and $v_s$ in Theorem \ref{condpolyaurnrep} respectively as $\mathcal{P}_{t-1}$, $\mathcal{A}_{t-1}$, $\mathcal{E}_{t-1}$ and $k_s[t]$, where $\mathcal{E}_{t-1}$ is the set of edges connecting the probed and active vertices in $\mathcal{P}_{t-1}$ and $\mathcal{A}_{t-1}$, and $k_s[t]$ is the smallest vertex in $\mathcal{P}_{t-1}$. 
\end{proof}

\noindent \textbf{Acknowledgement.} The author thanks Nathan Ross for his careful reading of the manuscript and many helpful suggestions. This research is supported by an Australian Government Research Training Program scholarship, and partially by ACEMS.

\nocite{*}
\printbibliography

\end{document}